\documentclass{amsart}

\usepackage{amsmath, amssymb, amsthm, amsfonts}

\input xy
\xyoption{all}

\usepackage{hyperref}

\swapnumbers
\numberwithin{equation}{section}

\theoremstyle{plain}
\newtheorem{theorem}[subsection]{Theorem}
\newtheorem{lemma}[subsection]{Lemma}
\newtheorem{prop}[subsection]{Proposition}
\newtheorem{cor}[subsection]{Corollary}

\newtheorem*{claim}{Claim}

\theoremstyle{definition}
\newtheorem{defn}[subsection]{Definition}
\newtheorem{cons}[subsection]{Construction}
\newtheorem{remark}[subsection]{Remark}
\newtheorem{exam}[subsection]{Example}

\newtheorem{warning}[subsection]{Warning}

\def\AA{\mathbb{A}}

\def\CC{\mathbb{C}}
\def\DD{\mathbb{D}}

\def\FF{\mathbb{F}}
\def\GG{\mathbb{G}}

\def\MM{\mathbb{M}}

\def\PP{\mathbb{P}}
\def\QQ{\mathbb{Q}}

\def\SS{\mathbb{S}}

\def\ZZ{\mathbb{Z}}

\newcommand\cA{\mathcal{A}}

\newcommand\cC{\mathcal{C}}
\newcommand\cD{\mathcal{D}}
\newcommand\cE{\mathcal{E}}
\newcommand\cF{\mathcal{F}}
\newcommand\cG{\mathcal{G}}
\newcommand\cH{\mathcal{H}}

\newcommand\cK{\mathcal{K}}
\newcommand\cL{\mathcal{L}}
\newcommand\cM{\mathcal{M}}

\newcommand\cO{\mathcal{O}}
\newcommand\cP{\mathcal{P}}

\newcommand\cS{\mathcal{S}}

\newcommand\cZ{\mathcal{Z}}

\def\bH{\mathbf{H}}

\def\bR{\mathbf{R}}

\newcommand\fB{\mathfrak{B}}
\newcommand\fC{\mathfrak{C}}
\newcommand\fD{\mathfrak{D}}
\newcommand\fE{\mathfrak{E}}

\newcommand\fH{\mathfrak{H}}

\newcommand\fO{\mathfrak{O}}
\newcommand\fP{\mathfrak{P}}

\newcommand\frb{\mathfrak{b}}
\newcommand\frc{\mathfrak{c}}

\newcommand\frg{\mathfrak{g}}

\newcommand\fo{\mathfrak{o}}

\newcommand\frt{\mathfrak{t}}

\newcommand\dG{\widehat{G}}

\newcommand\dT{\widehat{T}}
\newcommand\dH{\widehat{H}}

\newcommand{\acyc}{\textup{acyc}}

\newcommand\AS{\textup{AS}}

\newcommand{\can}{\textup{can}}
\newcommand{\ch}{\textup{char}}

\newcommand{\Cone}{\textup{Cone}}

\newcommand\Fr{\textup{Fr}}

\newcommand\Gal{\textup{Gal}}

\newcommand{\Gr}{\textup{Gr}}

\newcommand\IC{\textup{IC}}
\newcommand\id{\textup{id}}
\renewcommand{\Im}{\textup{Im}}

\newcommand\inv{\textup{inv}}

\newcommand\Isom{\textup{Isom}}

\newcommand{\Nm}{\textup{Nm}}
\newcommand{\norm}{\textup{norm}}

\newcommand{\opp}{\textup{opp}}

\newcommand{\perf}{\textup{perf}}
\newcommand\Perv{\textup{Perv}}
\newcommand{\Pic}{\textup{Pic}}
\newcommand\pr{\textup{pr}}

\newcommand\pt{\textup{pt}}
\newcommand\pure{\textup{pure}}

\newcommand\Rep{\textup{Rep}}
\newcommand{\Res}{\textup{Res}}
\newcommand\res{\textup{res}}

\newcommand\Spec{\textup{Spec}\ }

\newcommand\Stab{\textup{Stab}}
\newcommand\Supp{\textup{Supp}}
\newcommand\Sym{\textup{Sym}}

\newcommand{\Tr}{\textup{Tr}}

\newcommand{\Weil}{\textup{Weil}}

\newcommand\Aut{\textup{Aut}}
\newcommand\Hom{\textup{Hom}}
\newcommand\HOM{\textup{HOM}}
\newcommand\End{\textup{End}}

\newcommand\uHom{\underline{\Hom}}

\newcommand{\RHom}{\bR\Hom}
\newcommand{\Homb}{\Hom^{\bullet}}
\newcommand\RuHom{\bR\uHom}
\newcommand{\Ext}{\textup{Ext}}

\newcommand\uC{\underline{C}}
\newcommand\uD{\underline{\Delta}}
\newcommand\unb{\underline{\nabla}}
\newcommand\uIC{\underline{\textup{IC}}}
\newcommand\uTh{\underline{\Theta}}
\newcommand\uMM{\underline{\MM}}

\newcommand\SL{\textup{SL}}

\newcommand\SU{\textup{SU}}

\newcommand\SO{\textup{SO}}

\newcommand\Sp{\textup{Sp}}

\newcommand\Spin{\textup{Spin}}
\newcommand{\Gm}{\GG_m}
\def\Ga{\GG_a}

\newcommand{\Ad}{\textup{Ad}}

\newcommand\xch{\mathbb{X}^*}
\newcommand\xcoch{\mathbb{X}_*}

\newcommand{\incl}{\hookrightarrow}
\newcommand{\isom}{\stackrel{\sim}{\to}}
\newcommand{\bij}{\leftrightarrow}

\newcommand{\leftexp}[2]{{\vphantom{#2}}^{#1}{#2}}
\newcommand{\pH}{\leftexp{p}{\textup{H}}}

\newcommand{\Qlbar}{\overline{\QQ}_\ell}

\newcommand{\twtimes}[1]{\stackrel{#1}{\times}}

\newcommand{\jiao}[1]{\langle{#1}\rangle}
\newcommand{\wt}[1]{\widetilde{#1}}

\newcommand\quash[1]{}
\newcommand\mat[4]{\left(\begin{array}{cc} #1 & #2 \\ #3 & #4 \end{array}\right)}  
\newcommand\un{\underline}
\newcommand{\bu}{\bullet}
\newcommand{\ov}{\overline}
\newcommand{\bs}{\backslash}

\newcommand\sss{\subsubsection}
\newcommand\xr{\xrightarrow}
\newcommand\op{\oplus}
\newcommand\ot{\otimes}
\newcommand\one{\mathbf{1}}

\newcommand\vn{\varnothing}

\newcommand{\cohog}[2]{\textup{H}^{#1}({#2})}     
\newcommand{\cohoc}[2]{\textup{H}_{c}^{#1}({#2})}     
\newcommand\RG{\bR\G}
\newcommand\upH{\textup{H}}

\renewcommand\a\alpha
\renewcommand\b\beta
\newcommand\g\gamma
\newcommand\G\Gamma
\renewcommand\d\delta
\newcommand\D\Delta
\newcommand{\e}{\epsilon}
\newcommand{\io}{\iota}
\renewcommand{\k}{\kappa}
\renewcommand{\th}{\theta}
\newcommand{\Th}{\Theta}
\newcommand{\ph}{\varphi}
\renewcommand{\r}{\rho}
\newcommand{\s}{\sigma}
\renewcommand{\t}{\tau}

\newcommand{\y}{\eta}
\newcommand{\z}{\zeta}

\renewcommand{\l}{\lambda}
\renewcommand{\L}{\Lambda}
\newcommand{\om}{\omega}
\newcommand{\Om}{\Omega}

\newcommand\hs{\heartsuit}
\newcommand\na{\natural}

\newcommand\da{\dagger}

\newcommand\Ch{\textup{Ch}}
\newcommand\nb{\nabla}

\renewcommand\c{\circ}
\renewcommand\j{\jiao}

\newcommand\lmod{\textup{-mod}}
\newcommand\gmod{\textup{-gmod}}
\newcommand\SB{\textup{SB}}
\newcommand\SBim{\textup{SBim}}
\newcommand\RRM{R\ot R\textup{-gmod}}
\newcommand\RRF{(R\ot R,\Fr)\textup{-gmod}}
\newcommand\dw{\dot{w}}
\newcommand\ddw{\ddot{w}}
\newcommand\ds{\dot{s}}
\newcommand\dv{\dot{v}}
\newcommand\CS{\mathcal{CS}}

\newcommand\cont{\textup{cont}}

\newcommand\dow{\dot{\ov w}}
\newcommand\Ob{\textup{Ob}}
\newcommand\Fun{\textup{Fun}}
\newcommand\Cat{\textup{Cat}}
\newcommand\PSB{R^{M'}\textup{-Mod-}R^{M}}
\newcommand{\bc}{\mathbf{c}}
\newcommand\WL{W^{\c}_{\cL}}
\newcommand\bx{\mathbf{x}}

\title{Endoscopy for Hecke categories, character sheaves and representations}

\dedicatory{}
\author{George Lusztig}
\thanks{G.L. is partially supported by the NSF grant DMS-1566618.}
\address{Department of Mathematics, Massachusetts Institute of Technology, 77 Massachusetts Ave, Cambridge, MA 02139}
\email{gyuri@math.mit.edu}
\author{Zhiwei Yun}
\thanks{Z.Y. is partially supported by the Packard Foundation.}
\address{Department of Mathematics, Massachusetts Institute of Technology, 77 Massachusetts Ave, Cambridge, MA 02139}
\email{zyun@mit.edu}
\date{}

\begin{document}

\begin{abstract}
This is a revised version of the paper with same title published in Forum Math. Pi 8 (2020), e12. An error is corrected and main statements are simplified.

For a split reductive group $G$ over a finite field,  we show that the neutral block of its mixed Hecke category with a fixed monodromy under the torus action is monoidally equivalent to the mixed Hecke category of the corresponding endoscopic group $H$ with trivial monodromy. We also extend this equivalence to all blocks. We give two applications.  One is a relationship between character sheaves on $G$ with a fixed semisimple parameter and unipotent character sheaves on the endoscopic group $H$, after passing to asymptotic versions. The other is a similar relationship between representations of $G(\FF_{q})$ with a fixed semisimple parameter and unipotent representations of $H(\FF_{q})$.
\end{abstract}

\maketitle

\tableofcontents

\section{Introduction}

\subsection{Hecke categories}
Let $G$ be a connected split reductive group over a finite field $\FF_{q}$. Let $B$ be a Borel subgroup of $G$. The (mixed) Hecke category of $G$ is the $B$-equivariant derived category of complexes of sheaves with $\Qlbar$-coefficients on the flag variety $G/B$ of $G$ whose cohomology sheaves are mixed in the sense of \cite[1.2.2]{Del-Weil2}. We denote this category by $D^{b}_{m}(B\bs G/B)$.  The Hecke category $D^{b}_{m}(B\bs G/B)$ carries a monoidal structure under convolution. It gives a categorification of the Hecke algebra $H_{q}(W)$ attached to the Weyl group of $G$. 

The Hecke category and its variants play a central role in geometric representation theory. On the one hand, when the base field is $\CC$, the category of perverse sheaves $\Perv(B_{\CC}\bs G_{\CC}/B_{\CC})$ is equivalent to a version of category $\cO$ for the Lie algebra $\frg$ of $G_{\CC}$, by the Beilinson--Bernstein localization theorem.  The Kazhdan--Lusztig conjecture \cite{KL} relates the stalks of simple perverse sheaves on $B_{\CC}\bs G_{\CC}/B_{\CC}$  to characters of simple modules in the  category $\cO$. On the other hand,  by the work of Ben-Zvi--Nadler \cite{BN} (characteristic zero), Bezrukavnikov--Finkelberg--Ostrik \cite{BFO} (characteristic zero) and Lusztig \cite{L-center-unip} (characteristic $p>0$), the categorical center of the Hecke category is equivalent to the category of unipotent character sheaves (the exact statement varies in different papers; in particular, \cite{BFO} and \cite{L-center-unip} contain statements about the asymptotic versions), which in turn is closely related to irreducible characters of finite groups of Lie type $G(\FF_{q})$.

\subsection{Monodromic Hecke categories}
In this paper we consider the monodromic version of the Hecke category. More precisely, let $B=UT$ where $U$ is the unipotent radical of $B$ and $T$ a maximal torus. For two rank one character sheaves $\cL,\cL'$ on the torus $T$ (which is the same as a rank one local system with finite monodromy together with a rigidification at the origin), we consider the equivariant derived category ${}_{\cL'}\cD_{\cL}$ of mixed $\Qlbar$-complexes on $U\bs G/U$ under the left and right translation action of $T$ with respect to the character sheaves $\cL'$ and $\cL$ respectively.  When $\cL$ is the trivial local system, ${}_{\cL}\cD_{\cL}$ is the usual Hecke category $D^{b}_{m}(B\bs G/B)$.

In \cite[Chapter 1]{L-book}, the first-named author proves that the stalks of the simple perverse sheaves in the monodromic Hecke category ${}_{\cL'}\cD_{\cL}$ are given by Kazhdan--Lusztig polynomials for a smaller Weyl group inside $W$ defined using $\cL$ or $\cL'$. Our main result is a categorical equivalence which implies this numerical statement. To state it, we need to introduce on the one hand blocks in ${}_{\cL'}\cD_{\cL}$ and on the other hand the endoscopic group attached to $\cL$.

For simplicity let us restrict to the case $\cL'=\cL$. The monoidal category ${}_{\cL}\cD_{\cL}$ can in general be decomposed into a direct sum of subcategories called {\em blocks}. Let ${}_{\cL}\cD_{\cL}^{\c}\subset {}_{\cL}\cD_{\cL}$ be the block containing the monoidal unit. The simple perverse sheaves in ${}_{\cL}\cD_{\cL}^{\c}$ up to Frobenius twists are parametrized by a normal subgroup $\WL$ of the stabilizer of $\cL$ under $W$. For details see Definition \ref{d:block}. When the center of $G$ is connected (e.g., $G$ is of adjoint type), we have ${}_{\cL}\cD_{\cL}^{\c}={}_{\cL}\cD_{\cL}$.

Let $\Phi_{\cL}$ be the set of roots $\a$ of $(G,T)$ such that the pullback of $\cL$ along its coroot $\a^{\vee}:\Gm\to T$ is a trivial local system on $\Gm$. Then $\Phi_{\cL}$ is a root system. Let $H$ be a connected reductive group over $k$ with $T$ as a maximal torus and $\Phi_{\cL}$ as its roots.  This is the {\em endoscopic group} attached to $\cL$. The Weyl group $W_{H}$ of $H$ is canonically identified with $\WL$. The choice of the Borel $B$ of $G$ gives a Borel $B_{H}$ of $H$. Let $\cD_{H}=D^{b}_{m}(B_{H}\bs H/B_{H})$ be the usual mixed Hecke category for $H$.  

\begin{theorem}[For more precise version see Theorem \ref{th:main}]\label{th:intro neutral} There is a canonical monoidal equivalence of triangulated categories
\begin{equation*}
\Psi^{\c}_{\cL}: \cD_{H}\isom {}_{\cL}\cD^{\c}_{\cL}
\end{equation*} 
sending simple perverse sheaves to simple perverse sheaves. 
\end{theorem}

At the level of Grothendieck groups, Theorem \ref{th:intro neutral} implies an isomorphism between the Hecke algebra for $W_{H}$ and a monodromic version of the Hecke algebra defined using $W$ and $\cL$ (see \S\ref{ss:Hk}) preserving the canonical bases of the two Hecke algebras. Such a statement as well as its extension to all blocks is proved by the first-named author in \cite[1.6]{L-ConjCell} and implicitly in \cite[Lemma 34.7]{L-CSDGVII}. 

In $\cD_{H}$ there are simple perverse sheaves $\IC(w)_{H}$ (normalized by a Tate twist to be pure of weight zero) for $w\in W_{H}$. In contrast, in ${}_{\cL}\cD^{\c}_{\cL}$ we do not a priori have canonical simple perverse sheaves indexed by $w\in \WL$; it always involves a choice of a lifting $\dw$ of $w$ to $N_{G}(T)$. However, the above theorem gives canonical simple perverse sheaves  $\Psi^{\c}_{\cL}(\IC(w)_{H})\in {}_{\cL}\cD^{\c}_{\cL}$. These canonical objects, denoted $\IC(w)_{\cL}^{\da}$, are defined in Definition \ref{d:can IC} using the constructions in \S\ref{ss:rig neutral}.

As a consequence of our theorem, we prove that the stalks of $\IC(w)_{\cL}^{\da}$ are semisimple as Frobenius modules (Prop. \ref{p:stalk ss}), and similarly Frobenius semisimplicity holds for the convolution (Prop. \ref{p:conv ss}).

We also have a version of the theorem covering all blocks of ${}_{\cL'}\cD_{\cL}$ for $\cL$ and $\cL'$ in the same $W$-orbit, but it is more complicated to state. See Theorem \ref{th:all blocks}. It involves a groupoid $\fH$ whose components are torsors of endoscopic groups. Apriori, there is reason to worry that the convolution between minimal IC sheaves could involve a $3$-cocycle of the finite abelian group $\Om_{\cL}=W_{\cL}/\WL$. See \S\ref{ss:warning} for a discussion. We show that such a cocycle does not arise eventually by rigidifying minimal IC sheaves using a geometric Whittaker model \S\ref{ss:Wh}.


\subsection{Remarks on the proof}
The initial difficulty for proving Theorem \ref{th:intro neutral} lies in the fact that there is no nontrivial homomorphism between $H$ and $G$ in general. For example, when $G=\Sp_{2n}$ and $\cL$ of order $2$ and fixed by the Weyl group of $G$, we have $H\cong\SO_{2n}$. 

The strategy to prove Theorem \ref{th:intro neutral} is to relate both categories to Soergel bimodules for the Coxeter group $\WL=W_{H}$. For $\cD_{H}$, this is by now well-known, following the insight of Soergel \cite{S}: taking global sections of simple perverse sheaves on $B_{H}\bs H/B_{H}$ preserves the graded Hom spaces (see \cite[Erweiterungssatz]{S}). In this paper we develop an analog of Soergel's theory for the monodromic Hecke categories ${}_{\cL'}\cD_{\cL}$. To do this, we replace the global sections functor in the non-monodromic case by the functor co-represented by the simple perverse sheaf with largest support in each block of ${}_{\cL'}\cD_{\cL}$. We show that the resulting functor carries a monoidal structure (Corollary \ref{c:MM mono}) and preserves graded Hom spaces between simple perverse sheaves (Theorem \ref{th:Hom}).  Using this,  we show that  ${}_{\cL}\cD_{\cL}^{\c}$ is equivalent to a certain derived category of Soergel bimodules (Theorem \ref{th:SB}). 

The results in \S\ref{s:mon Hk}-\S\ref{s:Soergel} hold with the same proofs when the mixed category ${}_{\cL'}\cD_{\cL}$ is replaced with the constructible equivariant derived category for the situation over an arbitrary algebraically closed field $k$.  We expect Theorem \ref{th:intro neutral} and Theorem \ref{th:all blocks} to hold as well over any  algebraically closed base field. It is likely that the argument in \cite[\S6.5]{BF} would allow one to deduce such results from our main results over  finite fields. 

\subsection{Application to representations} Let $G_{1}$ be a form of $G_{k}$ over $\FF_{q}$. Let us recall a rough statement of the classification of irreducible characters of $G_{1}(\FF_{q})$.  By \cite[10.1]{DL}, one can assign to each irreducible $G_{1}(\FF_{q})$-representation over $\Qlbar$ a semisimple geometric conjugacy class $s$ in $G^{*}_{1}$ defined over $\FF_{q}$ (here $G^{*}_{1}$ is reductive group over $\FF_{q}$ whose root system is dual to that of $G$), called the {\em semisimple parameter} of the irreducible representation.  This assignment requires a choice of an isomorphism $\Hom_{\cont}(\varprojlim_{n}\mu_{n}(\ov\FF_{q}),\Qlbar^{\times})\isom \ov\FF^{\times}_{q}$. We may alternatively think of a semisimple parameter of $G_{1}$ as a $W$-orbit $\fo$ of character sheaves on $T_{k}$ that are stable under the Frobenius map for $G_{1}$.

Let $I_{\fo}(G(\FF_{q}))$ be the set of irreducible representations of $G_{1}(\FF_{q})$ with semisimple parameter  $\fo$. Let $I_{u}(G_{1}(\FF_{q}))$ is the set of unipotent irreducible representations of $G_{1}(\FF_{q})$ (i.e., the case $\fo$ consists of the unit element in $\Ch(T_{k})$).  It is shown by the first-named author in \cite[Theorem 4.23]{L-book} that the parametrization of $I_{\fo}(G(\FF_{q}))$ is closely related to that of $I_{u}(H_{1}(\FF_{q}))$ where $H_{1}$ is the endoscopic group attached to $\fo$, under the assumption that the center of $G$ is connected. An extension of such a relationship to all reductive groups $G_{1}$ is announced in \cite[2.1]{L-ICM} and proved in \cite{L-RepDiscCent} and \cite{L-Spin}.  

As an application of Theorem \ref{th:intro neutral},  using results from \cite{L-unip-rep} and \cite{L-nonunip-rep} relating representations of $G_{1}(\FF_{q})$ to twisted categorical centers of the Hecke categories, we prove a relationship between representations of $G_{1}(\FF_{q})$ with a fixed semisimple parameter and unipotent representations of its endoscopic group, without appealing to the classifications mentioned above. We state it under the simplifying assumption that $G_{k}$ has connected center, and the general case is in Corollary \ref{c:rep}.

\begin{theorem} Assume $G_{k}$ has connected center. Let $G_{1}$ be a form $G_{k}$ over $\FF_{q}$ and let $\fo\subset \Ch(T_{k})$ be a $W$-orbit that is a semisimple parameter for $G_{1}$. Let $\cL\in\fo$ and let $H_{k}$ be the endoscopic group of $G_{k}$ attached to $\cL\in \fo$.  Then there is a form $H_{1}$ of the endoscopic group $H$ over $\FF_{q}$ and an equivalence of categories
\begin{equation*}
\Rep^{\bc}_{\fo}(G_{1}(\FF_{q}))\cong \Rep^{\bc}_{u}(H_{1}(\FF_{q}))
\end{equation*}
for each two-sided cell $\bc$ of $(G_{1},\fo)$ (which determines a two-sided cell, also denoted $\bc$,  for unipotent representations of $H_{1}$).
\end{theorem}

\subsection{Application to character sheaves} Character sheaves on $G_{k}$ ($k=\ov\FF_{q}$) are certain simple perverse sheaves equivariant under the conjugation action of $G_{k}$.  Each character sheaf has a semisimple parameter which is a $W$-orbit $\fo\subset \Ch(T_{k})$. Unipotent character sheaves on $G_{k}$ are those with trivial semisimple parameter. From the classification of character sheaves in \cite[23.1]{L-CS5},  there is a close relationship between character sheaves on $G_{k}$ with semisimple parameter $\fo$ and unipotent character sheaves on $H_{k}$, the endoscopic group attached to some $\cL\in\fo$.


As another application of Theorem \ref{th:intro neutral},  using results from \cite{L-center-nonunip}, we derive a relationship between the asymptotic versions of character sheaves on $G_{k}$ with a fixed semisimple parameter and unipotent character sheaves on its endoscopic group. Again we state it under the simplifying assumption that $G_{k}$ has connected center, and the general case is Theorem \ref{th:ch}.

\begin{theorem}\label{th:intro CS} Assume $G_{k}$ has connected center. Let $\fo\subset \Ch(T_{k})$ be a $W$-orbit. Let  $\cL\in\fo$ and let $H_{k}$ be the endoscopic group of $G_{k}$ attached to $\cL$. Then there is a canonical equivalence of braided monoidal categories
\begin{equation*}
\un\CS^{\bc}_{\fo}(G_{k})\cong \un\CS^{\bc}_{u}(H_{k})
\end{equation*}
for each two-sided cell $\bc$ of $W_{H}=\WL$. 
\end{theorem}
For definitions of $\un\CS^{\bc}_{u}(H_{k})$ and $\un\CS^{\bc}_{\fo}(G_{k})$, see \S\ref{ss:unip chsh} and \S\ref{ss:nonunip chsh}.

In Theorem \ref{th:rep} we also prove a generalization of the above equivalence for character sheaves on disconnected groups.


\subsection{Connection to Soergel's work}\label{ss:Soergel work} In this subsection the base field is $\CC$, and we use the same notations $G,B$ and $T$ but now they are understood to be algebraic groups over $\CC$. Let $\frg, \frb$ and $\frt$ be the Lie algebras of $G,B$ and $T$. In \cite[Theorem 11]{S}, Soergel proves the following result: For a dominant but not necessarily integral character $\l\in\frt^{*}$, let $\cO^{\c}_{\l}$ be the block of the category $\cO_{\l}$ (category $\cO$ of $\frg$ with infinitesimal character corresponding to $\l$ under the Harish-Chandra isomorphism) containing the simple module $L(\l)$ with highest weight $\l$.  Then up to equivalence, $\cO_{\l}^{\c}$ only depends on the Coxeter group $(W(\l), S(\l))$, which is the Weyl group attached to the based root system $\Phi_{\l}=\{\a\in \Phi(\frg,\frt)|\jiao{\a^{\vee}, \l}\in\ZZ\}$ with positive roots $\Phi_{\l}^{+}=\Phi_{\l}\cap\Phi^{+}(\frg,\frt)$.  

From $\l$ we get a character $\xcoch(T)\subset \frt\xr{\j{-,\l}}\CC\xr{\exp(2\pi i(-))}\CC^{\times}$, giving a rank one character sheaf $\cL_{\l}$ on $T(\CC)$. By the localization theorem of Beilinson--Bernstein and the Riemann--Hilbert correspondence,  $\cO^{\c}_{\l}$ can be identified with a block ${}_{\l}P^{\c}$ in the category $\Perv_{(T,\cL_{\l})}(U\bs G/U)$ (with the $T$-action on the left). Soergel's result can then be formulated as an equivalence of abelian categories ${}_{\l}P^{\c}\cong P_{H}$, where $H$ is the endoscopic group attached to $\cL_{\l}$ and 
$P_{H}=\Perv(B_{H}\bs H/U_{H})$. We expect the method used in Soergel's paper can be extended to prove the Koszul dual version of the characteristic zero analogue of Theorem \ref{th:intro neutral}, with  equivariance replaced by weak equivariance (or monodromicity) as in \cite{BY} .

\subsection{Notation and conventions}
\sss{Frobenius}\label{sss:Fr}
Throughout the article let $k=\ov\FF_{q}$ be an algebraic closure of $\FF_{q}$.   Let $\ell$ be a prime different from $p=\ch(k)$. 

Let $\Fr\in \Gal(\ov \FF_{q}/\FF_{q})$ be the geometric Frobenius. A $\Fr$-module $M$ is a $\Qlbar$-vector space with a $\Qlbar$-linear automorphism $\Fr_{M}:M\to M$ such that each $v\in M$ is contained in a finite-dimensional subspace stable under $\Fr_{M}$.   The $\Fr$-module $M$ is called {\em pure of weight $n$} if for any eigenvalue $\l$ of $\Fr_{M}$, $\l$ is algebraic over $\QQ$ with all conjugates in $\CC$ of absolute value $q^{n/2}$.

Fix a nontrivial additive character $\psi_{0}: \FF_{q}\to \Qlbar^{\times}$.

Fix a square root $q^{1/2}$ of $q$ in $\Qlbar$. We denote by $\Qlbar(1/2)$ the one-dimensional $\Fr$-module $M=\Qlbar$ equipped with the automorphism $\Fr_{M}$ by scalar multiplication by $q^{-1/2}$.

\sss{Geometry}  We denote
\begin{equation*}
\pt:=\Spec \FF_{q}.
\end{equation*}
For a scheme $X$ over $\FF_{q}$, let $D^{b}_{m}(X)$ be the derived category of \'etale $\Qlbar$-complexes on $X$ whose cohomology sheaves are mixed, see \cite[5.1.5]{BBD}. 
If $X$ is equipped  with an action of an algebraic group $H$ over $\FF_{q}$, one can follow the method of \cite{BL} to define the $H$-equivariant derived category of mixed $\Qlbar$-complexes denoted  $D^{b}_{H,m}(X)$ or $D^{b}_{m}(H\bs X)$(i.e., working with Cartesian complexes on the standard simplicial scheme resolving the stack $H\bs X$). 

Similarly we have the (constructible, $\Qlbar$-coefficient) equivariant derived category $D^{b}(H_{k}\bs X_{k})$. We have a pullback functor $\om: D^{b}_{m}(H\bs X)\to D^{b}(H_{k}\bs X_{k})$. For $\cF,\cF'\in D^{b}_{m}(H\bs X)$, we define
\begin{equation}\label{Homconvention}
\Hom(\cF,\cF'):=\Hom_{D^{b}(H_{k}\bs X_{k})}(\om\cF,\om\cF').
\end{equation}
In other words, the Hom space between two mixed complexes on $H\bs X$ is taken to be the Hom space of their pullback to $H_{k}\bs X_{k}$. Similarly, $\Ext^{i}(\cF,\cF')$ means $\Hom(\cF,\cF'[i])$ calculated again in $D^{b}(H_{k}\bs X_{k})$, and $\bR\Hom(\cF,\cF')$ means $\bR\Hom(\om\cF,\om\cF')$, which is an object in $D^{b}_{m}(\pt)$. The actual morphisms inside the category $D^{b}_{m}(H\bs X)$ will be denoted
\begin{equation*}
\hom(\cF,\cF'):=\Hom_{D^{b}_{m}(H\bs X)}(\cF,\cF').
\end{equation*}

A {\em semisimple complex} in $D^{b}(H_{k}\bs X_{k})$ means an object isomorphic to  finite direct sum of shifted simple perverse sheaves. A {\em semisimple complex} in $D^{b}_{m}(H\bs X)$ are those whose image in $D^{b}(H_{k}\bs X_{k})$ are semisimple.

For any mixed complex $\cF\in D^{b}_{m}(H\bs X)$ and $n\in\ZZ$, we denote by $\cF(n/2):=\cF\ot\pi^{*}\Qlbar(1/2)^{\ot n}$, where $\pi:H\bs X\to \pt$ is the natural projection. Then $\cF(1)$ is the usual Tate twist. Also we define
\begin{equation}\label{jn}
\cF\j{n}:=\cF[n](n/2), \quad n\in\ZZ.
\end{equation}

For an algebraic group $H$ over $\FF_{q}$ acting on a scheme $X$ on the right and on another scheme $Y$ on the left, we denote by $X\twtimes{H}Y$ the quotient stack  $(X\times Y)/H$ where $h\in H$ acts by $h\cdot(x,y)=(xh^{-1},hy)$. 

\sss{Group theory} Let $G$ be a connected split reductive group over $\FF_{q}$. We fix a pinning $(T,B,\{\bx_{-\a}\})$ of $G$. More precisely, fix a Borel subgroup $B$ of $G$ with unipotent radical $U$ and a  maximal torus $T\subset B$. Let $\Phi(G,T)$ (resp. $\Phi^{\vee}(G,T)$) be the set of roots (resp. coroots) of $G$ with respect to $T$. The choice of $B$ gives the set of positive roots $\Phi^{+}:=\Phi^{+}(G,B,T)$, negative roots $\Phi^{-}:=\Phi^{-}(G,B,T)$, and  a set of simple roots. For each simple root $\a$ we fix an isomorphism $\bx_{-\a}: U_{-\a}\cong \Ga$ where $U_{-\a}$ is the root subgroup corresponding to $-\a$.

Let $W=N_{G}(T)/T$ be the Weyl group of $G$, with simple reflections coming from simple roots. For a simple reflection $s\in W$, let $\a_{s}$ and $\a^{\vee}_{s}$ be the corresponding simple root and simple coroot.

We use $e$ to denote the identity element of $W$. We use $\dot e$ to denote the identity element of $G$. 

For each $w\in W$, we use $\dw$ to denote a lifting of $w$ in $N_{G}(T)(\FF_{q})$. For $w=e$ we always lift it to $\dot e$, the identity element of $G$. The equivalence in Theorem \ref{th:intro neutral} will not depend on the choice of such liftings, while its extension Theorem \ref{th:all blocks} will depend on choices of liftings on a subset of $W$.

\sss{Other}\label{sss:not other}
For a triangulated category $\cD$ and $\{X_{\a}\}_{\a\in I}$ a collection of objects in $\cD$, we denote $\j{X_{\a}; \a\in I}$ the full subcategory of $\cD$ whose objects are successive extensions of objects that are isomorphic to $X_{\a}$, $\a\in I$. 

Let $S$ be a set with a left action of $H_{1}$ and a right action of $H_{2}$. We say $S$ is a {\em $(H_{1}, H_{2})$-bitorsor} if $S$ is a torsor under the $H_{1}$-action and a torsor under the $H_{2}$-action. Similarly we define the notion of bitorsors for schemes with left and right actions of group schemes.

For a category $\cC$, let $\Ob(\cC)$ be the collection of objects in $\cC$, and $|\cC|$ be the set of isomorphism classes of objects in $\cC$.



\subsection*{Acknowledgement} 
We would like to thank R.Bezrukavnikov, P.Etingof, D.Nadler and R.Rouquier for helpful discussions. We thank an anonymous referee for careful reading and for pointing out several gaps.

\section{Monodromic Hecke categories}\label{s:mon Hk}
In \S\ref{s:mon Hk}-\S\ref{s:all blocks}, we work over a fixed finite field $\FF_{q}$. In this section we introduce the main players of the paper: the monodromic Hecke categories.

\subsection{Rank one character sheaves}\label{ss:cs1}
For an algebraic group $H$ over $\FF_{q}$, there is the notion of {\em rank one character sheaves} on $H$. These are rank one $\Qlbar$-local systems $\cL$ on $H$ equipped with an isomorphism $m^{*}\cL\cong \cL\boxtimes\cL$ over $H\times H$ (where $m:H\times H\to H$ is the multiplication map) and a trivialization of the stalk $\cL_{e}$ ($e\in H$ is the identity element) satisfying the associativity and unital axioms.  We refer to \cite[Appendix A]{Yun-CDM} for a systematic treatment of rank one character sheaves. Let $\Ch(H)$ denote the group of isomorphism classes of rank one character sheaves. When $H$ is connected, the automorphisms of a rank one character sheaf reduce to identity. 

Recall $k=\ov\FF_{q}$. We define
\begin{equation*}
\Ch(H_{k})=\varinjlim_{n}\Ch(H_{\FF_{q^{n}}})
\end{equation*}
with transition maps given by the pullback.

Let $\nu:\wt H\to H$ be a finite \'etale central isogeny (where $\wt H$ is a connected algebraic group) with discrete kernel $\ker(\nu)$ (discrete as a group scheme over $\FF_{q}$, i.e., a finite abelian group). Let $\chi: \ker(\nu)\to \Qlbar^{\times}$ be a homomorphism. Then $\cL:=\nu_{*}\Qlbar[\chi]$, the sub-local system of $\nu_{*}\Qlbar$ on which $\ker(\nu)$ acts via $\chi$, is a rank one character sheaf on $H$ of finite order. It is shown in \cite[A.2]{Yun-CDM} that any element in $\Ch(H)$ arises in this way. 
 
\subsection{The case of a torus}When $H=T$ is a torus,  the Lang map $\l_{T}: T\to T$ given by $t\mapsto \Fr_{T/\FF_{q}}(t)t^{-1}$ is a finite \'etale isogeny with kernel $T(\FF_{q})$. The above construction gives a homomorphism 
\begin{equation}\label{cs1T}
\Hom(T(\FF_{q}), \Qlbar^{\times})\to \Ch(T).
\end{equation} 
This is in fact a bijection with inverse given by taking the Frobenius trace function of character sheaves, see \cite[A.3.3]{Yun-CDM}. 


\begin{lemma}\label{l:ext L} Let $H$ be a connected reductive group over $\FF_{q}$ with maximal torus $T$. Let $\cL\in\Ch(T)$. Then $\cL$ extends to a rank one character sheaf $\wt \cL\in \Ch(H)$ (necessarily unique) if and only if for every coroot $\a^{\vee}:\GG_{m,k}\to T_{k}$ of $H_{k}$,  the pullback $(\a^{\vee})^{*}\cL$ is the trivial rank one character sheaf on $\GG_{m,k}$.
\end{lemma}
\begin{proof}
First suppose $\cL$ extends to $\wt\cL\in\Ch(H)$, and denote its pullback to $H_{k}$ again by $\wt\cL$. Let $\a^{\vee}$ be a coroot of $H_{k}$ and $\ph_{\a}:\SL_{2,k}\to H_{k}$ be the homomorphism whose image is the rank one subgroup of $H_{k}$ containing the roots $\pm\a$. Let $\GG_{m,k}\subset \SL_{2,k}$ be the diagonal torus. Then $\ph_{\a}|_{\GG_{m,k}}=\a^{\vee}$. Since $\SL_{2,k}$ does not admit any nontrivial finite central isogeny, $\ph_{\a}^{*}\wt\cL$ is trivial, hence $(\a^{\vee})^{*}\cL=(\ph_{\a}^{*}\wt\cL)|_{\GG_{m,k}}$ is trivial.

Conversely, suppose $\cL$ is trivial after pullback along each coroot. The restriction map $\Ch(H)\to \Ch(T)$ is injective by \cite[A.2.2]{Yun-CDM}. Let $k'/\FF_{q}$ be a finite extension and $\s\in\Gal(k'/\FF_{q})$ is the Frobenius element. By \cite[A.1.2(4)]{Yun-CDM},  $\Ch(H)=\Ch(H_{k'})^{\s}$, therefore $\Ch(H)=\Ch(H_{k'})\cap \Ch(T)\subset \Ch(T_{k'})$. In other words, it suffices to show that $\cL_{k'}\in \Ch(T_{k'})$ extends to $H_{k'}$ for some finite extension $k'/\FF_{q}$.  Therefore we may base change the situation by a finite extension of $\FF_{q}$ so that $T$ is split. Below we assume $T$ is split. 

Let $\chi:  T(\FF_{q})=\xcoch(T)\ot_{\ZZ}\FF_{q}^{\times}\to \Qlbar^{\times}$ be the character corresponding to $\cL$ under the bijection \eqref{cs1T}.   We view $\chi$ as a homomorphism $\wt\chi: \xcoch(T)\to \Hom(\FF_{q}^{\times},\Qlbar^{\times})$. 
Let $\L=\ker(\wt\chi)$. The assumption on $\cL$ implies that $\L$ contains the coroot lattice of $H$.   By the structure theory of reductive groups, there is a connected split reductive group $\wt H$ over $\FF_{q}$ with maximal torus $\wt T$ such that $\L=\xcoch(\wt T)$ with coroots $\Phi^{\vee}(H,T)$. The embedding $\L\subset \xcoch(T)$ gives a homomorphism $\nu:\wt H\to H$ such that $\nu^{-1}(T)=\wt T$. By construction, the Lang map $\l_{T}: T\to T$ factors as $T\xr{\b}\wt T\xr{\nu|_{\wt T}} T$ such that $\ker(\b)=\ker(\chi)$ and $\ker(\nu)=\ker(\nu|_{\wt T})=T(\FF_{q})/\ker(\chi)$. Hence $\chi$ factors through a character $\ov\chi: \ker(\nu)\to \Qlbar^{\times}$. Let $\cL'=\nu_{*}\Qlbar[\ov\chi]\in \Ch(H)$. Then $\cL'_{T}=(\nu|_{\wt T})_{*}\Qlbar[\ov\chi]$ which is isomorphic to $\cL$ by construction.  Therefore $\cL$ extends to $\cL'\in \Ch(H)$.
\end{proof}

\subsection{Root system attached to $\cL$}\label{ss:rs} The action of $W$ on $T$ induces an action of $W$ on $\Ch(T)$: for $w\in W$ and  $\cL\in\Ch(T)$, define $w\cL=(w^{-1})^{*}\cL$.  For a root $\a\in \Phi(G,T)$, the action of $r_{\a}$ on $\Ch(T)$ is given by $\cL\mapsto \cL\ot (\a^{\vee}\c \a)^{*}\cL^{-1}$ (the map $\a^{\vee}\c \a: T\to\Gm\to T$). For $\cL\in \Ch(T)$, let $W_{\cL}$ be its stabilizer under $W$. 

For $\cL\in \Ch(T)$, we have a subset of the coroots  
\begin{equation*}
\Phi^{\vee}_{\cL}=\{\a^{\vee}\in \Phi^{\vee}(G,T)|(\a^{\vee})^{*}\cL\mbox{ is trivial on $\Gm$, where $\a^{\vee}:\Gm\to T$ }\}.
\end{equation*}
Let $\Phi_{\cL}$ be the subset of $\Phi(G,T)$ corresponding to $\Phi^{\vee}_{\cL}$; it is a sub root system of $\Phi(G,T)$. Let $\WL$ be the Weyl group of $\Phi_{\cL}$: it is the subgroup of $W$ generated by the reflections $r_{\a}$ for $\a\in \Phi_{\cL}$. Then $\WL$ is a normal subgroup of $W_{\cL}$ \footnote{In \cite{L-center-nonunip} and \cite{L-ConjCell}, $\WL$ and $W_{\cL}$ are denoted  $W_{\l}$ and $W_{\l}'$ respectively. }. In fact, $\Phi^{\vee}_{\cL}$ is stable under $W_{\cL}$, hence $W_{\cL}$ normalizes $\WL$. 
We will see in \S\ref{ss:endo} that if $G$ has connected center, then $\WL=W_{\cL}$.

The subset $\Phi^{+}_{\cL}=\Phi^{+}\cap \Phi_{\cL}$ (where $\Phi^{+}\subset \Phi(G,T)$ is the set of positive roots defined by $B$) gives a notion of positive roots in $\Phi_{\cL}$. This defines a Coxeter group structure on $\WL$, where the simple reflections are the reflections given by indecomposable roots in $\Phi^{+}_{\cL}$. We denote the length function of the Coxeter group $\WL$ by
\begin{equation}\label{length cL}
\ell_{\cL}: \WL\to \ZZ_{\ge0}.
\end{equation}

For $w\in W$ and $\cL\in \Ch(T)$ we have
\begin{equation}\label{RwL}
\Phi_{w\cL}=w(\Phi_{\cL})\subset \Phi(G,T).
\end{equation}

For $\cL,\cL'\in \Ch(T)$, let ${}_{\cL'}W_{\cL}=\{w\in W|w\cL=\cL'\}$.  This is non-empty only when $\cL$ and $\cL'$ are in the same $W$-orbit. When $\cL$ and $\cL'$ are in the same $W$-orbit, ${}_{\cL'}W_{\cL}$ is a $(W_{\cL'}, W_{\cL})$-bitorsor. Since $\WL$ is normal in $W_{\cL}$, for any $x\in {}_{\cL'}W_{\cL}$, we have $W^{\c}_{\cL'}x=W^{\c}_{\cL'}x\WL=x\WL$.

\subsection{Monodromic complexes}\label{ss:mono}
Let $H$ be a connected algebraic group over $\FF_{q}$ acting on a scheme $X$ of finite type over $\FF_{q}$. Let $\cL\in\Ch(H)$. We will define a triangulated category $D^{b}_{(H,\cL),m}(X)$ of mixed $\Qlbar$-complexes on $X$ equivariant with respect to $(H,\cL)$. The case where $\cL$ is trivial corresponds to the usual equivariant derived category $D^{b}_{H,m}(X)=D^{b}_{m}(H\bs X)$ as defined by Bernstein--Lunts in \cite{BL}.

By the discussion in \S\ref{ss:cs1}, there is a finite \'etale central isogeny $\nu:\wt H\to H$ and a character $\chi: \ker(\nu)\to \Qlbar^{\times}$ such that $\cL$ appears as the direct summand of $\nu_{*}\Qlbar$ where $\ker(\nu)$ acts through $\chi$. Consider the equivariant derived category $D^{b}_{m}(\wt H\bs X)$ where the action of $\wt H$ is through $H$ via $\nu$. Since the finite abelian group $\ker(\nu)$ acts trivially on $X$, it acts on the identity functor of the $\Qlbar$-linear category $D^{b}_{m}(\wt H\bs X)$. This allows us to decompose $D^{b}_{m}(\wt H\bs X)$ into a direct sum of full triangulated subcategories according to characters of $\ker(\nu)$. Let $D^b_{(H,\cL),m}(X)$ be the direct summand of $D^{b}_{m}(\wt H\bs X)$ corresponding to  $\chi$. It can be checked that, up to canonical equivalence, the category $D^b_{(H,\cL),m}(X)$ does not depend on the choice of $(\wt H,\chi)$ attached to $\cL$ (the essential point is that $\upH^{*}_{A}(\pt_{k} ,\Qlbar)=\Qlbar$ for a finite group $A$).

Similarly one defines the constructible derived category $D^{b}_{(H_{k},\cL)}(X_{k})$ for the spaces base-changed to $k=\ov\FF_{q}$. We use the convention \eqref{Homconvention} for the Hom spaces in $D^{b}_{m}(\wt H\bs X)$.

\subsection{Functoriality}\label{ss:fun}
Let $\nu: H'\to H$ be a homomorphism of algebraic groups and let $H$ act on $X$. Let $\cL\in \Ch(H)$ and $\cL'=\nu^{*}\cL\in\Ch(H')$. Let $\pi: H'\bs X\to H\bs X$ be the natural map of quotient stacks. Then we have a pair of adjoint functors
\begin{equation*}
\xymatrix{ \pi^{*}: D^{b}_{(H,\cL),m}(X)\ar@<.5ex>[r] & D^{b}_{(H',\cL'),m}(X): \pi_{*} \ar@<.5ex>[l]}
\end{equation*}
defined as follows. For $\cF\in D^{b}_{(H,\cL),m}(X)$, $\pi^{*}\cF$ has the same underlying sheaf  on $X$ as $\cF$, with the $(H',\cL')$-equivariant structure obtained by pulling back the $(H,\cL)$-equivariant structure on $\cF$.

For $\cF'\in D^{b}_{(H',\cL'),m}(X)$, consider the action map $a:H\twtimes{H'}X\to X$. The complex $\cL\boxtimes\cF'$ on $H\times X$ carries a natural $H'$-equivariant structure with respect to the action $h'(h,x)=(hh'^{-1}, h'x)$ (because $\cL$ is $(H',\cL')$-equivariant with respect to the right translation action of $H'$). Let $\cL\wt\boxtimes\cF'$ be the descent of $\cL\boxtimes\cF'$ to $H\twtimes{H'}X$ and define $\pi_{*}\cF'=a_{*}(\cL\wt\boxtimes\cF')$. Then $\pi_{*}\cF'$ carries a natural $(H,\cL)$-equivariant structure coming from the left $(H,\cL)$-equivariant structure on $\cL$ itself, and hence defines an object in $D^{b}_{(H,\cL),m}(X)$.

\subsection{Monodromic Hecke categories}\label{ss:mono Hecke} Let $\cL,\cL'\in\Ch(T)$. Apply the construction in \S\ref{ss:mono} to the $T\times T$-action on $U\bs G/U$ by $(t_{1},t_{2}): g\mapsto t_{1}gt_{2}^{-1}$, we get the category  
\begin{equation*}
{}_{\cL'}\cD_{\cL}=D^{b}_{(T\times T, \cL'\boxtimes \cL^{-1}),m}(U\bs G/U).
\end{equation*}
Note that the inverse $\cL^{-1}$ (dual local system) of $\cL$ appears in the definition, but we still write $\cL$ in our notation ${}_{\cL'}\cD_{\cL}$. 

We denote the non-mixed counterpart of ${}_{\cL'}\cD_{\cL}$ by
\begin{equation*}
{}_{\cL'}\un\cD_{\cL}=D^{b}_{(T_{k}\times T_{k}, \cL'\boxtimes \cL^{-1})}(U_{k}\bs G_{k}/U_{k}).
\end{equation*}
We have the pullback functor
\begin{equation*}
\om:{}_{\cL'}\cD_{\cL}\to {}_{\cL'}\un\cD_{\cL}.
\end{equation*}
For variants of ${}_{\cL'}\cD_{\cL}$ that we introduce later, we put an underline to denote the corresponding non-mixed version.

Each object $\cF\in {}_{\cL'}\cD_{\cL}$ is equipped with an isomorphism
\begin{equation*}
a^{*}\cF\cong \cL'\boxtimes \cF\boxtimes \cL
\end{equation*}
where $a: T\times [U\bs G/U]\times T\to [U\bs G/U]$ is the action map given by $(t_{1},g,t_{2})=t_{1}gt_{2}$, together with compatibility data. 

In particular, when $\cL=\Qlbar=\cL'$, ${}_{\cL'}\cD_{\cL}$ is the usual Hecke category $D^{b}_{m}(B\bs G/B)$.

\subsection{Basic operations}

 
Let $\cL',\cL,\cK',\cK\in\Ch(T)$. For $\cF\in {}_{\cL'}\cD_{\cL}$ and $\cG\in {}_{\cK'}\cD_{\cK}$, the inner Hom $\bR\uHom_{U\bs G/U}(\cF,\cG)$, viewed as a complex on $U\bs G/U$, defines an object in ${}_{\cL'^{-1}\ot\cK}\cD_{\cL^{-1}\ot\cK}$. This gives a bifunctor
\begin{equation*}
\bR\uHom: ({}_{\cL'}\cD_{\cL})^{\opp}\times {}_{\cK'}\cD_{\cK}\to _{\cL'^{-1}\ot\cK'}\cD_{\cL^{-1}\ot\cK}.
\end{equation*}
In particular, when $\cK=\cL$ and $\cK'=\cL'$,  $\bR\uHom_{U\bs G/U}(\cF,\cG)$ descends to $B\bs G/B$ and defines an object in the usual Hecke category $D^{b}_{m}(B\bs G/B)$.
  
We define a renormalized version of Verdier duality on ${}_{\cL'}\cD_{\cL}$. Let $\DD_{G/B}$ be the dualizing complex of $G/B$, viewed as an object in $D^{b}_{m}(B\bs G/B)$.  For $\cF\in {}_{\cL'}\cD_{\cL}$, by the discussion on the inner Hom above, we define the object 
\begin{equation*}
\DD(\cF)=\bR\uHom(\cF, \DD_{G/B})\in {}_{\cL'^{-1}}\cD_{\cL^{-1}}.
\end{equation*} 
This defines a functor
\begin{equation*}
\DD: ({}_{\cL'}\cD_{\cL})^{\opp}\to {}_{\cL'^{-1}}\cD_{\cL^{-1}}
\end{equation*}
which is an involutive equivalence of categories. We refer to this functor as the {\em Verdier duality} on ${}_{\cL'}\cD_{\cL}$.


We define the perverse $t$-structure on ${}_{\cL'}\cD_{\cL}$ in the following way. We define a full subcategory ${}_{\cL'}\cD^{\le0}_{\cL}$ (resp. ${}_{\cL'}\cD^{\ge0}_{\cL}$) to consist of objects $\cF\in {}_{\cL'}\cD_{\cL}$ such that $\cF[\dim T]$, as a complex on $G/U$, lies in ${}^{p}D^{\le0}(G/U)$ (resp. ${}^{p}D^{\ge0}(G/U)$). Then $({}_{\cL'}\cD^{\le0}_{\cL},{}_{\cL'}\cD^{\ge0}_{\cL})$ defines a $t$-structure, which we shall call the {\em perverse $t$-structure} on ${}_{\cL'}\cD_{\cL}$. With this definition, the Verdier duality functor $\DD$  sends perverse sheaves in ${}_{\cL'}\cD_{\cL}$ to perverse sheaves in ${}_{\cL'^{-1}}\cD_{\cL^{-1}}$.

\subsection{Strata}
For $w\in W$, let $G_{w}\subset G$ be the $B$-double coset $G_{w}=BwB$ (this is abuse of notation as we should have written $B\dw B$ for some $\dw\in N_{G}(T)(\FF_{q})$ lifting $w$, however the resulting subscheme is independent of the choice of the lifting. In the sequel we will use such abuse of notation freely). Let $G_{\le w}$ be the closure of $G_{\le w}$ and $G_{<w}=G_{\le w}-G_{w}$.

Let ${}_{\cL'}\cD(w)_{\cL}:=D^{b}_{(T\times T,\cL'\boxtimes \cL^{-1}),m}(U\bs G_{w}/U)$. Similarly define ${}_{\cL'}\cD(\le w)_{\cL}$ and ${}_{\cL'}\cD(<w)_{\cL}$ by replacing $G_{w}$ with $G_{\le w}$ and $G_{<w}$.

The inclusion $i_{w}: U\bs G_{w}/U\incl U\bs G/U$ induces adjoint pairs
\begin{eqnarray*}
\xymatrix{    i_{w,!}:     {}_{\cL'}\cD(w)_{\cL} \ar@<.5ex>[r] &  {}_{\cL'}\cD_{\cL}: i^{!}_{w}\ar@<.5ex>[l] }\\
\xymatrix{ i^{*}_{w}:     {}_{\cL'}\cD_{\cL} \ar@<.5ex>[r] & {}_{\cL'}\cD(w)_{\cL}: i_{w, *} \ar@<.5ex>[l]}
\end{eqnarray*}

Let $\G(w)\subset T\times T$ be the graph consisting of $(wt, t), t\in T$.  The $T\times T$-action  on any point in $U\bs G_{w}/U$  has stabilizer $\G(w)$. From the definitions we have the following lemma.

\begin{lemma}\label{l:one cell}
The category ${}_{\cL'}\cD(w)_{\cL}$ is zero unless $\cL'=w\cL$.  When $\cL'=w\cL$, taking stalk at the lifting $\dw\in N_{G}(T)(\FF_{q})$ of $w$ induces an equivalence
\begin{equation*}
i_{\dw}^{*}: {}_{w\cL}\cD(w)_{\cL}\isom \cD^{b}_{\G(w),m}(\{\dw\}).
\end{equation*}
Here we are using that $w^{*}\cL'=w^{-1}(\cL')\cong\cL$ so that $\cL'\boxtimes \cL^{-1}$ restricts to the trivial character sheaf on $\G(w)$.
\end{lemma}

\subsection{Some objects}\label{ss:obj}  In view of Lemma \ref{l:one cell}, ${}_{\cL'}\cD_{\cL}=0$ unless $\cL$ and $\cL'$ are in the same $W$-orbit of $\Ch(T)$. In the remainder of the paper, we fix a $W$-orbit $\fo\subset \Ch(T)$.

For $w\in W$ with lifting $\dw$,  and $\cL\in \fo$, let $C(\dw)_{\cL}\in {}_{w\cL}\cD(w)_{\cL}$ be the object that corresponds to the constant sheaf $\Qlbar\j{\ell(w)}$ under the equivalence $i_{\dw}^{*}$ in Lemma \ref{l:one cell}.  Note that the isomorphism class of $C(\dw)_{\cL}$ is independent of the lifting $\dw$ while for different liftings the identifications between the $C(\dw)_{\cL}$'s are only unique up to scalars.

Define the following perverse sheaves  in ${}_{w\cL}\cD_{\cL}$
\begin{eqnarray}\label{std obj}
\D(\dw)_{\cL}=i_{w,!}C(\dw)_{\cL}, \quad \nb(\dw)_{\cL}=i_{w,*}C(\dw)_{\cL}, \\
 \IC(\dw)_{\cL}=i_{w,!*}C(\dw)_{\cL}:=\Im(\D(\dw)_{\cL}\to\nb(\dw)_{\cL}).
\end{eqnarray}

\begin{remark} 
The isomorphism classes of $\om C(\dw)_{\cL}, \om \D(\dw)_{\cL}, \om \nb(\dw)_{\cL}$ and $\om \IC(\dw)_{\cL}$ in ${}_{\cL'}\un\cD_{\cL}$ are independent of the lifting $\dw$.  For this reason, we denote these isomorphism classes in ${}_{\cL'}\un\cD_{\cL}$ by
\begin{equation*}
\uC(w)_{\cL}\in {}_{w\cL}\un\cD(w)_{\cL}, \quad \uD(w)_{\cL}, \quad \unb(w)_{\cL} \textup{ and }\uIC(w)_{\cL}\in {}_{w\cL}\un\cD_{\cL}.
\end{equation*}
However, in the mixed category ${}_{w\cL}\cD_{\cL}$, if we change the lifting $\dw$ to another lifting $\ddot{w}=\dw t^{-1}$ ($t\in T(\FF_{q})$), then we have a canonical isomorphism in ${}_{w\cL}\cD_{\cL}$
\begin{equation}\label{tensor Lt}
\IC(\dw)_{\cL}\cong \IC(\ddot w)_{\cL}\ot\cL_{t}
\end{equation}
where $\cL_{t}$ (the stalk of $\cL$ at $t$) is viewed as a one-dimensional $\Fr$-module. Similar isomorphisms hold for $C(\dw)_{\cL}, \D(\dw)_{\cL}$ and $\nb(\dw)_{\cL}$.
\end{remark}


\section{Convolution}
In this section we define and study properties of the convolution functor on the monodromic Hecke categories. We also use convolution to prove the parity and purity properties of $\IC(\dw)_{\cL}$ in Proposition \ref{p:parity purity}.

\subsection{Convolution}\label{ss:conv} Recall that we fix a $W$-orbit $\fo\subset \Ch(T)$. Let $\cL,\cL',\cL''\in\fo$. Consider the diagram
\begin{equation*}
\xymatrix{& U\bs G\twtimes{U}G/U \ar[ld]_{\pi}\ar[d]\\
U\bs G/U\times U\bs G/U & U\bs G\twtimes{B}G/U\ar[r]^-{m} & U\bs G/U
}
\end{equation*}
For $\cF\in {}_{\cL''}\cD_{\cL'}$ and $\cG\in {}_{\cL'}\cD_{\cL}$, $\pi^{*}(\cF\boxtimes \cG)$ carries an equivariant structure under the $T$-action on $U\bs G\twtimes{U}G/U$ given by $T\ni t: (g_{1},g_{2})\mapsto (g_{1}t^{-1}, tg_{2})$ (using that $\cF$ is $(T,\cL'^{-1})$-equivariant for the second $T$-action and $\cG$ is $(T,\cL')$-equivariant for the first $T$-action), therefore it descends to a complex $\cF\wt\boxtimes \cG\in D^{b}_{(T\times T, \cL''\boxtimes \cL^{-1}),m}(U\bs G\twtimes{B}G/U)$. Define
\begin{equation*}
\cF\star\cG=m_{*}(\cF\wt\boxtimes \cG)\in {}_{\cL''}\cD_{\cL}.
\end{equation*}
This construction gives a convolution bifunctor
\begin{equation*}
(-)\star(-): {}_{\cL''}\cD_{\cL'}\times {}_{\cL'}\cD_{\cL}\to {}_{\cL''}\cD_{\cL}.
\end{equation*}
It is easy to see that convolution carries a natural associativity structure in the obvious sense. Under convolution, ${}_{\cL}\cD_{\cL}$ becomes a monoidal category with the unit object
\begin{equation*}
\d_{\cL}:=\D(\dot e)_{\cL}\cong\IC(\dot e)_{\cL}\cong\nb(\dot e)_{\cL}.
\end{equation*}

The properness of the multiplication map $m: G\twtimes{B}G\to G$ implies the following.
\begin{lemma}\label{l:conv vs duality} There is a natural isomorphism functorial in $\cF\in {}_{\cL''}\cD_{\cL'}$ and $\cG\in {}_{\cL'}\cD_{\cL}$
\begin{equation*}
\DD(\cF\star\cG)\cong \DD(\cF)\star\DD(\cG).
\end{equation*}
\end{lemma}

\begin{lemma}\label{l:conv pure}
\begin{enumerate}
\item If $\cF\in {}_{\cL''}\cD_{\cL'}$ and $\cG\in {}_{\cL'}\cD_{\cL}$ are semisimple complexes, so is $\cF\star\cG$.
\item If $\cF\in {}_{\cL''}\cD_{\cL'}$ and $\cG\in {}_{\cL'}\cD_{\cL}$ are pure of weight zero, so is $\cF\star\cG$.
\end{enumerate}
\end{lemma}
\begin{proof} 
(2) follows from the properness of the multiplication map $m: G\twtimes{B}G\to G$ and Deligne's weight estimates \cite[5.1.14]{BBD}.  (1) follows from the properness of the multiplication map $m: G\twtimes{B}G\to G$ and the decomposition theorem \cite{BBD}. 
\end{proof}

\begin{lemma}\label{l:basic conv}
Suppose $\ell(w_{1}w_{2})=\ell(w_{1})+\ell(w_{2})$, then there are canonical isomorphisms
\begin{equation*}
\D(\dw_{1})_{w_{2}\cL}\star\D(\dw_{2})_{\cL}\cong \D(\dw_{1}\dw_{2})_{\cL}, \quad
\nb(\dw_{1})_{w_{2}\cL}\star\nb(\dw_{2})_{\cL}\cong \nb(\dw_{1}\dw_{2})_{\cL}, 
\end{equation*}
\end{lemma}
\begin{proof}
Both isomorphisms follows directly from the fact that the multiplication map $G_{w_{1}}\twtimes{B}G_{w_{2}}\to G_{w_{1}w_{2}}$ is an isomorphism if $\ell(w_{1}w_{2})=\ell(w_{1})+\ell(w_{2})$.
\end{proof}

\begin{lemma}\label{l:invertible} Let $w\in W$.  Then there are isomorphisms
\begin{equation*}
\D(\dw^{-1})_{w\cL}\star \nb(\dw)_{\cL}\cong \D(e)_{\cL}\cong \nb(\dw^{-1})_{w\cL}\star \D(\dw)_{\cL}.
\end{equation*}
In particular, the functor
\begin{equation*}
(-)\star\D(\dw)_{\cL}: {}_{\cL'}\cD_{w\cL}\to {}_{\cL'}\cD_{\cL}
\end{equation*}
is an equivalence of categories with inverse given by $(-)\star\nb(\dw^{-1})_{w\cL}$.
\end{lemma}
\begin{proof}
Writing $w$ into a reduced word in simple reflections and using Lemma \ref{l:basic conv}, it is enough to prove the statements for $w=s$ a simple reflection. When $s$ a simple reflection,    we may replace $G$ by its Levi subgroup $L_{s}$ with roots $\pm\a_{s}$. Therefore it suffices to treat the case $G$ has semisimple rank one. In this case,  from the definition of convolution, we have
\begin{equation*}
i_{\dot s}^{*}(\D(\dot s^{-1})_{s\cL}\star \nb(\dot s)_{\cL})=\cohog{*}{(G/B)_{k}, \D(\dot s^{-1})_{s\cL}\ot \inv^{*}R_{\dot s}^{*}\nb(\dot s)_{\cL}}.
\end{equation*}  
Here $R_{\dot s}:U\bs G\to U\bs G$ is the right translation by $\dot s$ and $\inv: G/U\to U\bs G$ is given by inversion. Now $\D(\dot s^{-1})_{s\cL}$ is $(T,s\cL)$-equivariant with respect to the right translation of $T$ on $G/U$, and $\inv^{*}R_{\dot s}^{*}\nb(\dot s)_{\cL}$ is $(T, s\cL^{-1})$-equivariant with respect to the right translation, their tensor product is $T$-equivariant on the right hence descends to $G/B$. We choose an identification $G/B\cong \PP^{1}$ such that the unit coset $B$ corresponds to $0\in\PP^{1}$, and $sB$ corresponds to $\infty\in \PP^{1}$. Let $Y_{0}, Y_{\infty}$ be the preimages of $0,\infty$ in $Y=G/U$, and let $j_{0}: Y-Y_{0}\incl Y$, $j_{\infty}: Y-Y_{\infty}\incl Y$ be open embeddings. Then $\D(\dot s^{-1})_{s\cL}\cong j_{0!}\cK$ for some rank one tame local system $\cK$ on $Y-Y_{0}$, and $\inv^{*}R_{\dot s}^{*}\nb(\dot s)_{\cL}\cong j_{\infty*}\cK'$ for some rank one tame local system $\cK'$ on $Y-Y_{\infty}$ (tame means the corresponding representation of the fundamental group factors through the tame fundamental group). The tensor product $j_{0!}\cK\ot j_{\infty*}\cK'$ descends to a complex $\cG$ on $\PP^{1}$, which is a rank one tame local system $\cK''$ (descent of $\cK\ot\cK'|_{Y-Y_{0}-Y_{\infty}}$) on $\PP^{1}-\{0,\infty\}$ with $!$-extension at $0$ and $*$-extension at $\infty$. Therefore 
\begin{equation*}
i_{\dot s}^{*}(\D(\dot s^{-1})_{s\cL}\star \nb(\dot s)_{\cL})\cong \cohog{*}{(G/B)_{k}, \cG}
\end{equation*}
is the cone shifted by $[-1]$ of the restriction map
\begin{equation*}
\cohog{*}{\PP^{1}_{k}-\{0,\infty\}, \cK''}\to i_{\infty}^{*}j_{*}\cK''
\end{equation*}
where $j:\PP^{1}-\{0,\infty\}\to \PP^{1}$ and $i_{\infty}: \{\infty\}\incl \PP^{1}$ are the inclusions. Since $\cK''$ is a tame local system on $\PP^{1}-\{0,\infty\}\cong\Gm$, the above restriction map is an isomorphism.  This shows that the stalk of $\D(\dot s^{-1})_{s\cL}\star \nb(\dot s)_{\cL}$ vanishes at any lifting $\dot s$ of $s$. Hence $\D(\dot s^{-1})_{s\cL}\star \nb(\dot s)_{\cL}$ is concentrated in the closed stratum $U\bs G_{e}/U$.

We calculate the stalk of $\D(\dot s^{-1})_{s\cL}\star \nb(\dot s)_{\cL}$ at $e$ by the same method
\begin{equation*}
i_{e}^{*}(\D(\dot s^{-1})_{s\cL}\star \nb(\dot s)_{\cL})\cong\cohog{*}{(G/B)_{k}, \D(\dot s^{-1})_{s\cL}\ot \inv^{*}\nb(\dot s)_{\cL}}.
\end{equation*}
Now $\D(\dot s^{-1})_{s\cL}\ot \inv^{*}\nb(\dot s)_{\cL}$ is the extension by zero of the trivial local system on $\PP^{1}-\{0\}$ whose stalk at $\infty$ (image of $\dot s$ under $G/U\to \PP^{1}$) is canonically identified with $\Qlbar\j{2}$, therefore its cohomology is canonically isomorphic to $\cohoc{*}{\PP^{1}_{k}-\{0\}, \Qlbar\j{2}}\cong\Qlbar$. This gives the canonical isomorphism $\D(\dot s^{-1})_{s\cL}\star \nb(\dot s)_{\cL}\cong \D(e)_{\cL}$. The second isomorphism follows from the first one by applying Verdier duality and Lemma \ref{l:conv vs duality}.
\end{proof}

\begin{lemma}\label{l:clean s equiv} Let $s\in W$ be a simple reflection and $s\notin \WL$. 
\begin{enumerate}
\item The natural maps $\Delta(\dot s)_{\cL}\to \IC(\dot s)_{\cL}\to \nb(\dot s)_\cL$ are isomorphisms.

\item The functor
\begin{equation*}
(-)\star \IC(\dot s)_{\cL}: {}_{\cL'}\cD_{s\cL}\to {}_{\cL'}\cD_{\cL}
\end{equation*}
is an equivalence of categories with inverse given by $(-)\star \IC(\dot s^{-1})_{s\cL}$.

\item The equivalence $(-)\star \IC(\dot s)_{\cL}$ sends  $\D(\dw)_{s\cL}$, $\nb(\dw)_{s\cL}$ and $\IC(\dw)_{s\cL}\in {}_{ws\cL}\cD_{s\cL}$ to   $\D(\dw\ds)_{\cL}$, $\nb(\dw\ds)_{\cL}$ and $\IC(\dw\ds)_{\cL}\in {}_{ws\cL}\cD_{\cL}$ respectively, for any $w\in {}_{\cL'}W_{s\cL}$.
\end{enumerate}
\end{lemma}
\begin{proof}
(1) We need to show that $i^{*}_{e}\IC(\ds)_{\cL}=0$ and  $i^{!}_{e}\IC(\ds)_{\cL}=0$. Replacing $G$ by its Levi subgroup $L_{s}$ containing $T$ and with roots $\pm\a_{s}$, we reduce to the case $G$ has semisimple rank one. In this case, there is a central isogeny $\nu: Z^{\c}\times \SL_{2}\to G$ where $Z^{\c}$ is the neutral component of the center of $G$. Let $\cL_{1}=(\a^{\vee})^{*}\cL\in\Ch(\Gm)$. The condition $s\notin \WL$ is equivalent to $\cL_{1}$ being nontrivial. Identify $\Gm$ with the diagonal torus $T_{1}\subset\SL_{2}$, $\IC(\ds)_{\cL_{1}}\in D^{b}_{(\Gm\times\Gm, s\cL_{1}\boxtimes\cL_{1}^{-1}), m}(U_{1}\bs \SL_{2}/U_{1})$ is defined ($U_{1}\subset \SL_{2}$ is the unipotent upper triangular subgroup).  Let $\cL_{0}=\cL|_{Z^{\c}}$.  Then $\nu^{*}C(\ds)_{\cL}\cong \cL_{0}\boxtimes C(\ds)_{\cL_{1}}$ on the open stratum of $Z^{\c}\times U_{1}\bs \SL_{2}/U_{1}$. Since $\nu$ is finite, $\IC(\ds)_{\cL}$ is a direct summand of $\nu_{*}(\cL_{0}\boxtimes \IC(\ds)_{\cL_{1}})$. By proper base change, it suffices to show that the stalks and costalks of $\IC(\ds)_{\cL_{1}}$ vanish along the identity coset of $U_{1}\bs \SL_{2}/U_{1}$. We identify $\SL_{2}/U_{1}$ with $\AA^{2}-\{0\}$ with $\SL_{2}$ acting as the standard representation on $\AA^{2}$. The right $T_{1}$-translation on $\SL_{2}/U_{1}$ is the scaling action of $\Gm$ on $\AA^{2}-\{0\}$. The open stratum $(\SL_{2}-B_{1})/U_{1}\subset \SL_{2}/U_{1}$ is $j: \AA^{1}\times\Gm\incl \AA^{2}-\{0\}$. A direct calculation shows that $C(\ds)_{\cL_{1}}\cong \pr_{2}^{*}\cL_{1}$, where $\pr_{2}:\AA^{1}\times\Gm\to\Gm$ is the projection to the second factor. Since $\cL_{1}$ is nontrivial, $\IC(\ds)_{\cL_{1}}=j_{!*}C(\ds)_{\cL_{1}}$ has zero stalk and costalk along the closed stratum $\Gm\times\{0\}\subset \AA^{2}-\{0\}$. This proves (1).
 
(2) follows from (1) and Lemma \ref{l:invertible}.

(3) If $\ell(ws)>\ell(w)$, then by (1), $\D(\dw)_{s\cL}\star\IC(\ds)_{\cL}\cong \D(\dw)_{s\cL}\star\D(\ds)_{\cL}$, which is isomorphic to $\D(\dw\ds)_{\cL}$ by Lemma \ref{l:basic conv}. If $\ell(ws)<\ell(w)$, then by (1), $\D(\dw)_{s\cL}\star\IC(\ds)_{\cL}\cong \D(\dw)_{s\cL}\star\nb(\ds)_{\cL}$. By Lemma \ref{l:basic conv} we have $\D(\dw)_{s\cL}\cong \D(\dw\ds)_{\cL}\star\D(\ds^{-1})_{s\cL}$, therefore $\D(\dw)_{s\cL}\star\nb(\ds)_{\cL}\cong \D(\dw\ds)_{\cL}\star\D(\ds^{-1})_{s\cL}\star\nb(\ds)_{\cL}\cong \D(\dw\ds)_{\cL}\star\D(\dot e)_{\cL}\cong \D(\dw\ds)_{\cL}$ by Lemma \ref{l:invertible}. In any case we have $\D(\dw)_{s\cL}\star\IC(\ds)_{\cL}\cong \D(\dw\ds)_{\cL}$.

The proof of $\nb(\dw)_{s\cL}\star\IC(\ds)_{\cL}\cong \nb(\dw\ds)_{\cL}$ is similar. 

The equivalence $(-)\star\IC(\ds)_{\cL}$ then preserves the standard objects and costandard objects, hence it is $t$-exact for the perverse $t$-structure, and sends simple perverse sheaves to  simple perverse sheaves. Now for $w\in W$,  $\IC(\dw)_{s\cL}\star\IC(\ds)_{\cL}$ is a simple perverse sheaf in ${}_{\cL'}\cD_{\cL}$. Since $\IC(\dw)_{s\cL}\star\IC(\ds)_{\cL}$ receives a nonzero map from $\D(\dw)_{s\cL}\star\IC(\ds)_{\cL}\cong\D(\dw\ds)_{\cL}$, it must be isomorphic to $\IC(\dw\ds)_{\cL}$.
\end{proof}

\subsection{The object  $\IC(s)_{\cL}$ when $s\in \WL$}\label{ss:inert s} Suppose $s\in \WL$. Let $\a_{s}$ be the simple root corresponding to $s$. Let $P_{s}$ be the standard parabolic subgroup whose Levi subgroup $L_{s}$ has roots $\{\pm \a_{s}\}$. Let $U^{s}$ be the unipotent radical of $P_{s}$. Since $(\a^{\vee}_{s})^{*}\cL$ is trivial, the local system $\cL$ extends to a rank one character sheaf $\wt\cL$ on $L_{s}$ by Lemma \ref{l:ext L}. In particular, the stalk of $\wt\cL$ at $e\in L_{s}$ has a canonical trivialization. We use the same notation $\wt\cL$ to denote its pullback to $P_{s}$. The object $\wt\cL\j{1}\in \cD^{b}_{(T\times T, \cL\boxtimes\cL^{-1}),m}(U\bs P_{s}/U)$, extended by zero, can be viewed as an object of ${}_{\cL}\cD_{\cL}$, and as such, it is isomorphic to $\IC(\ds)_{\cL}$.
 
In other words, when $s\in \WL$, we have a {\em canonical} object $\IC(s)_{\cL}:=i_{\le s*}\wt\cL\j{1}\in {}_{\cL}\cD_{\cL}$  equipped with an isomorphism  of its stalk at the identity $e$ with $\Qlbar\j{1}$. We have $\om\IC(s)_{\cL}\cong\om\IC(\ds)_{\cL}$.

Let ${}_{\cL'}\cD_{\wt \cL}=\cD^{b}_{(T\times L_{s}, \cL'\boxtimes\wt\cL^{-1}),m}(U\bs G/U^{s}) $, where the action of $T\times L_{s}$ on $U\bs G/U^{s}$ is given by $(t,h)\cdot g=tgh^{-1}$, $t\in T, h\in L_{s}, g\in G$.   Applying the constructions in \S\ref{ss:fun}, we get an adjoint pair
\begin{equation*}
\xymatrix{\pi^{*}_{s}: {}_{\cL'}\cD_{\wt\cL}\ar@<.5ex>[r] & {}_{\cL'}\cD_{\cL}: \pi_{s*} \ar@<.5ex>[l].
}
\end{equation*}
Since $P_{s}/B$ is proper, $\pi_{s*}$ also admits a right adjoint
\begin{equation*}
\xymatrix{\pi_{s*}: {}_{\cL'}\cD_{\cL}\ar@<.5ex>[r] & {}_{\cL'}\cD_{\wt\cL}: \pi^{!}_{s}\ar@<.5ex>[l]}
\end{equation*} 
and $\pi^{!}_{s}\cong \pi^{*}_{s}\j{2}$.

\begin{lemma}\label{l:inert s} Let $\cL,\cL'\in \fo$ and $s$ be a simple reflection in $W$ such that $s\in \WL$, then there is a canonical isomorphism of endo-functors
\begin{equation*}
(-)\star\IC(s)_{\cL}\cong \pi_{s}^{*}\pi_{s*}(-)\j{1}\cong \pi_{s}^{!}\pi_{s*}(-)\j{-1}\in \End({}_{\cL'}\cD_{\cL}).
\end{equation*}
\end{lemma}
\begin{proof}
Let $a: U\bs G\twtimes{B}P_{s}\to U\bs G$ be map given by the right action of $P_{s}$ on $G$. By the definition of convolution and $\IC(s)_{\cL}=i_{\le s*}\wt\cL\j{1}$, we have for $\cF\in{}_{\cL'}\cD_{\cL}$
\begin{equation*}
\cF\star\IC(s)_{\cL}\cong a_{*}(\cF\wt\boxtimes\wt\cL\j{1})
\end{equation*}
where $\cF\wt\boxtimes\wt\cL\j{1}$ is the descent of $\cF\boxtimes\wt\cL\j{1}$ to $U\bs G\twtimes{B}P_{s}$. Comparing with the definition of $\pi_{s*}$, we see that $a_{*}(\cF\wt\boxtimes\wt\cL)$ is exactly the underlying complex of $\pi_{s*}\cF$. If we only remember the $(T,\cL)$-equivariance of $a_{*}(\cF\wt\boxtimes\wt\cL\j{1})$ (by right translation), it is the same as $\pi^{*}_{s}\pi_{s*}\cF\j{1}$.  
\end{proof}

\begin{cor}\label{c:adj ICs} Let $s\in W$ be a simple reflection such that $s\in \WL$. Then the functor 
\begin{equation*}
(-)\star\IC(s)_{\cL}: {}_{\cL'}\cD_{\cL}\to {}_{\cL'}\cD_{\cL}
\end{equation*}
has a right adjoint also given by $(-)\star\IC(s)_{\cL}$.
\end{cor}
\begin{proof}
Let $\cF\in {}_{\cL'}\cD_{s\cL}$ and $\cG\in {}_{\cL'}\cD_{\cL}$.  We have natural isomorphisms by Lemma \ref{l:inert s}
\begin{eqnarray*}
\Hom(\cF\star\IC(s)_{\cL}, \cG)&\cong& \Hom(\pi^{*}_{s}\pi_{s*}\cF\j{1}, \cG)\cong\Hom(\pi_{s*}\cF\j{1}, \pi_{s*}\cG)\\
&\cong&\Hom(\cF, \pi_{s}^{!}\pi_{s*}\cG\j{-1})\cong \Hom(\cF, \cG\star\IC(s)_{\cL}).
\end{eqnarray*}
\end{proof}

\subsection{}\label{ss:descend IC} Let $s$ be a simple reflection in $W$ and $s\in \WL$. Recall the rank one character sheaf $\wt\cL$ on $L_{s}$, the category ${}_{\cL'}\cD_{\wt\cL}$ and the functors $(\pi^{*}_{s},\pi_{s*})$ from \S\ref{ss:inert s}. 

Bruhat decomposition gives $G=\sqcup_{\ov w\in W/\j{s}}B\ov wP_{s}$. For a lifting $\dw\in N_{G}(T)$ of $\ov w\in W/\j{s}$, we have an isomorphism 
\begin{equation*}
U\bs B\ov w P_{s}/U^{s}\cong \dw \cdot (\Ad(\dw^{-1})U\cap L_{s})\bs L_{s}.
\end{equation*}
The left translation by $t\in T$ on the left side becomes the left translation of $\Ad(\dw^{-1})t$ on $(\Ad(\dw^{-1})U\cap L_{s})\bs L_{s}$. From this we get an equivalence
\begin{equation*}
i_{\dw}^{*}: {}_{\cL'}\cD(\ov w)_{\wt\cL}:=D^{b}_{(T\times L_{s}, \cL'\boxtimes\wt\cL^{-1}),m}(U\bs B\ov wP_{s}/U^{s})\cong D^{b}_{(T, \cL'\ot \ov w\wt\cL^{-1}|_{T}),m}(\{\dw\})=D^{b}_{(T, \cL'\ot \ov w\cL^{-1}),m}(\{\dw\}).
\end{equation*}
Therefore ${}_{\cL'}\cD(\ov w)_{\wt\cL}=0$ unless $\cL'=\ov w\cL$, in which case it is equivalent to $D^{b}_{T,m}(\{\dw\})$. Let $C(\dw)_{\wt\cL}\in {}_{\cL'}\cD(\ov w)_{\wt\cL}$ correspond to $\Qlbar\j{\ell'(\ov w)}$ under $i^{*}_{\dw}$ (here $\ell'(\ov w)$ is the maximal length of elements in the coset $\ov w\in W/\j{s}$).   Let $\IC(\dw)_{\wt\cL}\in {}_{\cL'}\cD_{\wt\cL}$ be the middle extension of $C(\dw)_{\wt\cL}$. The isomorphism classes of $\om C(\dw)_{\wt\cL}$ and $\om\IC(\dw)_{\wt\cL}$ depend only on $\ov w$. Similarly one defines the $!$ and $*$-extensions $\D(\dw)_{\wt\cL}$ and $\nb(\dw)_{\wt\cL}$ of $C(\dw)_{\wt\cL}$.

\begin{lemma}\label{l:IC descend} Let $s$ be a simple reflection in $W$ and $s\in \WL$.
Suppose $\ell(w)>\ell(ws)$, then 
\begin{equation*}
\pi^{*}_{s}\IC(\dw)_{\wt\cL}\cong \IC(\dw)_{\cL}.
\end{equation*}
\end{lemma}
\begin{proof}
Unwinding the definition of $\pi^{*}_{s}$, it is the pullback along the smooth map $\wt\pi_{s}:  G/\wt B\to G/\wt P_{s}$. Here we take  a finite \'etale central isogeny $\nu:\wt L_{s}\to L_{s}$ such that $\wt\cL$ is defined in terms of a character of $\ker(\nu)$ as in \S\ref{ss:cs1}; $\wt P_{s}=P_{s}\times_{L_{s}}\wt L_{s}$ and $\wt B=B\times_{L_{s}}\wt L_{s}$. Since $\wt\pi_{s}$ is a smooth $\PP^{1}$-fibration, $\wt\pi^{*}_{s}$ sends simple perverse sheaves to simple perverse sheaves up to a shift. In particular, $\wt\pi^{*}_{s}\IC(\dw)_{\wt\cL}$ is the middle extension of $\wt\pi^{*}_{s}C(\dw)_{\wt\cL}\j{1}$, a shifted local system on $\wt\pi^{-1}_{s}(B\ov wP_{s}/\wt P_{s})=(G_{w}\cup G_{ws})/\wt B$. By looking at stalks at $\dw$, we have  $\wt\pi^{*}_{s}C(\dw)_{\wt\cL}|_{G_{w}/\wt B}\cong C(\dw)_{\cL}$, therefore their middle extensions agree, i.e., $\wt\pi^{*}_{s}\IC(\dw)_{\wt\cL}\cong \IC(\dw)_{\cL}$.
\end{proof}

The next Proposition shows that the stalks and costalks of $\IC(\dw)_{\cL}$ have the same parity and purity properties as their non-monodromic counterparts.

\begin{prop}\label{p:parity purity} Let $w\in W$, $\cL\in\fo$ and $v\in {}_{w\cL}W_{\cL}$.
\begin{enumerate} 
\item The complexes $i^{*}_{v}\IC(\dw)_{\cL}$ and $i^{!}_{v}\IC(\dw)_{\cL}$ are pure of weight zero as objects in ${}_{w\cL}\cD(v)_{\cL}$.
\item The (non-mixed) complexes $i^{*}_{v}\uIC(w)_{\cL}$ and $i^{!}_{v}\uIC(w)_{\cL}$ are isomorphic to direct sums of $\uC(v)_{\cL}[n]$ for $n\equiv \ell(w)-\ell(v) \mod 2$.
\end{enumerate}
\end{prop}
\begin{proof}
(1) It is enough to show the statement for the stalks;  the costalk statement follows by Verdier duality.  We will prove the stalk statement in (1) together with a weak version of the stalk statement in (2) simultaneously by induction on $\ell(w)$. 

Denote by $\Fr\lmod_{0}$ the category of finite-dimensional $\Fr$-modules pure of weight zero (see \S\ref{sss:Fr}). We show by induction on $\ell(w)$ that for any $x\in G$, the stalk
\begin{equation}\label{stalk purity}
i_{x}^{*}\IC(\dw)_{\cL}\in \j{\Qlbar\j{n}\ot V; n\equiv \ell(w)-\ell(v) \mod 2, V\in\Fr\lmod_{0}}.
\end{equation}
For notation $\j{\cdots}$, see \S\ref{sss:not other}. The truth of this statement is independent of the lifting $\dw$ of $w$. 

For $w=e$ this is clear. Suppose it is proven for $\ell(w)<N$. For $\ell(w)=N$ let $s$ be a simple reflection such that $\ell(w)=\ell(ws)+1$. Over $k$, by Lemma \ref{l:conv pure}, $\IC(\dw)_{\cL}$ is a direct summand of $\IC(\dw\ds^{-1})_{s\cL}\star\IC(\ds)_{\cL}$. When the situation is over $\FF_{q}$, although we do not know a priori that $\IC(\dw)_{\cL}$ is a direct summand of $\IC(\dw\ds^{-1})_{s\cL}\star\IC(\ds)_{\cL}$ over $\FF_{q}$, its stalks are subquotients of stalks of $\IC(\dw\ds^{-1})_{s\cL}\star\IC(\ds)_{\cL}$ as $\Fr$-modules (as the perverse Leray spectral sequence for $\IC(\dw\ds^{-1})_{s\cL}\star\IC(\ds)_{\cL}$ degenerates at $E_{2}$ by the decomposition theorem). Therefore it suffices to show that the stalks of $\IC(\dw\ds^{-1})_{s\cL}\star\IC(\ds)_{\cL}$ lie in $\j{\Qlbar\j{n}\ot V; n\equiv \ell(w)-\ell(v) \mod 2, V\in\Fr\lmod_{0}}$.

By inductive hypothesis for $ws$, we have
\begin{equation*}
\IC(\dw\ds^{-1})_{s\cL}\in\j{\D(\dv)_{s\cL}\j{n}\ot V; v\in {}_{w\cL}W_{s\cL}, n\equiv \ell(w)-\ell(v)\mod2, V\in\Fr\lmod_{0}}.
\end{equation*}
Therefore
\begin{equation*}
\IC(\dw\ds^{-1})_{s\cL}\star\IC(\ds)_{\cL}\in\j{\D(\dv)_{s\cL}\star\IC(\ds)_{\cL}\j{n}\ot V, v\in {}_{w\cL}W_{s\cL}, n\equiv \ell(w)-\ell(v)\mod2, V\in\Fr\lmod_{0}}.
\end{equation*}
We will show that the stalks of $\D(\dv)_{s\cL}\star\IC(\ds)_{\cL}$ are either zero or of the form $\Qlbar\j{\ell(vs)}\ot V$ for some one-dimensional $V\in \Fr\lmod_{0}$, which would finish the induction step.

If $s\notin \WL$, by  Lemma \ref{l:clean s equiv}(3), we have $\D(\dv)_{s\cL}\star\IC(\ds)_{\cL}\cong \D(\dv\ds)_{\cL}$ and obviously satisfies the desired stalk property.

If $s\in \WL$, by Lemma \ref{l:inert s}, we have $\D(\dv)_{\cL}\star\IC(s)_{\cL}\cong \pi_{s}^{*}\pi_{s*}\D(\dv)_{\cL}\j{1}$. We use the notation from \S\ref{ss:inert s}. We first consider the case where $\ell(vs)>\ell(v)$. In this case, $BvB/B$ maps isomorphically to $B\ov vP_{s}/P_{s}$, therefore $\pi_{s*}\D(\dv)_{\cL}\cong \D(\dv)_{\wt\cL}\j{-1}$, whose nonzero stalks are of the form $\Qlbar\j{\ell(v)}\ot V$ for one-dimensional $V\in\Fr\lmod_{0}$. Therefore, the nonzero stalks of $\pi_{s}^{*}\pi_{s*}\D(\dv)_{\cL}\j{1}$ are of the form $\Qlbar\j{\ell(v)+1}\ot V=\Qlbar\j{\ell(vs)}\ot V$ for one-dimensional $V\in\Fr\lmod_{0}$. 

Finally the case $\ell(vs)<\ell(v)$. We have $\D(\dv)_{\cL}\star\IC(\ds)_{\cL}\cong \D(\dv\ds^{-1})_{\cL}\star\D(\ds)_{\cL}\star\IC(s)_{\cL}$. A calculation inside of $L_{s}$ gives $\D(\ds)_{\cL}\star\IC(s)_{\cL}\cong \pi_{s}^{*}\pi_{s*}\D(\ds)_{\cL}\j{1}\cong \IC(\ds)_{\cL}\j{-1}$. Therefore $\D(\dv)_{\cL}\star\IC(s)_{\cL}\cong \D(\dv\ds^{-1})_{\cL}\star\IC(\ds)_{\cL}\j{-1}$. We are back to the previous case (applied to $vs$ in place of $v$) to conclude that the nonzero stalks of $\D(\dv\ds^{-1})_{\cL}\star\IC(\ds)_{\cL}\j{-1}$ are of the form $\Qlbar\j{\ell(v)}\j{-1}\ot V=\Qlbar\j{\ell(vs)}\ot V$ for one-dimensional $V\in\Fr\lmod_{0}$.  This completes the induction step for proving \eqref{stalk purity}.

(2) By (1), $i_{v}^{*}\uIC(w)_{\cL}$ and $i_{v}^{!}\uIC(w)_{\cL}$ are successive extensions of $\uC(v)_{\cL}[n]$ for $n\equiv\ell(w)-\ell(v)\mod2$. However, there are no nontrivial extensions between $\uC(v)_{\cL}$ and $\uC(v)_{\cL}[2m]$ ($m\in\ZZ$) in ${}_{v\cL}\un\cD(v)_{\cL}\cong D^{b}_{\G(v)_{k}}(\pt_{k})$,  because $\upH^{\textup{odd}}_{\G(v)_{k}}(\pt_{k})=0$. Therefore $i_{v}^{*}\uIC(w)_{\cL}$ and $i_{v}^{!}\uIC(w)_{\cL}$ are direct sums of $\uC(v)_{\cL}[n]$ for $n\equiv\ell(w)-\ell(v)\mod2$.
\end{proof}
The above proposition will be strengthened in Proposition \ref{p:stalk ss} to include Frobenius semisimplicity of stalks and costalks of $\IC(\dw)_{\cL}$.

\begin{cor}\label{c:filM} Let $\cF, \cG\in {}_{\cL'}\cD_{\cL}$ be semisimple complexes. Then $\Homb(\cF,\cG):=\oplus_{n\in\ZZ}\Hom(\cF,\cG[n])$ admits an increasing filtration $F_{\le v}$ by $\Fr$-submodules indexed by $v\in {}_{\cL'}W_{\cL}$ (with its partial order inherited from $W$) such that the associated graded
\begin{equation}\label{GrFw}
\Gr^{F}_{v}\Homb(\cF,\cG)\cong \Homb(i^{*}_{v}\cF, i^{!}_{v}\cG).
\end{equation}
Moreover, this filtration is functorial in $\cF$ and $\cG$.
\end{cor}
\begin{proof}
The Schubert stratification gives a filtration on $\cG$ with $\Gr_{v}\cG\cong i_{v*}i_{v}^{!}\cG$. We get a corresponding filtered complex structure on $\RHom(\cF, \cG)$ with associated graded pieces quasi-isomorphic to $\RHom(i^{*}_{v}\cF, i^{!}_{v}\cG)$. We need to show that the spectral sequence converging to $\Homb(\cF, \cG)$ corresponding to this filtration degenerates at $E_{1}$. For this, it suffices to work in the non-mixed category ${}_{\cL'}\un\cD_{\cL}$, and we may assume $\cF=\uIC(w)_{\cL}$ and $\cG=\uIC(w')_{\cL}$. 

Proposition \ref{p:parity purity}(2) implies that $i^{*}_{v}\cF$ is a direct sum of $\uC(v)_{\cL}[n]$ for $n\equiv \ell(w)-\ell(v)\mod2$ and $i^{!}_{v}\un\cG$ is a direct sum of $\uC(v)_{\cL}[n]$ for $n\equiv \ell(w')-\ell(v)\mod2$. Therefore  $\RHom(i^{*}_{v}\cF, i^{!}_{v}\cG)$ is isomorphic to a direct sum of even shifts of $\RHom(\uC(v)_{\cL}, \uC(v)_{\cL})[\ell(w')-\ell(w)]\cong\upH^{*}_{\G(v)_{k}}(\pt_{k})[\ell(w')-\ell(w)]$. Therefore $\Homb(i^{*}_{v}\cF, i^{!}_{v}\cG)$  is concentrated in degrees of a fixed parity independent of $v$, hence the degeneration at $E_{1}$.
\end{proof}

\subsection{Monodromic version of the Hecke algebra}\label{ss:Hk}
Recall the monodromic version of the Hecke algebra for $W$ with monodromy in $\fo$ defined in \cite[1.4]{L-ConjCell}.  Let $\bH_{\fo}$ be the unital associative $\ZZ[v,v^{-1}]$-algebra with generators $T_w (w\in W)$ and  $1_\cL (\cL\in\fo)$ and relations
\begin{eqnarray}
\notag&&1_{\cL}1_{\cL'}=\d_{\cL,\cL'}1_{\cL},\quad\textup{ for }\cL,\cL'\in\fo;\\
\label{lengths add}&&T_wT_{w'}=T_{ww'},\quad\text{ if }w,w'\in W\text{ and } \ell(ww')=\ell(w)+\ell(w');\\
\notag&&T_w1_{\cL}=1_{w\cL}T_w, \quad\text{ for }w\in W,\cL\in\fo;\\
\label{quad rel}&&T_s^2=v^2T_1+(v^2-1)\sum_{\cL;s\in \WL}T_s1_{\cL},\text{  for simple reflections }s\in W;\\
\notag&&T_1=1=\sum_{\cL\in\fo}1_{\cL}.
\end{eqnarray}
The algebra $\bH_{\fo}$ is closely related to the algebra introduced by Yokonuma \cite{Yo}, as explained in \cite[35.3, 35.4]{L-CSDGVII}. Note that $\{T_w1_{\cL};(w,\cL)\in W\times\fo\}$ is a $\ZZ[v,v^{-1}]$-basis of $\bH_{\fo}$. For $w\in W$ we set $\wt{T}_w=v^{-\ell(w)}T_w\in \bH_{\fo}$.  There is a unique involution  $\bar{}:\bH_{\fo}\to\bH_{\fo}$
defined by $\ov{v^mT_w1_{\cL}}=v^{-m}T_{w^{-1}}^{-1}1_{\cL}$ for any $(w,\cL)\in W\times\fo$ and any $m\in\ZZ$. 

For any $(w,\cL)\in W\times\fo$ there is a unique element $c_{w,\cL}\in \bH_{\fo}$ such that
\begin{itemize}
\item $\ov{c_{w,\cL}}=c_{w,\cL}$.
\item $c_{w,\cL}=\sum_{y\in W}p_{y,\cL;w,\cL}\wt{T}_y1_\cL$ where $p_{y,\cL;w,\cL}\in v^{-1}\ZZ[v^{-1}]$ if $y\ne w$, and $p_{w,\cL;w,\cL}=1$.
\end{itemize}
The elements $\{c_{w,\cL}\}_{(w,\cL)\in W\times\fo}$ form a $\ZZ[v,v^{-1}]$-basis of $\bH_{\fo}$ called the {\em canonical basis}. This is analogous to the basis $\{C_{w}\}$ introduced in \cite[Theorem 1.1]{KL}.

Let $\cD_{\fo}=\oplus_{\cL',\cL\in\fo}({}_{\cL'}\cD_{\cL})$. The Grothendieck group $K_{0}(\cD_{\fo})$ is a $\ZZ[v,v^{-1}]$-module where the action of $v$ is given by $\j{-1}$. As in \cite[2.9]{L-center-nonunip}, there is a unique $\ZZ[v,v^{-1}]$-linear map
\begin{equation*}
\g: K_{0}(\cD_{\fo})\to\bH_{\fo}
\end{equation*}
sending $\cF\in {}_{\cL'}\cD_{\cL}$ to the element $\sum_{w\in {}_{\cL'}W_{\cL}}A_{w,\cF}( v)T_{w}1_{\cL}$ where $A_{w,\cF}(v)\in \ZZ[v,v^{-1}]$ is the virtual Poincar\'e polynomial of the stalk $i^{*}_{\dw}\cF$, i.e., $A_{w,\cF}(v)=\sum_{i,j\in\ZZ}(-1)^{i}(\dim\Gr^{W}_{j}\upH^{i}i^{*}_{\dw}\cF)(-v)^{j}$. 

By construction, $\g(\D(\dw)_{\cL})=\wt T_{w}1_{\cL}$. Under $\g$, $\IC(\dw)_{\cL}$ is sent to $c_{w,\cL}$ for $(w,\cL)\in W\times\fo$. Indeed, the purity property proved in Proposition \ref{p:parity purity} gives the degree bound for $p_{y,\cL;w,\cL}$ needed to characterize $c_{w,\cL}$.

By \cite[2.10]{L-center-nonunip}, $\g$ is a ring homomorphism.

\section{Blocks}
In this section we give a decomposition of ${}_{\cL'}\cD_{\cL}$ into a direct sum of full triangulated subcategories called blocks. First we need some preparation on Weyl groups.

\subsection{Blocks in ${}_{\cL'}\un W_{\cL}$}\label{ss:block W}
Let $\cL,\cL'\in\fo$. Denote
\begin{equation*}
{}_{\cL'}\un W_{\cL}={}_{\cL'}W_{\cL}/\WL=W^{\c}_{\cL'}\bs {}_{\cL'}W_{\cL}. 
\end{equation*}
Each element $\b\in{}_{\cL'}\un W_{\cL}$ is called a {\em block} of ${}_{\cL'}\un W_{\cL}$, and it inherits a partial order restricted from the Bruhat order in $W$.

Let $\cL,\cL'$ and $\cL''\in \fo$. Let $\b\in {}_{\cL'}\un W_{\cL}$ and $\g\in {}_{\cL''}\un W_{\cL'}$. Then  the set $\g\cdot \b:=\{w_{1}w_{2}|w_{1}\in \g, w_{2}\in \b\}$ is equal to $W^{\c}_{\cL''}w_{1}w_{2}=w_{1}W_{\cL'}^{\c}w_{2}=w_{1}w_{2}\WL$ (for any $w_{1}\in \g, w_{2}\in \b$), which defines an element in ${}_{\cL''}\un W_{\cL}$. This defines a map
\begin{equation*}
(-)\cdot(-): {}_{\cL''}\un W_{\cL'}\times {}_{\cL'}\un W_{\cL}\to {}_{\cL''}\un W_{\cL}
\end{equation*}
which is associative in the obvious sense. 

\begin{lemma}\label{l:max} Each block $\b\in{}_{\cL'}\un W_{\cL}$ contains a unique minimal element $w^{\b}$ and a unique maximal element $w_{\b}$ under its partial order. The minimal element $w^{\b}$ (resp. maximal element $w_{\b}$) is characterized by the property that $w^{\b}(\Phi^{+}_{\cL})\subset \Phi^{+}$ (resp. $w_{\b}(\Phi^{+}_{\cL})\subset \Phi^{-}$).
\end{lemma}
\begin{proof}
In \cite[Lemma 1.9(i)]{L-book} it is shown that each $\b$ contains a unique minimal length (hence minimal) element $w^{\b}$ characterized by the stated property. Let $\le_{\WL}$ be the Bruhat order on $\WL$ induced by the positive roots $\Phi^{+}_{\cL}$, see \S\ref{ss:rs}. By \cite[Lemma 1.9(ii)]{L-book}, if $v\le_{\WL}v'$, then $w^{\b}v\le w^{\b}v'$. Therefore, if we write $w_{\cL,0}$ for the longest (and maximal) element in $\WL$, $w^{\b}w_{\cL,0}$ is the unique maximal element in $\b$. Clearly it is characterized by the stated property.
\end{proof}

\begin{cor}\label{c:trans min}
For  $\b\in {}_{\cL'}\un W_{\cL}$ and $\g\in {}_{\cL''}\un W_{\cL'}$, we have the equalities in $W$
\begin{equation*}
w^{\g}w^{\b}=w^{\g\b}, \quad w^{\g}w_{\b}=w_{\g\b}, \quad w_{\g}w^{\b}=w_{\g\b}.
\end{equation*}
\end{cor}
\begin{proof}
Let us prove the first equality and the proof of the rest is similar. To show $w^{\g}w^{\b}$ is the minimal element in the block $\g\b$, by the criterion in Lemma \ref{l:max}, it suffices to show that $w^{\g}w^{\b}(\Phi^{+}_{\cL})\subset \Phi^{+}$. Since $w^{\b}$ is minimal in $\b$, $w^{\b}(\Phi^{+}_{\cL})\subset \Phi^{+}\cap \Phi_{\cL'}=\Phi^{+}_{\cL'}$ (we are using \eqref{RwL}), and indeed equality holds. Then by the same argument $w^{\g}w^{\b}(\Phi^{+}_{\cL})\subset w^{\g}(\Phi^{+}_{\cL'})\subset \Phi^{+}$, which shows that $w^{\g}w^{\b}$ is minimal in the block $\g\b$.
\end{proof}

\begin{cor}\label{c:conj wb} Let $\b\in {}_{\cL'}\un W_{\cL}$. Then $w\mapsto w^{\b}ww^{\b,-1}$ gives an isomorphism of {\em Coxeter groups} $W^{\c}_{\cL}\isom W^{\c}_{\cL'}$.
\end{cor}
\begin{proof}
Since $w^{\b}\Phi^{+}_{\cL}\subset \Phi^{+}\cap \Phi_{\cL'}=\Phi^{+}_{\cL'}$,  $w^{\b}$ sends simple roots in $\Phi_{\cL}$ to simple roots in $\Phi_{\cL'}$, hence conjugation by $w^{\b}$ sends simple reflections in $\WL$ to simple reflections in $W^{\c}_{\cL'}$.
\end{proof}

\subsection{The groupoid $\Xi$}\label{ss:Xi}  Let $\Xi$ be the groupoid whose object set is $\fo$, and the morphism set ${}_{\cL'}\Xi_{\cL}:=\Hom_{\Xi}(\cL,\cL')=\{w^{\b}|\b\in{}_{\cL'}\un W_{\cL}\}$. Clearly ${}_{\cL'}\Xi_{\cL}$ is in bijection with ${}_{\cL'}\un W_{\cL}$, and we often make the identification ${}_{\cL'}\Xi_{\cL}\isom {}_{\cL'}\un W_{\cL}$. The composition map is defined by the multiplication in $W$, since $w^{\g}w^{\b}=w^{\g\b}$ by Corollary \ref{c:trans min}.

Let $\b\in{}_{\cL'}\un W_{\cL}$. For $w\in \b$, there is a unique $v\in \WL$ such that $w=w^{\b}v$. Define
\begin{equation}\label{def ell beta}
\ell_{\b}(w)=\ell_{\cL}(v),
\end{equation}
where $\ell_{\cL}$ is the length function of the Coxeter group $\WL$, as defined in \eqref{length cL}.

The following lemma is a slight generalization of \cite[Lemma 5.3]{L-CS1}.

\begin{lemma}\label{l:ell beta} Let $\b\in{}_{\cL'}\un W_{\cL}$ and $w\in \b$. 
\begin{enumerate}
\item\label{ell beta neg} $\ell_{\b}(w)=\#\{\a\in \Phi^{+}_{\cL}|w\a<0\}$. 

\item\label{ell beta tran} For $\g\in {}_{\cL''}\un W_{\cL'}$, we have
\begin{equation*}
\ell_{\g\b}(w^{\g}w)=\ell_{\b}(w).
\end{equation*}

\item\label{ell beta ineq} Write $w$ into a product of simple reflections in $W$ (not necessarily reduced) $w=s_{i_{N}}\cdots s_{i_{2}}s_{i_{1}}$. Let  $\cL_{0}:=\cL$ and $\cL_{j}=s_{i_{j}}\cdots s_{i_{1}}(\cL)$ for $j\ge1$. Then 
\begin{equation}\label{ell beta}
\ell_{\b}(w)\le\#\{1\le j\le N|s_{i_{j}}\in W^{\c}_{\cL_{j-1}}\}.
\end{equation}

\item\label{ell beta eq} If in \eqref{ell beta ineq} $s_{i_{N}}\cdots s_{i_{2}}s_{i_{1}}$ is reduced, then equality in \eqref{ell beta} holds.
\end{enumerate}
\end{lemma}
\begin{proof}
\eqref{ell beta neg} Write $w=w^{\b}v$ for $v\in \WL$. Since $w^{\b}$ sends $\Phi^{+}_{\cL}$ (resp. $\Phi^{-}_{\cL}$) to positive (resp. negative) roots, for $\a\in \Phi^{+}_{\cL}$, $w^{\b}v\a<0$ if and only if $v\a<0$. Therefore, $\#\{\a\in \Phi^{+}_{\cL}|w\a<0\}=\#\{\a\in \Phi^{+}_{\cL}|v\a<0\}=\ell_{\cL}(v)=\ell_{\b}(w)$.

\eqref{ell beta tran} Write $w=w^{\b}v$ for $v\in \WL$.  Then $w^{\g}w=(w^{\g}w^{\b})v$. By Corollary \ref{c:trans min}, $w^{\g}w^{\b}=w^{\g\b}$, we have $\ell_{\g\b}(w^{\g}w)=\ell_{\cL}(v)=\ell_{\b}(w)$.

\eqref{ell beta ineq} Let $\un w$ be the sequence of simple reflections $(s_{i_{N}},\cdots, s_{i_{2}},s_{i_{1}})$. Denote the right side of \eqref{ell beta} by $L(\un w)$. We argue by induction on the length $N$ of the sequence $\un w$. 

For $N=0$ the statement is clear. Suppose the statement is proved for all $\un w$ of length $<N$. Let $s=s_{i_{N}}$ and 
$\un w'=(s_{i_{N-1}}, \cdots, s_{i_{2}},s_{i_{1}})$, $w'=s_{i_{N-1}}\cdots s_{i_{2}}s_{i_{1}}$ so that $w=sw'$.  Note that $\cL_{N-1}=w'\cL=s\cL'$. Let $\b'\in {}_{s\cL'}\un W_{\cL}$ be the block containing $w'$. By inductive hypothesis we have $\ell_{\b'}(w')\le L(\un w')$. We have two cases:

Case 1: $s\notin W_{s\cL'}^{\c}$. In this case $L(\un w)=L(\un w')$. On the other hand, $s$ is minimal in its block $\g\in {}_{\cL'}\un W_{s\cL'}$. By part \eqref{ell beta tran}, $\ell_{\b}(w)=\ell_{\b}(sw')=\ell_{\b'}(w')$. Therefore $\ell_{\b}(w)=\ell_{\b'}(w')\le L(\un w')=L(\un w)$.

Case 2: $s\in W_{s\cL'}^{\c}$. In this case, $L(\un w)=L(\un w')+1$. On the other hand, we have $w\Phi^{+}_{\cL}\cap \Phi^{-}=sw'\Phi^{+}_{\cL}\cap \Phi^{-}=s(w'\Phi^{+}_{\cL}\cap s\Phi^{-})$. Since the only difference between $\Phi^{-}$ and $s\Phi^{-}$ is that $-\a_{s}$ has been changed to $\a_{s}$, $\#(w'\Phi^{+}_{\cL}\cap s\Phi^{-})\le \#(w'\Phi^{+}_{\cL}\cap \Phi^{-})+1$. By part (1), $\ell_{\b}(w)=\#(w\Phi^{+}_{\cL}\cap \Phi^{-})=\#(w'\Phi^{+}_{\cL}\cap s\Phi^{-})\le \#(w'\Phi^{+}_{\cL}\cap \Phi^{-})+1=\ell_{\b'}(w')+1$. Therefore $\ell_{\b}(w)\le \ell_{\b'}(w')+1\le L(\un w')+1=L(\un w)$.

\eqref{ell beta eq} To prove the equality in the case $\un w$ is a reduced word, one uses the same inductive argument. The only point that needs modification is in Case 2. Since $\ell(w)=\ell(w')+1$ in this case, we have $w'^{-1}\a_{s}\in \Phi^{+}$, or $\a_{s}\in w'\Phi^{+}$. But since $s\in W_{w'\cL}^{\c}$, we also have $\a_{s}\in \Phi_{w'\cL}=w'\Phi_{\cL}$, therefore $\a_{s}\in w'\Phi_{\cL}\cap w'\Phi^{+}\cap s\Phi^{-}=w'\Phi^{+}_{\cL}\cap s\Phi^{-}$. Hence $w\Phi^{+}_{\cL}\cap \Phi^{-}=s(w'\Phi^{+}_{\cL}\cap s\Phi^{-})=s(w'\Phi^{+}_{\cL}\cap \Phi^{-})\sqcup\{-\a_{s}\}$. By part \eqref{ell beta neg}, $\ell_{\b}(w)=\ell_{\b'}(w')+1$. Therefore $\ell_{\b}(w)=\ell_{\b'}(w')+1=L(\un w')+1=L(\un w)$.
\end{proof}

\subsection{Partial order on a block} For a block $\b\in {}_{\cL'}\un W_{\cL}$ we define a partial order $\le_{\b}$ on elements of $\b$ as follows. Every element in $\b$ can be written uniquely as $w^{\b}w$ for some $w\in\WL$. Then we define $w^{\b}w'\le_{\b} w^{\b}w$ if and only if $w'\le_{\WL}w$ (under the Bruhat order of $\WL$). For $w',w\in W^{\c}_{\cL'}$, $w'w^{\b}\le_{\b}ww^{\b}$ if and only if $w'\le_{W^{\c}_{\cL'}}w$ (using Corollary \ref{c:conj wb}).

Later we will need the following result comparing the partial order $\le_{\b}$ with the partial order restricted from the Bruhat order of $W$. 

\begin{lemma}\label{l:order} Let $\b\in {}_{\cL'}\un W_{\cL}$.  
\begin{enumerate}
\item If $\g\in {}_{\cL''}\un W_{\cL'}$, left multiplication by $w^{\g}$ gives an isomorphism of posets $(\b, \le_{\b})\isom(\g\b, \le_{\g\b})$.
\item If $\d\in {}_{\cL}\un W_{\cL''}$, right multiplication by $w^{\d}$ gives an isomorphism of posets $(\b, \le_{\b})\isom(\b\d, \le_{\b\d})$. 
\item For $w,w'\in\b$, if $w'\le_{\b}w$, then $w'\le w$.
\end{enumerate}
\end{lemma}
\begin{proof}
(1) follows directly from the definition and Corollary \ref{c:trans min}.

(2) It suffices to show that $x\le_{\b}y$ implies $xw^{\d}\le_{\b\d}yw^{\d}$, for the reverse implication can be obtained by inverting $\d$. Left multiplying by $w^{\b,-1}$ we reduce to the case $\b=\WL$, the neutral block. Let $x,y\in \WL$ and $x\le_{\WL}y$, and we show $xw^{\d}\le_{\d}yw^{\d}$. Let $x'=w^{\d,-1}xw^{\d}$, $y'=w^{\d,-1}yw^{\d}\in W^{\c}_{\cL''}$. By Corollary \ref{c:conj wb}, $x\le_{\WL}y$ implies $x'\le_{W^{\c}_{\cL''}}y'$, hence $w^{\d}x'\le_{\d}w^{\d}y'$ by definition. Therefore  $xw^{\d}=w^{\d}x'\le_{\d}w^{\d}y'=yw^{\d}$.

(3) Induction on $\ell(w)$. The statement is clear for $w=e$. Assume the statement is true for all $w$ with $\ell(w)<N$ (for varying $\b$). Now suppose $\ell(w)=N$. Write $w=w_{1}s$ for some simple reflection $s$ such that $\ell(w)=\ell(w_{1})+1$. 

If $s\notin\WL$, write $\b'=\b s$ and $w'=w'_{1}s$. Then $w'_{1},w_{1}\in \b'$, and $w'_{1}\le_{\b'}w_{1}$ since right multiplication by $s$ is an isomorphism of posets $\b'\isom \b$ by (2). Applying the inductive hypothesis to $w_{1}$, we get $w'_{1}\le w_{1}$. Hence $w'=w'_{1}s\le \max\{w_{1},w_{1}s\}=w$. 

If $s\in\WL$, then $s$ is a simple reflection in $\WL$. Since $w'\le_{\b}w_{1}s$, either $w'\le_{\b}w_{1}$, or $w'=w'_{1}s$ and $w'_{1}\le_{\b}w_{1}$. In the former case, applying the inductive hypothesis to $w_{1}$, we see $w'\le w_{1}\le w$. In the latter case, applying the inductive hypothesis to $w_{1}$, we get $w'_{1}\le w_{1}$, hence $w'=w'_{1}s\le \max\{w_{1},w_{1}s\}=w$.

\end{proof}

\begin{remark}\label{r:order} The converse of Lemma \ref{l:order}(3) is not true in general. In particular, if $w,w'\in \WL$,  then $w'\le w$ does not necessarily imply $w'\le_{\WL}w$.
\end{remark}

\begin{defn}\label{d:block}
\begin{enumerate}
\item For each $\b\in {}_{\cL'}\un W_{\cL}$, let ${}_{\cL'}\un\cD_{\cL}^{\b}$ be the full triangulated subcategory of ${}_{\cL'}\un\cD_{\cL}$ generated by $\{\uD(w)_{\cL}\}_{w\in \b}$. Let ${}_{\cL'}\cD_{\cL}^{\b}\subset {}_{\cL'}\cD_{\cL}$ be the preimage of ${}_{\cL'}\un\cD_{\cL}^{\b}$ under $\om$. We call ${}_{\cL'}\cD_{\cL}^{\b}$ (resp. ${}_{\cL}\un\cD^{\b}_{\cL}$) a {\em block} of ${}_{\cL'}\cD_{\cL}$ (resp. ${}_{\cL}\un\cD_{\cL}$). 
\item When $\b$ is the unit coset $\WL$, we denote the block ${}_{\cL}\cD^{\b}_{\cL}$ (resp. ${}_{\cL}\un\cD^{\b}_{\cL}$) by ${}_{\cL}\cD_{\cL}^{\c}$ (resp. ${}_{\cL}\un\cD^{\c}_{\cL}$), and call it the {\em neutral block} .
\end{enumerate}
\end{defn}

The terminology ``block'' is justified by the next Proposition.

\begin{prop}[Block decomposition]\label{p:block} We have direct sum decompositions of the triangulated categories
\begin{equation}\label{block decomp}
{}_{\cL'}\cD_{\cL}=\bigoplus_{\b\in {}_{\cL'}\un W_{\cL}}{}_{\cL'}\cD_{\cL}^{\b}, \quad {}_{\cL'}\un\cD_{\cL}=\bigoplus_{\b\in {}_{\cL'}\un W_{\cL}}{}_{\cL'}\un\cD_{\cL}^{\b}.
\end{equation}
\end{prop}
\begin{proof}
We prove the non-mixed statement and the mixed version follows. Clearly the subcategories $\{{}_{\cL'}\un\cD_{\cL}^{\b}\}_{\b\in{}_{\cL'}\un W_{\cL}}$ generate ${}_{\cL'}\un\cD_{\cL}$. It remains to show that if $w_{1}, w_{2}\in {}_{\cL'}W_{\cL}$ are not in the same right $\WL$ coset, then
\begin{equation}\label{Hom van}
\bR\Hom(\uD(w_{1})_{\cL},\uD(w_{2})_{\cL})=0.
\end{equation}
We prove \eqref{Hom van} by induction on $\ell(w_{2})$.  For $\ell(w_{2})=0$, i.e., $w_{2}=e$, by adjunction,
$\bR\Hom(\uD(w_{1})_{\cL},\uD(e)_{\cL})\cong \bR\Hom(\uC(w_{1}), i^{!}_{w_{1}}\uD(e)_{\cL})$ which vanishes whenever $w_{1}\ne e$. This verifies \eqref{Hom van} for $\ell(w_{2})=0$.

Suppose \eqref{Hom van} is proved for $\ell(w_{2})<n$ ($n>0$). Consider the case $\ell(w_{2})=n$ and $w_{1}\in {}_{\cL'}W_{\cL}- w_{2}\WL$. Let $s$ be a simple reflection such that $\ell(w_{2})=\ell(w_{2}s)+1$. By Lemma \ref{l:invertible}, we have $\uD(w_{2})_{\cL}\star\unb(s)_{s\cL}\cong \uD(w_{2}s)_{s\cL}$. Since $\star\unb(s)_{s\cL}$ is an equivalence, we have
\begin{eqnarray*}
\bR\Hom(\uD(w_{1})_{\cL},\uD(w_{2})_{\cL})&\cong& \bR\Hom(\uD(w_{1})_{\cL}\star \unb(s)_{s\cL},\uD(w_{2})_{\cL}\star\unb(s)_{s\cL})\\
&\cong&\bR\Hom(\uD(w_{1})_{\cL}\star \unb(s)_{s\cL},\uD(w_{2}s)_{\cL}).
\end{eqnarray*}
If either $\ell(w_{1})=\ell(w_{1}s)+1$, or $s\notin \WL$, then by either Lemma \ref{l:invertible} or Lemma \ref{l:clean s equiv}(1) we similarly have $\uD(w_{1})_{\cL}\star \unb(s)_{s\cL}\cong \uD(w_{1}s)_{s\cL}$, hence $\RHom(\uD(w_{1})_{\cL}\star \unb(s)_{s\cL},\uD(w_{2}s)_{\cL})=\bR\Hom(\uD(w_{1}s)_{s\cL},\uD(w_{2}s)_{s\cL})$ which vanishes by inductive hypothesis since $\ell(w_{2}s)<n$. 

It remains to treat the case $s\in \WL$ and $\ell(w_{1})=\ell(w_{1}s)-1$. Since $\unb(s)_{\cL}$ is in the triangulated subcategory generated by $\uD(s)_{\cL}$ and $\uD(e)_{\cL}$,  $\uD(w_{1})_{\cL}\star \unb(s)_{\cL}$ is in the triangulated subcategory generated by $\uD(w_{1})_{\cL}\star \uD(s)_{\cL}\cong \uD(w_{1}s)_{\cL}$ and $\uD(w_{1})_{\cL}\star\uD(e)_{\cL}=\uD(w_{1})_{\cL}$, and we are done again by inductive hypothesis applied to $w_{2}s$.
\end{proof}

\begin{cor} Let $\b\in {}_{\cL'}\un W_{\cL}$ and $w\in \b$. Then $\nb(\dw)_{\cL}$ and $\IC(\dw)_{\cL}\in {}_{\cL'}\cD_{\cL}^{\b}$. In particular, ${}_{\cL'}\un\cD_{\cL}^{\b}$ is also the full triangulated subcategory of ${}_{\cL'}\un\cD_{\cL}$ generated  either by the collection $\{\uIC(w)_{\cL}\}_{w\in \b}$ or by the collection  $\{\unb(w)_{\cL}\}_{w\in \b}$.
\end{cor}
\begin{proof}
Since $\unb(w)_{\cL}$ and $\uIC(w)_{\cL}$ and indecomposable objects and they admit nonzero maps from $\uD(w)_{\cL}$, they must lie in the same summand as $\uD(w)_{\cL}$ in the decomposition \eqref{block decomp} for ${}_{\cL'}\un\cD_{\cL}$.
\end{proof}

\begin{prop}[Convolution preserves blocks]\label{p:conv block} Let $\cL,\cL'$ and $\cL''\in \fo$. Let $\b\in {}_{\cL'}\un W_{\cL}$ and $\g\in {}_{\cL''}\un W_{\cL'}$. Then
\begin{equation*}
{}_{\cL''}\cD_{\cL'}^{\g}\star{}_{\cL'}\cD_{\cL}^{\b}\subset {}_{\cL''}\cD_{\cL}^{\g\cdot\b}.
\end{equation*}
\end{prop}
\begin{proof}
It suffices to show the same statement for the non-mixed categories. By definition, it suffices to show that for any $w_{1}$ in a block $\b\in {}_{\cL'}\un W_{\cL}$ and any other block $\g\in {}_{\cL''}\un W_{\cL'}$, 
\begin{equation}\label{conv D}
{}_{\cL''}\cD_{\cL'}^{\g}\star\uD(w_{1})_{\cL}\subset {}_{\cL''}\un\cD_{\cL}^{\g\cdot \b}.
\end{equation}
We prove this by induction on $\ell(w_{1})$. When $\ell(w_{1})=0$, $w_{1}=e$, the statement is clear since $(-)\star\uD(e)_{\cL}$ is the identity functor.

Next we consider the case $\ell(w_{1})=1$, i.e., $w_{1}=s$ is a simple reflection. If $s\notin \WL$, then by Lemma \ref{l:clean s equiv}(3), for any $w_{2}\in \g$, $\uD(w_{2})_{s\cL}\star\uD(s)_{\cL}\cong \uD(w_{2}s)_{\cL}$, which implies \eqref{conv D}. 

If $s\in \WL$, it suffices to show that
\begin{equation}\label{Dw2s}
\uD(w_{2})_{\cL}\star \uD(s)_{\cL}\in {}_{\cL''}\un\cD_{\cL}^{\g}
\end{equation}
(now $w_{2}$ and $w_{2}s$ are in the same block denoted $\g$).  If $\ell(w_{2}s)=\ell(w_{2})+1$, then by Lemma \ref{l:basic conv}, $\uD(w_{2})_{\cL}\star \uD(s)_{\cL}\cong \uD(w_{2}s)_{\cL}$ , which verifies \eqref{Dw2s}.  If  $\ell(w_{2}s)=\ell(w_{2})-1$, then by \ref{l:basic conv} we have $\uD(w_{2})_{\cL}\star \uD(s)_{\cL}\cong \uD(w_{2}s)_{\cL}\star\uD(s)_{\cL}\star\uD(s)_{\cL}$. Since $\uD(s)_{\cL}\star \uD(s)_{\cL}\in{}_{\cL}\un\cD(\le s)_{\cL}$ which is generated by $\uD(s)_{\cL}$ and $\uD(e)_{\cL}$, we have $\uD(w_{2})_{\cL}\star \uD(s)_{\cL}\in \j{\uD(w_{2}s)_{\cL}, \uD(w_{2}s)_{\cL}\star\uD(s)_{\cL})}=\j{\uD(w_{2}s)_{\cL},\uD(w_{2})_{\cL}}$, which verifies \eqref{Dw2s} in this case. This completes the proof when $\ell(w_{1})=1$.

Now consider the case $\ell(w_{1})\ge2$. Write $w_{1}=w'_{1}s$ where $s$ is a simple reflection in $W$ and $\ell(w_{1})=\ell(w_{1}')+1$. Then $\uD(w_{1})_{\cL}\cong \uD(w'_{1})_{s\cL}\star \uD(s)_{\cL}$. By inductive hypothesis applied to $\ell(w'_{1})$, we have ${}_{\cL''}\cD_{\cL'}^{\g}\star\uD(w'_{1})_{s\cL}\subset {}_{\cL''}\un\cD_{s\cL}^{\g\cdot \b'}$, where $\b'\in {}_{\cL'}\un W_{s\cL}$ is the block containing $w'_{1}$. By the proven case for simple reflections, ${}_{\cL''}\un\cD_{s\cL}^{\g\cdot \b'}\star\uD(s)_{\cL}\subset{}_{\cL''}\un\cD_{\cL}^{\g\cdot \b}$. Combining these two facts we get \eqref{conv D} for $w_{1}$.
\end{proof}

We will also need the following statement about stalks of $\uIC(w)_{\cL}$ later.
\begin{lemma}\label{l:stalk order}
Let $\b\in {}_{\cL'}\un W_{\cL}$ and $w\in \b$. Then $i_{v}^{*}\uIC(w)_{\cL}$ and $i^{!}_{v}\uIC(w)_{\cL}$ vanish unless $v\in \b$ and $v\le_{\b}w$.
\end{lemma}
\begin{proof} It is enough to prove the stalk statement and the costalk statement follows by Verdier duality.

Induction on $\ell(w)$. The statement is clear for $w=e$. Suppose it is proved for $\ell(w)<N$ (for varying $\b$), and we now prove it for $\ell(w)=N$. Write $w=w's$ for some simple reflection $s$ such that $\ell(w)=\ell(w')+1$. 

If $s\notin\WL$, then $\uIC(w)_{\cL}\cong \uIC(w')_{s\cL}\star\uD(s)_{\cL}$ by Lemma \ref{l:clean s equiv}(3). Let $\b'=\b s$. Applying inductive hypothesis to $\uIC(w')_{s\cL}$, we see that $\uIC(w')_{s\cL}\in \j{\uD(v)_{s\cL}[n]; v\le_{\b'}w',n\in\ZZ}$. Therefore, $\uIC(w')_{s\cL}\star  \uD(s)_{\cL}\in\j{\uD(v)_{s\cL}\star\uD(s)_{\cL}[n]; v\le_{\b'}w',n\in\ZZ}$. Note that $\uD(v)_{s\cL}\star\uD(s)_{\cL}=\uD(vs)_{\cL}$, and $v\le_{\b'}w'$ implies $vs\le_{\b}w's=w$, we see that  $\uIC(w)_{\cL}\in \j{\uD(v)_{\cL}[n]; v\le_{\b}w,n\in\ZZ}$, hence $i^{*}_{v}\uIC(w)_{\cL}$ is zero unless $v\le_{\b}w$.

If $s\in\WL$, then $\uIC(w)_{\cL}$ is a direct summand of $\uIC(w')_{\cL}\star\uIC(s)_{\cL}$, and we shall prove the stalk statement for the latter. Applying inductive hypothesis to $\uIC(w')_{\cL}$, we see that $\uIC(w')_{s\cL}\in \j{\uD(v)_{s\cL}[n]; v\le_{\b'}w',n\in\ZZ}$.  Therefore, $\uIC(w')_{\cL}\star  \uD(s)_{\cL}\in\j{\uD(v)_{\cL}\star\uIC(s)_{\cL}[n]; v\le_{\b'}w',n\in\ZZ}$. Since $\uD(v)_{\cL}\star\uIC(s)_{\cL}$ only has stalks along $G_{v}$ and $G_{vs}$, we see that $i^{*}_{v}(\uIC(w')_{\cL}\star\uIC(s)_{\cL})$ is nonzero only if either $v\le_{\b}w'$, or $vs\le_{\b}w'$. In the former case, $v\le_{\b}w'\le_{\b}w$; the latter, $vs\le_{\b}\max_{\b}\{w',w's\}=w$ (using that $s$ is a simple reflection in $\WL$). In either case,  $i^{*}_{v}(\uIC(w')_{\cL}\star\uIC(s)_{\cL})$ is zero unless $v\le_{\b}w$. This completes the induction step.
\end{proof}

\section{Minimal IC sheaves}\label{s:min}
In this section we study the simple perverse sheaves with minimal support in each block, and use them to prove categorical equivalences among different blocks.

\subsection{Minimal IC sheaves}\label{ss:min}
For $\b\in{}_{\cL'}\un W_{\cL}$, any object  $\xi\in {}_{\cL'}\cD^{\b}_{\cL}$ is called a {\em minimal IC sheaf} if $\om\xi\cong \uIC(w^{\b})_{\cL}$. We denote by ${}_{\cL'}\fP^{\b}_{\cL}$ the groupoid of minimal IC sheaves in ${}_{\cL'}\cD^{\b}_{\cL}$. The automorphism group of objects in ${}_{\cL'}\fP^{\b}_{\cL}$ are $\Qlbar^{\times}$.

\begin{prop}\label{p:min th} Let $\b\in{}_{\cL'}\un W_{\cL}$ and $\dw^{\b}$ be a lifting of $w^{\b}$.
\begin{enumerate}
\item The natural maps $\D(\dw^{\b})_{\cL}\to \IC(\dw^{\b})\to \nb(\dw^{\b})_{\cL}$ are isomorphisms.
\item Let $\cL''\in \fo$ and $\g\in{}_{\cL''}\un W_{\cL'}$. Then  the functor
\begin{equation*}
(-)\star\IC(\dw^{\b}): {}_{\cL''}\cD_{\cL'}^{\g}\to {}_{\cL''}\cD_{\cL}^{\g\b}
\end{equation*}
is an equivalence with inverse $(-)\star\IC(\dw^{\b,-1})$.  Similar statement is true for left convolution with $\IC(\dw^{\b})$.
\item The equivalence $(-)\star\IC(\dw^{\b})$ sends $\D(\dw)_{\cL}, \nb(\dw)_{\cL}$ and $\IC(\dw)_{\cL}$ to $\D(\dw\dw^{\b})_{\cL}, \nb(\dw\dw^{\b})_{\cL}$ and $\IC(\dw\dw^{\b})_{\cL}$, for all $w\in W$.
\end{enumerate}
\end{prop}
\begin{proof}
We prove all the statements by induction on $\ell(w^{\b})$.  For $\ell(w^{\b})=0$ the statements are clear. Suppose the statements are true for $\ell(w^{\b})<n$. Let $\b$ be such that $\ell(w^{\b})=n$. Write $w^{\b}=w's$ for some simple reflection $s$ such that $\ell(w')=n-1$. We have $s\notin \WL$ for otherwise $w'\in \b$ and it is shorter than $w$. Let $\b'\in {}_{\cL'}\un W_{s\cL}$ be the block containing $w'$. We must have $w'=w^{\b'}$ for otherwise $\ell(w^{\b'}s)\le \ell(w^{\b'})+1\le \ell(w')$ and $w^{\b'}s\in \b$ would be shorter than $w^{\b}$. Hence $w^{\b}=w^{\b'}s$. 

For part (1), it suffices to show its non-mixed version. By Lemma \ref{l:clean s equiv}(3), $\uIC(w^{\b'})_{s\cL}\star\uIC(s)_{\cL}\cong\uIC(w^{\b'}s)_{\cL}\cong \uIC(w^{\b})_{\cL}$. By inductive hypothesis, $\uD(w^{\b'})_{s\cL}\isom \unb(w^{\b'})_{s\cL}$. By Lemma \ref{l:clean s equiv}(1), $\uD(s)_{\cL}\isom \uIC(s)_{\cL}\isom\unb(s)_{\cL}$. Hence the natural map $\uD(w^{\b})_{\cL}\to\uIC(w^{\b})_{\cL}$ can be factorized into isomorphisms $\uD(w^{\b})_{\cL}=\uD(w^{\b'}s)_{\cL}\cong \uD(w^{\b'})_{s\cL}\star\uD(s)_{\cL}\cong\uIC(w^{\b'})_{s\cL}\star\uIC(s)_{\cL}\cong\uIC(w^{\b'}s)_{\cL}=\uIC(w^{\b})_{\cL}$. By Verdier duality the natural map $\uIC(w^{\b})_{\cL}\to \unb(w^{\b})_{\cL}$ is also an isomorphism. This proves part (1) for $\uIC(w^{\b})_{\cL}$.

Part (2) follows from (1) together with Lemma \ref{l:invertible}.

Finally we show part (3).  By inductive hypothesis, $\D(\dw)_{\cL'}\star\IC(\dw^{\b'})_{s\cL}\cong \D(\dw\dw^{\b'})_{s\cL}$. Therefore $\D(\dw)_{\cL'}\star\IC(\dw^{\b'}\ds)_{\cL}\cong \D(\dw)_{\cL'}\star\IC(\dw^{\b'})_{s\cL}\star\IC(\ds)_{\cL}\cong \D(\dw\dw^{\b'})_{s\cL}\star\IC(\ds)_{\cL}\cong \D(\dw\dw^{\b'}\ds)_{\cL}$, where we use Lemma \ref{l:clean s equiv}(3). Write $\dw^{\b}=\dw^{\b'}\ds t$ for $t\in T(\FF_{q})$.  Then by \eqref{tensor Lt}, $\IC(\dw^{\b})_{\cL}=\IC(\dw^{\b'}\ds)_{\cL}\ot \cL_{t}$,  and $\D(\dw\dw^{\b})_{\cL}=\D(\dw\dw^{\b'}\ds)_{\cL}\ot \cL_{t}$. Therefore $\D(\dw)_{\cL'}\star\IC(\dw^{\b'}\ds)_{\cL}\cong \D(\dw\dw^{\b'}\ds)_{\cL}$ implies $\D(\dw)_{\cL'}\star\IC(\dw^{\b})_{\cL}\cong \D(\dw\dw^{\b})_{\cL}$. The argument for $\nb$ and $\IC$ are similar.
\end{proof}

We may strengthen statement (3) in the above proposition to canonical isomorphisms. To do this, we first need a lemma. The rest of this section is only used in \S\ref{s:all blocks}.

\begin{lemma}\label{l:aff fiber} Let $\dw, \dw'\in N_{G}(T)$ be any liftings of $w, w'\in W$ respectively. Let  $m_{w,w'}: G_{w}\twtimes{B}G_{w'}\to G$ be the multiplication map. Let $B^{-}$ be the Borel subgroup of $G$ such that $B\cap B^{-}=T$, and let $U^{-}$ be the unipotent radical of $B^{-}$. We denote $\Ad(\dw)U$ by ${}^{w}U$. 
\begin{enumerate}
\item The following map is an isomorphism
\begin{eqnarray*}
U^{-}\cap {}^{w^{-1}}U\cap {}^{w'}U&\isom & m_{w,w'}^{-1}(\dw\dw')\\
u &\mapsto& (\dw u, u^{-1}\dw'). 
\end{eqnarray*}
\item We have $\dim (U^{-}\cap {}^{w^{-1}}U\cap {}^{w'}U)=\frac{1}{2}(\ell(w)+\ell(w')-\ell(ww'))$. In particular, $m_{w,w'}^{-1}(\dw\dw')$   is isomorphic to an affine space  of  dimension $\frac{1}{2}(\ell(w)+\ell(w')-\ell(ww'))$.
\end{enumerate}
\end{lemma}
\begin{proof}
(1) By Bruhat decomposition, any $g\in G_{w}$ can be written uniquely as $\dw ub$ where $u\in {}^{w^{-1}}U\cap U^{-}$ and $b\in B$; any $g'\in G_{w'}$ can be written uniquely as $b'u'\dw'$ where $b'\in B$ and $u'\in {}^{w'}U\cap U^{-}$. Using these facts we have an isomorphism
\begin{eqnarray}\label{ph 3 parts}
\ph: ({}^{w^{-1}}U\cap U^{-})\times B\times ({}^{w'}U\cap U^{-}) & \isom & G_{w}\twtimes{B}G_{w'}\\
\notag (u,b,u')& \mapsto & (\dw ub, u'\dw').
\end{eqnarray}
We write a  point $(g,g')\in m_{w,w'}^{-1}(\dw\dw')$ as $\ph(u,b,u')$ as above, then $gg'=\dw\dw'$ implies $ubu'=1$, or $b=u^{-1}u'^{-1}$. Since $b\in B$ and $u^{-1}u'^{-1}\in U^{-}$, we must have $b=1$ and $u'=u^{-1}$,  and the latter implies $u\in U^{-}\cap {}^{w^{-1}}U\cap {}^{w'}U$. Therefore, restricting $\ph$ to triples $(u,b,u')$ where $b=1$ and $u'=u^{-1}$ gives an isomorphism
\begin{eqnarray*}
U^{-}\cap {}^{w^{-1}}U\cap {}^{w'}U &\isom & m_{w,w'}^{-1}(\dw\dw')\\
u &\mapsto & \ph(u,1,u^{-1})=(\dw u,u^{-1}\dw').
\end{eqnarray*}

(2) Since $\dim(U^{-}\cap {}^{w^{-1}}U\cap {}^{w'}U)=\#(\Phi^-\cap w^{-1}\Phi^+\cap w'\Phi^+)$, and $\ell(w)=\#(\Phi^{-}\cap w\Phi^{+})$ for all $w\in W$, the dimension formula is equivalent to
\begin{equation}\label{b}
2\#(\Phi^-\cap w^{-1}\Phi^+\cap w'\Phi^+)=\#(\Phi^-\cap w^{-1}\Phi^+)+\#(\Phi^-\cap w'\Phi^+)-\#(\Phi^-\cap ww'\Phi^+).
\end{equation}
We have 
\begin{eqnarray*}
\#(\Phi^-\cap w^{-1}\Phi^+)&=&\#(\Phi^-\cap w^{-1}\Phi^+\cap w'\Phi^+)+\#(\Phi^-\cap w^{-1}\Phi^+\cap w'\Phi^-),\\
\#(\Phi^-\cap w'\Phi^+)&=&\#(\Phi^-\cap w^{-1}\Phi^+\cap w'\Phi^+)+\#(\Phi^-\cap w^{-1}\Phi^-\cap w'\Phi^+),\\
\#(\Phi^-\cap ww'\Phi^+)&=&\#(w^{-1}\Phi^-\cap w'\Phi^+).
\end{eqnarray*}
Thus to prove \eqref{b} it is enough to prove
\begin{equation}\label{c}
\#(\Phi^-\cap w^{-1}\Phi^+\cap w'\Phi^-)+\#(\Phi^-\cap w^{-1}\Phi^-\cap w'\Phi^+)=\#(w^{-1}\Phi^-\cap w'\Phi^+).
\end{equation}
By the change of variable $\a\mapsto-\a$, we see that
\begin{equation*}
\#(\Phi^-\cap w^{-1}\Phi^+\cap w'\Phi^-)=\#(\Phi^+\cap w^{-1}\Phi^-\cap w'\Phi^+),
\end{equation*}
so that \eqref{c} is equivalent to
\begin{equation*}
\#(\Phi^+\cap w^{-1}\Phi^-\cap w'\Phi^+)+\#(\Phi^-\cap w^{-1}\Phi^-\cap w'\Phi^+)=\#(w^{-1}\Phi^-\cap w'\Phi^+),
\end{equation*}
which is obvious. 
\end{proof}

The next result will not be used in the rest of the paper.

\begin{cor} Let $\fB$ be the flag variety of $G$, and $\fO_{w}\subset \fB\times\fB$ the $G$-orbit containing $(1,\dw)$, $w\in W$. Let $w_1,w_2,w_3$ be elements of $W$ such that $w_1w_2w_3=1$. Let 
\begin{equation*}
A_{w_{1},w_{2},w_{3}}=\{(B_1,B_2,B_3)\in\fB^{3}|(B_1,B_2)\in\fO_{w_1},(B_2,B_3)\in\fO_{w_2},(B_3,B_1)\in\fO_{w_3}\}.
\end{equation*}
Then $A_{w_{1},w_{2},w_{3}}$ is a single $G$-orbit under the diagonal $G$-action on $\fB^{3}$, and $\dim(A_{w_{1},w_{2},w_{3}})=\dim\fB+(\ell(w_1)+\ell(w_2)+\ell(w_3))/2$.
\end{cor}
\begin{proof} Since $G$ acts transitively (by simultaneous conjugation) on $\fO_{w_3}$, it is enough to show that
for fixed $(B_3,B_1)\in\fO_{w_3}$, the conjugation action of $B_1\cap B_3$ on
$A':=\{B_2\in\fB|(B_1,B_2)\in\fO_{w_1},(B_2,B_3)\in\fO_{w_2}\}$ 
is transitive and that 
$\dim(A')=\dim\fB+(\ell(w_1)+\ell(w_2)+\ell(w_3))/2-\dim\fO_{w_{3}}=(\ell(w_1)+\ell(w_2)-\ell(w_3))/2.$ We may assume that $B_1=B,B_3={}^{w_{3}^{-1}}B={}^{w_{1}w_{2}}B$. Then $A'=\{gB\in G/B|g\in G_{w_{1}}, g^{-1}\dw_{1}\dw_{2}\in G_{w_{2}}\}$, and  it can be identified with the fiber $m_{w_{1},w_{2}}^{-1}(\dw_{1}\dw_{2})$ considered in Lemma \ref{l:aff fiber}: $gB\in A'$ corresponds to $(g,g^{-1}\dw_{1}\dw_{2})\in m_{w_{1},w_{2}}^{-1}(\dw_{1}\dw_{2})$. By Lemma \ref{l:aff fiber}(1), the action of $U\cap {}^{w_{1}w_{2}}U$ on $m_{w_{1},w_{2}}^{-1}(\dw_{1}\dw_{2})$ by $u\cdot (g,g^{-1}\dw_{1}\dw_{2})=(ug, g^{-1}u^{-1}\dw_{1}\dw_{2})$ is already transitive, therefore the action of $B\cap {}^{w_{1}w_{2}}B$ on $A'$ by left translation on $gB$ is also transitive. The dimension formula follows from Lemma \ref{l:aff fiber}(2).
\end{proof}

\begin{cons}\label{con:can} Let $\b\in{}_{\cL'}\un W_{\cL}$ and $w\in W$. We will construct {\em canonical} isomorphisms
\begin{eqnarray}
\label{can D}\D(\dw)_{\cL'}\star\IC(\dw^{\b})_{\cL}\cong \D(\dw\dw^{\b})_{\cL}, \\
\label{can nb}\nb(\dw)_{\cL'}\star\IC(\dw^{\b})_{\cL}\cong \nb(\dw\dw^{\b})_{\cL}, \\
\label{can IC}\IC(\dw)_{\cL'}\star\IC(\dw^{\b})_{\cL}\cong \IC(\dw\dw^{\b})_{\cL}.
\end{eqnarray}
There are similar canonical isomorphisms for left convolution with $\IC(\dw^{\b})_{\cL}$.

By Proposition \ref{p:min th}(3) we know that the two sides of the above equations are indeed isomorphic, and such isomorphisms are unique up to a scalar (for the endomorphisms of $\D(\dw\dw^{\b})_{\cL},\nb(\dw\dw^{\b})_{\cL}$ and $\IC(\dw\dw^{\b})_{\cL}$ are scalars). 

We first construct the  canonical isomorphism \eqref{can D}. For this it suffices to construct a canonical isomorphism between the stalks of the two sides at $\dw\dw^{\b}$. By  the definition of convolution, we have
\begin{equation*}
i_{\dw\dw^{\b}}^{*}(\D(\dw)_{\cL'}\star\IC(\dw^{\b})_{\cL})\cong \cohoc{*}{m_{w,w^{\b}}^{-1}(\dw\dw^{\b})_{k},C(\dw)_{\cL'}\stackrel{B}{\boxtimes}C(\dw^{\b})_{\cL}|_{m_{w,w^{\b}}^{-1}(\dw\dw^{\b})}}.
\end{equation*}
Here $C(\dw)_{\cL'}\stackrel{B}{\boxtimes}C(\dw^{\b})_{\cL}$ is the descent of $C(\dw)_{\cL'}\boxtimes C(\dw^{\b})_{\cL}$ to $G_{w}\twtimes{B}G_{w^{\b}}$, and $m_{w,w^{\b}}: G_{w}\twtimes{B}G_{w^{\b}}\to G$ is the multiplication map. Using Lemma \ref{l:aff fiber}(1), we may identify $m_{w,w^{\b}}^{-1}(\dw\dw^{\b})$ with the unipotent group $U^{-}\cap {}^{w^{-1}}U\cap {}^{w^{\b}}U$, under which the restriction of $C(\dw)_{\cL'}\stackrel{B}{\boxtimes}C(\dw^{\b})_{\cL}$ is canonically isomorphic to the constant sheaf $\Qlbar\j{\ell(w)+\ell(w^{\b})}$ since the stalk of $C(\dw)_{\cL'}$ at $\dw$ and the stalk of $C(\dw^{\b})_{\cL}$ at $\dw^{\b}$ are canonically isomorphic to $\Qlbar\j{\ell(w)}$ and $\Qlbar\j{\ell(w^{\b})}$ respectively by construction. Therefore we have a canonical isomorphism of $\Fr$-modules
\begin{eqnarray*}
&& i_{\dw\dw^{\b}}^{*}(\D(\dw)_{\cL'}\star\IC(\dw^{\b})_{\cL})\cong \cohoc{*}{U^{-}_{k}\cap {}^{w^{-1}}U_{k}\cap {}^{w^{\b}}U_{k}, \Qlbar\j{\ell(w)+\ell(w^{\b})}}\\
&\cong& \Qlbar\j{\ell(w)+\ell(w^{\b})}\j{-\ell(w)-\ell(w^{\b})+\ell(ww^{\b})}\\
&=&\Qlbar\j{\ell(ww^{\b})}\cong i_{\dw\dw^{\b}}^{*}\D(\dw\dw^{\b})_{\cL},
\end{eqnarray*}
where we used the dimension formula for $U^{-}\cap {}^{w^{-1}}U\cap {}^{w^{\b}}U$ proved in Lemma \ref{l:aff fiber}(2). We define the canonical isomorphism \eqref{can D} to be the one which restricts to the above isomorphism after taking stalks at $\dw\dw^{\b}$.

To construct the canonical isomorphism \eqref{can IC}, we consider the following diagram
\begin{equation*}
\xymatrix{ \D(\dw)_{\cL'}\star\IC(\dw^{\b})_{\cL}\ar[r]^-{\eqref{can D}}\ar[d] & \D(\dw\dw^{\b})_{\cL}\ar[d]\\
\IC(\dw)_{\cL'}\star\IC(\dw^{\b})_{\cL}\ar@{-->}[r]^-{\l} & \IC(\dw\dw^{\b})_{\cL}
}
\end{equation*} 
where the vertical maps are induced from the canonical maps $\D(\dw)_{\cL}\to \IC(\dw)_{\cL}$, and the upper horizontal map is the one constructed just now. Since $\Hom(\D(\dw)_{\cL'}\star\IC(\dw^{\b})_{\cL}, \IC(\dw\dw^{\b})_{\cL})$ is one-dimensional, an arbitrary choice of the isomorphism $\l$ (dashed arrow) would make the diagram commutative up to a nonzero scalar. Hence there is a unique choice of the isomorphism $\l$ making the above diagram commutative. This constructs the desired map \eqref{can IC}. The construction of \eqref{can nb} is similar.  
\end{cons}

\begin{warning}\label{ss:warning} For two blocks $\b\in{}_{\cL'}\un W_{\cL}$ and $\g\in{}_{\cL''}\un W_{\cL'}$, Construction \ref{con:can} gives a canonical isomorphism
\begin{equation}\label{can gb}
\can_{\dw^{\g},\dw^{\b}}: \IC(\dw^{\g})_{\cL'}\star\IC(\dw^{\b})_{\cL}\cong \IC(\dw^{\g}\dw^{\b})_{\cL}.
\end{equation}
Let $\d\in {}_{\cL'''}\un W_{\cL''}$ be yet another block. We have two isomorphism between $\IC(\dw^{\d})_{\cL''}\star \IC(\dw^{\g})_{\cL'}\star\IC(\dw^{\b})_{\cL}$ and $\IC(\dw^{\d}\dw^{\g}\dw^{\b})_{\cL}$ given by first doing convolution $\IC(\dw^{\g})_{\cL'}\star\IC(\dw^{\b})_{\cL}$ or doing $\IC(\dw^{\d})_{\cL''}\star\IC(\dw^{\g})_{\cL'}$:
\begin{equation}\label{two asso}
\xymatrix{\IC(\dw^{\d})_{\cL''}\star \IC(\dw^{\g})_{\cL'}\star\IC(\dw^{\b})_{\cL}\ar@<0.5ex>[rrrr]^-{\can_{\dw^{\d}\dw^{\g},\dw^{\b}}\c(\can_{\dw^{\d},\dw^{\g}}\star\id)}\ar@<-0.5ex>[rrrr]_-{\can_{\dw^{\d}, \dw^{\g}\dw^{\b}}\c(\id\star\can_{\dw^{\g},\dw^{\b}})} 
 &&&& \IC(\dw^{\d}\dw^{\g}\dw^{\b})_{\cL}}.
\end{equation}
However, these two maps are not equal in general, as we will see from the following example.
\end{warning}

\begin{exam}\label{ex:-1} Consider the case $G=\SL_{2}$,  and $\cL\in \Ch(T)$ nontrivial.  Let $\ds=\mat{}{1}{-1}{}$ be a lifting of the nontrivial element $s\in W$, and $s\cL=\cL^{-1}$. In this case, both $\IC(\ds)_{\cL}$ and $\IC(\dot e)_{\cL}=\d_{\cL}$ are minimal IC sheaves. We claim that the two isomorphisms between $\IC(\ds)_{\cL}\star\IC(\ds^{-1})_{s\cL}\star\IC(\ds)_{\cL}$ and $\IC(\ds)_{\cL}$ given as in \eqref{two asso} differ by a sign. 

Indeed, the stalk of $\cF=\IC(\ds)_{\cL}\star\IC(\ds^{-1})_{\cL}\star\IC(\ds)_{\cL}$ at $\ds$ can be calculated  from the definition of the convolution as follows. We identify $G/U$ with $\AA^{2}-\{0\}$, where $U$ is the stabilizer of $e_{1}=(1,0)$. The fiber of the three-fold convolution morphism $G\twtimes{U}G\twtimes{U}G\to G$ over $\ds$ can be identified with pairs of vectors $(v_{1}, v_{2})\in (\AA^{2}-\{0\})^{2}$ via the map $(g_{1},g_{2},g_{3})\mapsto (g_{1}e_{1}, g_{1}g_{2}e_{1})$. The open subset $Y=\{v_{1}=(x_{1},y_{1})\in \AA^{2}-\{0\}, v_{2}=(x_{2}, y_{2})\in \AA^{2}-\{0\}| y_{1}\ne0, x_{2}\ne0, x_{1}y_{2}-x_{2}y_{1}\ne0\}$ of $(\AA^{2}-\{0\})^{2}$ is relevant to our calculation. For any invertible function $f$ on $Y$ we use $\cL_{f}$ to denote the pullback $f^{*}\cL$. We consider the local system $\cK=\cL_{-y_{1}}\cL^{-1}_{x_{1}y_{2}-x_{2}y_{1}}\cL_{x_{2}}$. Let $\Gm\times\Gm$ act on $Y$ by scaling the vectors $(x_{1},y_{1})$ and $(x_{2}, y_{2})$ separately. Then $\cK$ is equivariant under the $\Gm^{2}$ action on $Y$, hence descends to a local system on $X=Y/\Gm^{2}$ which we still denote by $\cK$. We have a canonical isomorphism
\begin{equation*}
i_{\ds}^{*}\cF\cong\cohoc{*}{X,\cK}\j{3}.
\end{equation*}
Now $X\incl\AA^{2}$ by coordinates $u=x_{1}/y_{1}$ and $v=y_{2}/x_{2}$, and with image $\AA^{2}-\{uv=1\}$. The local system $\cK=\cL^{-1}_{1-uv}$ on $X$. Therefore we have canonically
\begin{equation*}
i_{\ds}^{*}\cF\cong\cohoc{*}{X,\cL^{-1}_{1-uv}}\j{3}.
\end{equation*}
The isomorphism $\can_{\dot e,\ds}\c(\can_{\ds,\ds^{-1}}\star\id)$ corresponds to the isomorphism by restriction to the line $v=0$
\begin{equation*}
i^{*}_{v=0}: \cohoc{*}{X,\cL^{-1}_{1-uv}}\j{3}\isom \cohoc{*}{\AA^{1}_{v=0}, \Qlbar}\j{3}\cong \Qlbar\j{1}.
\end{equation*}
Here we have used the canonical trivialization of the stalk of $\cL$ at $1$, and the fundamental class of $\AA^{1}$. Similarly, the other isomorphism $\can_{\ds, \dot e}\c(\id\star\can_{\ds^{-1},\ds})$ corresponds to the isomorphism by restriction to the line $u=0$. Let $\s:X\to X$ be the involution  $(u,v)\mapsto (v,u)$, then $\cL^{-1}_{1-uv}$ has a canonical $\s$-equivariant structure such that the $\s$-action on the stalk at $(0,0)$ is the identity. This induces an involution $\s^{*}$ on $\cohoc{*}{X,\cL^{-1}_{1-uv}}\j{3}$, and the following diagram is commutative
\begin{equation*}
\xymatrix{\cohoc{*}{X,\cL^{-1}_{1-uv}}\j{3} \ar[rr]^{\s^{*}}\ar[d]_{i^{*}_{v=0}}&& \cohoc{*}{X,\cL^{-1}_{1-uv}}\j{3}\ar[d]^{i^{*}_{u=0}}\\
\cohoc{*}{\AA^{1}_{v=0}, \Qlbar}\j{3}\ar@{=}[r] & \Qlbar\j{1}\ar@{=}[r] & \cohoc{*}{\AA^{1}_{u=0}, \Qlbar}}
\end{equation*}
We claim that $\s^{*}$ acts on the one-dimensional space $\cohoc{*}{X,\cL^{-1}_{1-uv}}$ by $-1$, which would imply our claim in the beginning of this example. 

We compare two traces $\Tr_{1}=\Tr(\Fr, \cohoc{*}{X,\cL^{-1}_{1-uv}})$ and $\Tr_{2}=\Tr(\s^{*}\c\Fr,\cohoc{*}{X,\cL^{-1}_{1-uv}})$. Let $\chi$ be the character of $\FF_{q}^{\times}$ corresponding to $\cL^{-1}$. By the Lefschetz trace formula, $\Tr_{1}=\sum_{u,v\in \FF_{q}, uv\ne1}\chi(1-uv)$. The fiber of $(u,v)\mapsto a=1-uv$ has $q-1$ elements over $a\ne1$, and has $2q-1$ elements over $a=1$. Therefore, $\Tr_{1}=(q-1)\sum_{a\ne0,1}\chi(a)+(2q-1)=q$ since $\chi\ne1$. On the other hand, $\s^{*}\c\Fr$ is the Frobenius for the variety $X'\subset \Res_{\FF_{q^{2}}/\FF_{q}}\AA^{1}-\{\Nm=1\}$ which becomes isomorphic to $X$ over $\FF_{q^{2}}$. Using this interpretation, we have $\Tr_{2}=\sum_{u\in \FF_{q^{2}}, \Nm(u)\ne1}\chi(1-\Nm(u))$. The fiber of the map $\FF_{q^{2}}\ni u\mapsto 1-\Nm(u)=a\in \FF_{q}$ has $q+1$ elements over $a\ne1$ and $1$ element over $a=1$. Therefore $\Tr_{2}=(q+1)\sum_{a\ne0,1}\chi(a)+1=-q$. This shows $\Tr_{1}=-\Tr_{2}$, hence $\s^{*}$ acts by $-1$ on the one-dimensional space $\cohoc{*}{X,\cL^{-1}_{1-uv}}$. 
\end{exam}

\subsection{The $3$-cocycle}\label{ss:3c} For three composable blocks $\b,\g,\d$ as in \S\ref{ss:warning}, let $\s(\dw^{\d},\dw^{\g},\dw^{\b})$ be the ratio of the two isomorphisms in \eqref{two asso} (top over bottom). It is easy to see that $\s(\dw^{\d},\dw^{\g},\dw^{\b})$ depends only on $\b,\g,\d$, so we denote it by $\s(w^{\d},w^{\g},w^{\b})$. Recall the groupoid $\Xi$ defined in \S\ref{ss:Xi}. By the pentagon axiom for the associativity of the convolution, the assignment $(w^{\d},w^{\g},w^{\b})\mapsto \s(w^{\d},w^{\g},w^{\b})$ defines a $3$-cocycle $\s\in Z^{3}(\Xi, \Qlbar^{\times})$. In other words, for four composable morphisms $w^{\e},w^{\d},w^{\g}$ and $w^{\b}$ in $\Xi$, 
\begin{equation*}
\s(w^{\d},w^{\g},w^{\b})\s(w^{\e\d}, w^{\g},w^{\b})^{-1}\s(w^{\e}, w^{\d\g},w^{\b})\s(w^{\e}, w^{\d},w^{\g\b})^{-1}\s(w^{\e}, w^{\d}, w^{\g})=1.
\end{equation*}

In \cite[\S4]{3c} it is shown that $\s$ always takes values in $\{\pm1\}$. In fact there is a $3$-cocycle $\e^{W}_{3}\in Z^{3}(W,\{\pm1\})$ canonically attached to the Coxeter group $(W,S)$, and $\s$ is the pullback of $\e^{W}_{3}$ along the natural map $\Xi\to [\pt/W]$. The cohomology class of $\s$ is also calculated in \cite[\S5]{3c},  and it often turns out to be nontrivial.

The rest of the section is devoted to the construction of a rigidification of the minimal objects.
\subsection{The Whittaker category}\label{ss:Wh}
Let $B^{-}$ be the Borel subgroup of $G$ containing $T$ and opposite to $B$; let $U^{-}$ be the unipotent radical of $B^{-}$. Fix a nontrivial additive character $\psi_{0}: \FF_{q}\to \Qlbar^{\times}$, which defines an Artin-Schreier character sheaf $\AS$ on the additive group $\Ga$. Using the pinning $\bx_{-\a_{s}}: U_{-\a_{s}}\cong \Ga$ that we fixed, consider the homomorphism
\begin{equation*}
\psi: U^{-}\to \prod U_{-\a_{s}}\xr{\prod \bx_{-\a_{s}}} \prod \Ga\xr{\textup{sum}}\Ga.
\end{equation*}
Here the products are over simple roots $\a_{s}$. Then $\psi^{*}\AS$ is a rank one character sheaf on $U^{-}$. For $\cL\in \Ch(T)$, consider the category
\begin{equation*}
{}_{\psi}\cM_{\cL}:=D^{b}_{(U^{-}\times T,\psi^{*}\AS\boxtimes \cL^{-1}),m}(G/U)
\end{equation*}
where $U^{-}$ acts on $G/U$ by left translation and $T$ by the inverse of right translation. It is well-known that all objects in ${}_{\psi}\cM_{\cL}$ vanish outside the open subset $U^{-}B/U$ of $G/U$, and taking the stalk at $e\in G$ gives an equivalence
\begin{equation*}
i^{*}_{e}: {}_{\psi}\cM_{\cL}\isom D^{b}_{m}(\pt).
\end{equation*}
See \cite[Lemma 4.2.1]{BY} (specialized to $W_{\Theta}=W$). Let ${}_{\psi}\cF_{\cL}\in {}_{\psi}\cM_{\cL}$ be the object corresponding to $\Qlbar\in D^{b}_{m}(\pt)$ under the above equivalence. Let ${}_{\psi}\un \cF_{\cL}$ be the image of ${}_{\psi}\cF_{\cL}$ in the unmixed version ${}_{\psi}\un\cM_{\cL}$.

For $\cL, \cL'\in \Ch(T)$, convolution gives a functor
\begin{equation*}
(-)\star(-): {}_{\psi}\cM_{\cL'}\times {}_{\cL'}\cD_{\cL}\to {}_{\psi}\cM_{\cL}.
\end{equation*}
For $\cL'=\cL$, this gives an action of ${}_{\cL}\cD_{\cL}$ on ${}_{\psi}\cM_{\cL}$. 

\begin{lemma}\label{l:action Wh} If $w\in {}_{\cL'}W_{\cL}$ then
\begin{equation*}
{}_{\psi}\un\cF_{\cL'}\star\un\IC(w)_{\cL}\cong\begin{cases}{}_{\psi}\un\cF_{\cL} & w=w^{\b} \mbox{ for some block }\b\in {}_{\cL'}\un W_{\cL};\\
0 & \mbox{otherwise.}\end{cases}
\end{equation*}
\end{lemma}
\begin{proof}
By induction on  $\ell(w)$, we reduce to two basic cases for $w=s$ a simple reflection (so $\cL'=s\cL$)
\begin{equation}\label{psiF s}
{}_{\psi}\un\cF_{\cL'}\star\un\IC(s)_{\cL}\cong\begin{cases}{}_{\psi}\un\cF_{\cL}, & s\notin W^{\c}_{\cL} ;\\
0, & s\in W^{\c}_{\cL}.\end{cases}
\end{equation}

If $s\notin W^{\c}_{\cL} $, then by $\un\IC(s)_{\cL}\star \un\IC(s)_{\cL'}\cong \un\d_{\cL'}$ we see that  $\star\un\IC(s)_{\cL'}: {}_{\psi}\un\cM_{\cL}\to {}_{\psi}\un\cM_{\cL'}$ gives an inverse to $\star\un\IC(s)_{\cL}: {}_{\psi}\un\cM_{\cL'}\to {}_{\psi}\un\cM_{\cL}$. In particular, $\un\IC(s)_{\cL}$ is an equivalence ${}_{\psi}\un\cM_{\cL'}\isom {}_{\psi}\un\cM_{\cL}$, hence ${}_{\psi}\un\cF_{\cL'}\star\un\IC(s)_{\cL}\cong {}_{\psi}\un\cF_{\cL}[n]$ for some $n\in\ZZ$. 

We normalize the perverse t-structure on ${}_{\psi}\un\cM_{\cL}$ so that ${}_{\psi}\un\cF_{\cL'}$ is in the heart. Using estimate of perverse degrees under affine maps, we see that ${}_{\psi}\un\cF_{\cL'}\star\un\D(s)_{\cL}$ lies in perverse degrees $\ge0$ while ${}_{\psi}\un\cF_{\cL'}\star\un\nb(s)_{\cL}$ lies in perverse degrees $\le0$. Since $\un\D(s)_{\cL}\cong \IC(s)_{\cL}\cong \un\nb(s)_{\cL}$, we see that ${}_{\psi}\un\cF_{\cL'}\star\un\IC(s)_{\cL}$ is also perverse, hence must be isomorphic to ${}_{\psi}\un\cF_{\cL}$. 

If $s\in W^{\c}_{\cL}$, consider the rank one character sheaf $\wt\cL$ on $L_{s}$ (see \S\ref{ss:inert s}) and the category ${}_{\psi}\un\cM_{\wt\cL}=D^{b}_{(U^{-}\times L_{s},\psi^{*}\AS\boxtimes \wt\cL^{-1})}((G/U^{s})_{k})
$. This category is zero for any object in ${}_{\psi}\un\cM_{\wt\cL}$  would pull back to an object in ${}_{\psi}\cM_{\cL}$ whose support is not contained in $U^{-}B$. By an analogue of Lemma \ref{l:inert s} for ${}_{\psi}\un\cM_{\cL}$ we have
\begin{equation}\label{psiF ICs}
{}_{\psi}\un\cF_{\cL}\star\un\IC(s)_{\cL}\cong \pi_{s}^{*}\pi_{s*}({}_{\psi}\un\cF_{\cL})[1]\in {}_{\psi}\un\cM_{\cL}
\end{equation}
where $\pi_{s*}: {}_{\psi}\un\cM_{\cL}\to {}_{\psi}\un\cM_{\wt\cL}$ and its left adjoint $\pi_{s}^{*}$ are defined using the construction in \S\ref{ss:fun}.  Since ${}_{\psi}\un\cM_{\wt\cL}=0$ we conclude by \eqref{psiF ICs} that ${}_{\psi}\un\cF_{\cL}\star\IC(s)_{\cL}=0$.  This proves \eqref{psiF s} and hence the lemma.
\end{proof}

\subsection{Rigidified minimal IC sheaves}\label{ss:rig min} 
Let $\b\in{}_{\cL'}\un W_{\cL}$.  A {\em rigidified minimal IC sheaf} in ${}_{\cL'}\fP_{\cL}^{\b}$ is a pair $(\xi, \e)$ where $\xi\in {}_{\cL'}\fP_{\cL}^{\b}$ (i.e., $\om\xi\cong \un\IC(w^{\b})_{\cL}$), and $\e$ is an isomorphism in ${}_{\psi}\cM_{\cL}$
\begin{equation*}
\e: {}_{\psi}\cF_{\cL'}\star \xi\cong {}_{\psi}\cF_{\cL}.
\end{equation*}
Rigidified minimal IC sheaves exist in all blocks. Indeed, starting with any $\xi_{0}\in {}_{\cL'}\fP_{\cL}^{\b}$, then by Lemma \ref{l:action Wh} we have an isomorphism $\e_{0}: {}_{\psi}\cF_{\cL'}\star \xi_{0}\cong {}_{\psi}\cF_{\cL}\ot V$ for some one-dimensional $\Fr$-module $V$. Then let $\xi=\xi_{0}\ot V^\vee$, and let $\e$ be the isomorphism ${}_{\psi}\cF_{\cL'}\star \xi={}_{\psi}\cF_{\cL'}\star (\xi_{0}\ot V^\vee)\cong {}_{\psi}\cF_{\cL}$ induced  from $\e_{0}$. Then $(\xi, \e)$ is  a rigidified minimal IC sheaf in ${}_{\cL'}\fP_{\cL}^{\b}$.
    
Let $(\xi_{1},\e_{1})$ and $(\xi_{2},\e_{2})$ be two rigidified minimal IC sheaves in ${}_{\cL'}\fP_{\cL}^{\b}$. Then by adjusting scalars, there is a unique isomorphism $\un\a: \om\xi_{1}\isom\om\xi_{2}$ such that $\e_{2}\c(\id_{{}_{\psi}\un\cF_{\cL'}}\star\un\a)=\e_{1}: {}_{\psi}\un\cF_{\cL'}\star \om\xi_{1}\to {}_{\psi}\un\cF_{\cL}$. The uniqueness of $\un\a$ implies that $\Fr(\un\a)=\un\a$; moreover, $\Hom({}_{\psi}\un\cF_{\cL'}\star \om\xi_{1}, {}_{\psi}\un\cF_{\cL}[-1])=0$, hence $\un\a$ uniquely lifts to an isomorphism $\a: \xi_{1}\isom \xi_{2}$  in ${}_{\cL'}\fP^{\b}_{\cL}$. Therefore any two rigidified minimal IC sheaves are isomorphic to each other. Moreover, the automorphism group of any rigidified minimal IC sheaf is trivial. Therefore we may identify all the rigidified minimal IC sheaves in ${}_{\cL}\fP^{\b}_{\cL}$ as a single object and denote it by $\IC(w^{\b})^{\da}_{\cL}$.

Note that $\IC(w^{\b})^{\da}_{\cL}$ depends on the definition of the Whittaker category ${}_{\psi}\cM_{\cL}$, hence ultimately depends on the choice of the isomorphisms $U_{-\a_{s}}\cong \Ga$ and the additive character $\psi_{0}$ on $\FF_{q}$. 

\begin{lemma}\label{l:lgb}
For  $\b\in {}_{\cL'}\un W_{\cL}$ and $\g\in{}_{\cL''}\un W_{\cL'}$ there is a canonical isomorphism
\begin{equation*}
\IC(w^{\g})^{\da}_{\cL'}\star\IC(w^{\b})^{\da}_{\cL}\isom\IC(w^{\g\b})^{\da}_{\cL}.
\end{equation*}
Moreover these isomorphisms are compatible with the associativity of the multiplication in the groupoid $\Xi$ and the associativity constraint of the convolution between $\{{}_{\cL'}\fP^{\b}_{\cL}\}$ for blocks $\b$. 
\end{lemma}
\begin{proof}
Let $\xi=\IC(w^{\g})^{\da}_{\cL'}\star\IC(w^{\b})^{\da}_{\cL}\in {}_{\cL''}\fP^{\g\b}_{\cL}$. Using the rigidifications of $\IC(w^{\g})^{\da}_{\cL'}$ and $\IC(w^{\b})^{\da}_{\cL}$, we have a canonical isomorphism
\begin{equation*}
\e: {}_{\psi}\cF_{\cL''}\star\xi={}_{\psi}\cF_{\cL''}\star \IC(w^{\g})^{\da}_{\cL'}\star\IC(w^{\b})^{\da}_{\cL}\cong {}_{\psi}\cF_{\cL'}\star \IC(w^{\b})^{\da}_{\cL}\cong {}_{\psi}\cF_{\cL}.
\end{equation*}
The pair $(\xi,\e)$ is a rigidified minimal IC sheaf in ${}_{\cL''}\fP^{\g\b}_{\cL}$. By the uniqueness of rigidified minimal IC sheaves (up to unique isomorphism), there is a canonical isomorphism $(\xi,\e)\cong \IC(w^{\g\b})^{\da}_{\cL}$. The compatibility with the associativity constraint follows similarly from the uniqueness of the isomorphism between rigidified minimal IC sheaves in a given block. 
\end{proof}

\section{Maximal IC sheaves}\label{s:max IC}

\subsection{Maximal IC sheaves} Let $\cL,\cL'\in\fo$ and $\b\in {}_{\cL'}\un W_{\cL}$. Recall that   $w_{\b}$ is the longest element in the block $\b$. An object $\cF\in {}_{\cL'}\cD^{\b}_{\cL}$ is called a {\em maximal IC sheaf} if $\om\cF\cong\uIC(w_{\b})_{\cL}$.

When $\cL,\cL'$ are trivial, there is only one block $\b$ in $D^{b}_{m}(B\bs G/B)$, $w_{\b}=w_{0}$ is the longest element in $W$ and $\uIC(w_{\b})_{\cL}\cong \Qlbar[\dim G/B]$ is a shifted constant sheaf on $B\bs G/B$. The constant sheaf $\Qlbar$ on $B\bs G/B$ has two remarkable properties: (a) convolution with it always yields a direct sum of constant sheaves; (b) its stalks and costalks are one-dimensional. Below we will prove analogs of these properties for maximal IC sheaves in each block.

\begin{prop}\label{p:conv Th} Let $\b\in {}_{\cL'}\un W_{\cL}$ and $\g\in {}_{\cL''}\un W_{\cL'}$. Let $w\in \b$.
\begin{enumerate}
\item For any $w\in \b$, the convolution $\uIC(w_{\g})_{\cL'}\star\uIC(w)_{\cL}$ is isomorphic to a direct sum of shifts of $\uIC(w_{\g\b})_{\cL}$.

\item The perverse cohomology $\pH^{i}(\uIC(w_{\g})_{\cL'}\star\uIC(w)_{\cL})$ vanishes unless $-\ell_{\b}(w)\le i\le \ell_{\b}(w)$. 

\item There are isomorphisms
\begin{eqnarray*}
\pH^{\ell_{\b}(w)}(\uIC(w_{\g})_{\cL'}\star\uIC(w)_{\cL})\cong \uIC(w_{\g\b})_{\cL},\\
\pH^{-\ell_{\b}(w)}(\uIC(w_{\g})_{\cL'}\star\uIC(w)_{\cL})\cong \uIC(w_{\g\b})_{\cL}.
\end{eqnarray*}
\end{enumerate}
\end{prop}
\begin{proof}
We prove the statements simultaneously by induction on $\ell(w)$. For $w=e$ the statement is clear. 

If $\ell(w)=1$,  $w$ is a simple reflection $s$. 

If $s\notin \WL$, Lemma \ref{l:clean s equiv}(3) implies that $\uIC(w_{\g})_{\cL'}\star\uIC(s)_{\cL}\cong \uIC(w_{\g}s)_{\cL}$. By Corollary  \ref{c:trans min},  $w_{\g}s=w_{\g\b}$ is the maximal element in $\g\b$, hence  $\uIC(w_{\g})_{\cL'}\star\uIC(s)_{\cL}\cong \uIC(w_{\g \b})_{\cL}$. Note that $\ell_{\b}(s)=0$ in this case, and (2)(3) hold trivially.

If $s\in \WL$ (hence $\cL'=s\cL=\cL$), then by Lemma \ref{l:IC descend}, $\uIC(w_{\g})_{\cL}\cong \pi_{s}^{*}\uIC(\ov w_{\g})_{\wt\cL}[1]$, here $\uIC(\ov w_{\g})_{\wt\cL}=\om \IC(\dw_{\g})_{\wt\cL}\in {}_{\cL}\un\cD_{\wt\cL}$. By Lemma \ref{l:inert s}, $\uIC(w_{\g})_{\cL}\star\uIC(s)_{\cL}\cong \pi_{s}^{*}\pi_{s*}\pi^{*}_{s}\uIC(\ov w_{\g})_{\wt\cL}[1]$. By the projection formula, $\pi_{s*}\pi^{*}_{s}\uIC(\ov w_{\g})_{\wt\cL}\cong \uIC(\ov w_{\g})_{\wt\cL}\ot\cohog{*}{\PP^{1}_{k}}$ because $\pi_{s}: G/B\to G/P_{s}$ is a $\PP^{1}$-fibration. Therefore
\begin{equation*}
\uIC(w_{\g})_{\cL}\star\uIC(s)_{\cL}\cong \pi^{*}_{s}\uIC(\ov w_{\g})_{\wt\cL}[1]\oplus \pi^{*}_{s}\uIC(\ov w_{\g})_{\wt\cL}[-1]=\uIC(w_{\g})_{\cL}[1]\oplus\uIC(w_{\g})_{\cL}[-1]
\end{equation*}
Note that $\ell_{\b}(s)=1$ in this case, and (2)(3) follows from the above isomorphism. This settles the case $\ell(w)=1$.

For $\ell(w)>1$, write $w=w's$ for some  simple reflection $s$ such that $\ell(w')=\ell(w)-1$. Let  $\b'=\b s\in {}_{\cL'}\un W_{s\cL}$ so $w'\in \b'$. We shall first prove the analogs of the statements (1)(2) for $\uIC(w')_{s\cL}\star\uIC(s)_{\cL}$ instead of $\uIC(w)_{\cL}$ (in statement (2), the range for $i$ is still $[-\ell_{\b}(w), \ell_{\b}(w)]$). 

By inductive hypothesis, $\uIC(w_{\g})_{\cL'}\star\uIC(w')_{s\cL}\cong \uIC(w_{\g\b'})_{s\cL}\ot V'$ for a graded $\Qlbar$-vector space $V'=\oplus_{n\in\ZZ}V'_{n}[-n]$ such that $V'_{n}=0$ unless $-\ell_{\b'}(w')\le n\le \ell_{\b'}(w')$ and $\dim V'_{\pm\ell_{\b'}(w')}=1$. 
Therefore
\begin{equation}\label{V'}
\uIC(w_{\g})_{\cL'}\star(\uIC(w')_{s\cL}\star\uIC(s)_{\cL})\cong \oplus_{n}V'_{n}[-n]\ot(\uIC(w_{\g\b'})_{s\cL}\star\uIC(s)_{\cL}).
\end{equation}

If $s\notin \WL$, we have $\ell_{\b}(w)=\ell_{\b'}(w')$ by the formula for $\ell_{\b}$ given in Lemma \ref{l:ell beta}\eqref{ell beta eq}.  We also have $\uIC(w_{\g\b'})_{s\cL}\star\uIC(s)_{\cL}\cong\uIC(w_{\g\b})$ by Lemma \ref{l:clean s equiv}(3). The statements (1)(2)(3) for $\uIC(w')_{s\cL}\star\uIC(s)_{\cL}$ follow easily from \eqref{V'}.

If $s\in \WL$, we have $\ell_{\b}(w)=\ell_{\b'}(w')+1$ by the formula for $\ell_{\b}$ given in Lemma \ref{l:ell beta}\eqref{ell beta eq}.  By the $w=s$ case already treated in the beginning of the proof, we have $\uIC(w_{\g\b'})_{s\cL}\star\uIC(s)_{\cL}\cong\uIC(w_{\g\b})_{\cL}[1]\oplus \uIC(w_{\g\b})_{\cL}[-1]$. Therefore $\uIC(w_{\g})_{\cL'}\star\uIC(w')_{s\cL}\star\uIC(s)_{\cL}\cong\oplus_{n}(V'_{n}[-n-1]\oplus V'_{n}[-n+1])\ot\uIC(w_{\g\b})_{\cL}\cong \oplus_{n\in\ZZ}(V'_{n-1}\oplus V'_{n+1})\ot \uIC(w_{\g\b})_{\cL}[-n]$. The statements (1)(2)(3) follow from the known properties of $V'$. 

Finally we deduce the statements (1)(2)(3) for $\uIC(w)_{\cL}$ from the proven statements for $\uIC(w')_{s\cL}\star\uIC(s)_{\cL}$. When $s\notin \WL$, we have $\uIC(w')_{s\cL}\star\uIC(s)_{\cL}\cong \uIC(w)_{\cL}$ by Lemma \ref{l:clean s equiv}(3), therefore the statements are already proven. Below we deal with the case $s\in \WL$.

By the decomposition theorem, $\uIC(w)_{\cL}$ is a direct summand of $\uIC(w')_{s\cL}\star\uIC(s)_{\cL}$. By Lemma \ref{l:ss} below, $\uIC(w')_{\cL}\star\uIC(s)_{\cL}$ is itself perverse, hence we can write 
\begin{equation*}
\uIC(w')_{\cL}\star\uIC(s)_{\cL}\cong \uIC(w)_{\cL}\oplus \cP
\end{equation*}
for some semisimple perverse sheaf $\cP\in {}_{\cL'}\un\cD_{\cL}$ with support in $U\bs G_{<w}/U$. We see that
\begin{equation*}
\uIC(w_{\g})_{\cL'}\star\uIC(w')_{\cL}\star\uIC(s)_{\cL}\cong \uIC(w_{\g})_{\cL'}\star\uIC(w)_{\cL}\oplus \uIC(w_{\g})_{\cL'}\star\cP.
\end{equation*}
By the proven statements (1)(2)(3) for the left side above, we have
\begin{equation}\label{ICw summand}
\uIC(w_{\g})_{\cL'}\star\uIC(w)_{\cL}\oplus \uIC(w_{\g})_{\cL'}\star\cP\cong \uIC(w_{\g\b})_{\cL}\ot(\oplus_{n}V_{n}[-n])
\end{equation}
where $V_{n}$ is a finite-dimensional $\Qlbar$-vector space, $V_{n}=0$ unless $-\ell_{\b}(w)\le n\le \ell_{\b}(w)$, and $\dim V_{\pm\ell_{\b}(w)}=1$.

Part (2) for $\uIC(w_{\g})_{\cL'}\star\uIC(w)_{\cL}$ is now clear from the degree range on the right side of \eqref{ICw summand}.

Part (1). In view of \eqref{ICw summand}, each perverse cohomology sheaf of $\uIC(w_{\g})_{\cL'}\star\uIC(w)_{\cL}$ is a direct summand of a direct sum of $\uIC(w_{\g\b})_{\cL}$, hence itself a direct sum of $\uIC(w_{\g\b})_{\cL}$. By the decomposition theorem, $\uIC(w_{\g})_{\cL'}\star\uIC(w)_{\cL}$ is then a direct sum of shifts of $\uIC(w_{\g\b})_{\cL}$.

Part (3).  We  claim that $\cP$ is a direct sum of $\uIC(v)_{\cL}$ for $v\in \b$ and $v<_{\b}w$. In fact, the argument in the third paragraph of Lemma \ref{l:stalk order} shows that  $\uIC(w')_{\cL}\star\uIC(s)_{\cL}$ only has stalks along $G_{v}$ for $v\le_{\b}w$. Now $\cP$ is supported on $U\bs G_{<w}/U$,  its direct summands can only be $\uIC(v)_{\cL}$ for $v\in \b$ and $v<_{\b}w$.

By the above claim, we have $\ell_{\b}(v)<\ell_{\b}(w)$ for any $\uIC(v)_{\cL}$ that show up in $\cP$. By inductive hypothesis applied to these $\uIC(v)_{\cL}$, we have $\pH^{\pm\ell_{\b}(w)}(\uIC(w_{\g})_{\cL'}\star\cP)=0$. Therefore $\pH^{\pm\ell_{\b}(w)}(\uIC(w_{\g})_{\cL'}\star\uIC(w)_{\cL})\cong \uIC(w_{\g\b})_{\cL}\ot V_{\pm\ell_{\b}(w)}\cong \uIC(w_{\g\b})_{\cL}$. 
\end{proof}

\begin{lemma}\label{l:ss}
Let $s\in W$ be a simple reflection and $w'\in W$ be such that $\ell(w's)=\ell(w')+1$. Then $\uIC(w')_{s\cL}\star\uIC(s)_{\cL}$ is a perverse sheaf.
\end{lemma}
\begin{proof}
If $s\notin \WL$ then $\uIC(w')_{s\cL}\star\uIC(s)_{\cL}\cong \uIC(w)_{\cL}$ by Lemma \ref{l:clean s equiv}(3). 

If $s\in \WL$, by Lemma \ref{l:inert s}, we have $\uIC(w')_{s\cL}\star\uIC(s)_{\cL}\cong \pi^{*}_{s}\pi_{s*}\uIC(w')_{s\cL}[1]$. It suffices to show that $\pi_{s*}\uIC(w')_{s\cL}$ is perverse in the following sense.  Let $\nu: \wt L_{s}\to L_{s}$ be a finite \'etale isogeny such that $\wt\cL$ is defined via a character of $\ker(\nu)$. Let $\wt P_{s}=P_{s}\times_{L_{s}}\wt L_{s}$ and $\wt B=B\times_{L_{s}}\wt L_{s}$.  We have the projection map $\wt\pi_{s}: G/\wt B\to G/\wt P_{s}$. Viewing $\uIC(w')_{s\cL}$ as a complex on the stack $G/\wt B$, then $\pi_{s*}\uIC(w')_{s\cL}$ as a complex on $G/\wt P_{s}$ is simply $\wt\pi_{s*}\uIC(w')_{s\cL}$. We shall show that $\wt\pi_{s*}\uIC(w')_{s\cL}$ is a perverse sheaf on $G/\wt P_{s}$.  Since $\wt\pi_{s}$ is smooth of relative dimension $1$, $\wt\pi^{*}_{s}[1]$ preserves perverse sheaves, which would imply that  $\uIC(w')_{s\cL}\star\uIC(s)_{\cL}\cong \wt\pi^{*}_{s}\wt\pi_{s*}\uIC(w')_{s\cL}[1]$ is perverse.

For $v\in W$, let $\ov v$ be its image in $W/\j{s}$. Then $G/\wt P_{s}=\bigsqcup_{\ov v\in W/\j{s}}B\ov v P_{s}/\wt P_{s}$ is a stratification of $G/\wt P_{s}$, and $\dim B\ov v P_{s}/\wt P_{s}=\ell(\ov v):=\min\{\ell(v),\ell(vs)\}$.   By Verdier duality, it suffices to show that for any $v\le w'$ and $x\in (B\ov v P_{s})/\wt P_{s}=(G_{v}\cup G_{vs})/\wt P_{s}$, the stalk of $\wt\pi_{s*}\uIC(w')_{s\cL}$ at $x$, which is $\cohog{*}{\wt\pi^{-1}_{s}(x), \uIC(w')_{s\cL}|_{\wt\pi_{s}^{-1}(x)}}$, lies in degrees $\le -\ell(\ov v)$.  Note that $\wt\pi^{-1}_{s}(x)\cong\PP^{1}$ for any $x\in G/\wt P_{s}$.

First consider the case $v<w'$, and we may assume $vs<v$. Then $\uIC(w')_{s\cL}|_{G_{v}}$ lies in degrees $\le -\ell(v)-1$ and $\uIC(w')_{s\cL}|_{G_{vs}}$ lies in degrees $\le -\ell(vs)-1=-\ell(v)$. We have $\wt\pi^{-1}_{s}(x)\cap G_{v}/\wt B\cong \AA^{1}$ and $\wt\pi^{-1}_{s}(x)\cap G_{vs}/\wt B\cong \pt$. Therefore $\cohog{*}{\wt\pi^{-1}_{s}(x), \uIC(w')_{s\cL}|_{\ov\pi_{s}^{-1}(x)}}$ lies in degrees $\le -\ell(v)-1+2=-\ell(\ov v)$. 

If $v=w'$, then the stalk of $\wt\pi_{s*}\uIC(w')_{s\cL}$ at $x\in (B\ov w' P_{s})/\wt P_{s}$ lies in degree $-\ell(w')$ because $G_{w'}/\wt B\to (B\ov w' P_{s})/\wt P_{s}$ is an isomorphism. This finishes the stalk degree estimates needed to show that $\wt\pi_{s*}\uIC(w')_{s\cL}$ is perverse.
\end{proof}

\begin{prop}\label{p:stalk Th} Let $N_{\cL}$ be the length of the longest element in the Coxeter group $\WL$ (with respect to its own simple reflections). For $\b\in{}_{\cL'}\un W_{\cL}$ and $w\in \b$, we have
\begin{equation*}
i_{w}^{*}\uIC(w_{\b})_{\cL}\cong \uC(w)_{\cL}[N_{\cL}-\ell_{\b}(w)]; \quad i_{w}^{!}\uIC(w_{\b})_{\cL}\cong \uC(w)_{\cL}[-N_{\cL}+\ell_{\b}(w)].
\end{equation*}
Here $\ell_{\b}$ is the function defined in \eqref{def ell beta}.
\end{prop}
\begin{proof}
The second isomorphism follows from the first one by Verdier duality. We prove the first one by backward induction on $\ell(w)$ (we allow $\cL$ to vary in $\fo$, and $w_{\b}$ is determined by $w$ and $\cL$). If $w=w_{0}$ is the longest element in $W$, then $w_{0}=w_{\b}$ for the block $\b$ containing $w_{0}$, and $i_{w_{0}}^{*}\uIC(w_{\b})_{\cL}\cong \uC(w_{0})_{\cL}$ by definition (and in this case $\ell_{\b}(w_{0})=N_{\cL}$). 

Now suppose the isomorphism holds for any $w\in W$ such that $\ell(w)>n$. Let $w\in W$ be such that $\ell(w)=n$, and let $\b\in {}_{w\cL}\un W_{\cL}$ be the block containing $w$. Let $s$ be a simple reflection in $W$ such that $\ell(ws)=\ell(w)+1$.  We denote $ws$ by $w'$. Let $\b'=\b s\in {}_{w\cL}\un W_{s\cL}$, the block containing $w'$.

If $s\notin \WL$, then by Lemma \ref{l:clean s equiv}(1), right convolution with $\uIC(s)_{s\cL}$ gives an equivalence ${}_{w\cL}\un\cD_{\cL}\to {}_{w\cL}\un\cD_{s\cL}$ sending $\uIC(w_{\b})_{\cL}$ to $\uIC(w_{\b}s)_{s\cL}$  and $\unb(w)_{\cL}$ to $\unb(w')_{s\cL}$. By Corollary \ref{c:trans min}, $w_{\b}s=w_{\b'}$. Therefore we have an isomorphism of graded $\upH^{*}_{T_{k}}(\pt_{k})$-modules (coming from the left $T$-action)
\begin{eqnarray}\label{Hom Th s}
\Hom(\uIC(w_{\b})_{\cL}, \unb(w)_{\cL})\cong\Hom(\uIC(w_{\b'})_{s\cL}, \unb(w')_{s\cL}).
\end{eqnarray}
Applying the inductive hypothesis to $\uIC(w_{\b'})_{s\cL}$ and $w'$ (which is longer than $w$), we get
\begin{eqnarray*}
&&\Hom(\uIC(w_{\b'})_{s\cL}, \unb(w')_{s\cL})\cong \Hom(i^{*}_{w'}\uIC(w_{\b'})_{s\cL}, \uC(w')_{s\cL})\\
&\cong& \End(\uC(w')_{s\cL})[-N_{s\cL}+\ell_{\b'}(w')]\cong\upH^{*}_{\G(w')_{k}}(\pt_{k})[-N_{s\cL}+\ell_{\b'}(w')].
\end{eqnarray*}
The last isomorphism uses Lemma \ref{l:one cell}.  Similarly, 
\begin{equation*}
\Hom(\uIC(w_{\b})_{\cL}, \unb(w)_{\cL})\cong\Hom(i^{*}_{w}\uIC(w_{\b})_{\cL}, \uC(w)_{\cL})\cong \Hom_{[\{\dw\}/\G(w)]}(i^{*}_{\dw}(\uIC(w_{\b})_{\cL}), \Qlbar)[\ell(w)].
\end{equation*}
In view of \eqref{Hom Th s}, we have an isomorphism of graded $\upH^{*}_{T_{k}}(\pt_{k})$-modules
\begin{equation*}
\upH^{*}_{\G(w')_{k}}(\{\dw\})[-N_{s\cL}+\ell_{\b'}(w')]\cong \Hom(i^{*}_{\dw}\uIC(w_{\b})_{\cL}, \Qlbar)[\ell(w)].
\end{equation*}
This forces $i^{*}_{\dw}(\uIC(w_{\b})_{\cL})\cong\Qlbar[\ell(w)+N_{s\cL}-\ell_{\b'}(w')]\in \cD^{b}_{\G(w)_{k}}(\{\dw\})$, which implies that $i_{w}^{*}\uIC(w_{\b})_{\cL}\cong \uC(w)_{\cL}[N_{s\cL}-\ell_{\b'}(w')]$ by Lemma \ref{l:one cell}. Clearly $N_{s\cL}=N_{\cL}$. By Lemma \ref{l:ell beta}\eqref{ell beta eq}, $\ell_{\b}(w)=\ell_{\b'}(w')$. Therefore $i_{w}^{*}\uIC(w_{\b})_{\cL}\cong \uC(w)_{\cL}[N_{\cL}-\ell_{\b}(w)]$.

If $s\in \WL$, then $\uIC(w_{\b})_{\cL}$ is in the image of $\pi^{*}_{s}$ by Lemma \ref{l:IC descend}, which implies that the stalks of $\uIC(w_{\b})_{\cL}$ at $\dw$ and at $\dw'$ are isomorphic to each other.  By inductive hypothesis, the stalk  $i^{*}_{\dw'}\uIC(w_{\b})_{\cL}\cong \Qlbar[\ell(w')+N_{\cL}-\ell_{\b}(w')]$ (now $w'\in \b$). By Lemma \ref{l:ell beta}\eqref{ell beta eq}, we have $\ell_{\b}(w')=\ell_{\b}(ws)=\ell_{\b}(w)+1$. Therefore $i^{*}_{\dw}\uIC(w_{\b})_{\cL}\cong \Qlbar[\ell(w')+N_{\cL}-\ell_{\b}(w')]\cong \Qlbar[\ell(w)+N_{\cL}-\ell_{\b}(w)]$, and hence $i^{*}_{w}\uIC(w_{\b})_{\cL}\cong \uC(w)_{\cL}[N_{\cL}-\ell_{\b}(w)]$.
\end{proof}

\subsection{Rigidified maximal IC sheaf in the neutral block}\label{ss:rig neutral}
Let $\cL\in\fo$. Recall that $\d_{\cL}=\IC(\dot e)_{\cL}\in{}_{\cL}\cD_{\cL}^{\c}$  is the monoidal unit of ${}_{\cL}\cD_{\cL}^{\c}$ under convolution. Recall that $N_{\cL}$ is the length of the longest element $w_{\cL,0}$ in the Coxeter group $\WL$ (in terms of simple reflections in $\WL$). 

A {\em rigidified maximal IC sheaf} in ${}_{\cL}\cD_{\cL}^{\c}$ is a pair $(\Th, \e)$ where $\Th\in {}_{\cL}\cD_{\cL}^{\c}$ is such that $\Th[N_{\cL}]$ is a maximal IC sheaf (i.e., $\om\Th[N_{\cL}]\cong\uIC(w_{\cL,0})_{\cL}$) and $\e: \Th\to \d_{\cL}$ is a nonzero map in ${}_{\cL}\cD_{\cL}^{\c}$.

Rigidified maximal IC sheaves exist. Indeed, by Proposition \ref{p:stalk Th}, $i^{*}_{e}\IC(\dw_{\cL,0})[-N_{\cL}]\cong C(\dot e)\ot V$ for a one-dimensional $\Fr$-module $V$. Therefore, letting $\Th=\IC(\dw_{\cL,0})[-N_{\cL}]\ot V^{*}$, we get a nonzero map $\e:\Th\to \d_{\cL}$ by adjunction. 
    
Let $(\Th,\e)$ and $(\Th',\e')$ be two rigidified maximal IC sheaves in ${}_{\cL}\cD_{\cL}^{\c}$. Then there is a unique isomorphism $\un\a: \om\Th\isom\om\Th'$ such that $\e'\c\un\a=\e$ as elements in $\Hom(\om\Th, \om\d_{\cL})$. The uniqueness of $\un\a$ implies that $\Fr(\un\a)=\un\a$; moreover, $\Hom(\om\Th, \om\Th'[-1])=0$, hence $\un\a$ uniquely lifts to an isomorphism $\a: \Th\isom \Th'$  inside ${}_{\cL}\cD^{\c}_{\cL}$. Therefore any two rigidified maximal IC sheaves are isomorphic to each other. Moreover, the automorphism group of any rigidified maximal IC sheaf is trivial. Therefore we may identify all the rigidified maximal IC sheaves in ${}_{\cL}\cD^{\c}_{\cL}$ as a single object and denote it by
\begin{equation*}
(\Th^{\c}_{\cL}, \e_{\cL}: \Th^{\c}_{\cL}\to \d_{\cL}).
\end{equation*}
We denote by
\begin{equation*}
(\uTh^{\c}_{\cL},\un\e_{\cL})=\om(\Th^{\c}_{\cL}, \e_{\cL}).
\end{equation*}
the rigidified maximal IC sheaf in ${}_{\cL}\un\cD^{\c}_{\cL}$.

\begin{prop}\label{p:coalg} There is a unique coalgebra structure on $\Th^{\c}_{\cL}$ (inside the monoidal category ${}_{\cL}\cD^{\c}_{\cL}$) with $\e_{\cL}$ as the counit map.
\end{prop}
\begin{proof} For each $n\ge2$, let $(\Th^{\c}_{\cL})^{\star n}$ be the $n$-fold convolution of $\Th^{\c}_{\cL}$. We will construct a comultiplication map
\begin{equation*}
\mu^{n}_{\cL}: \Th^{\c}_{\cL}\to(\Th^{\c}_{\cL})^{\star n}
\end{equation*}
characterized as the unique map such that the following diagram is commutative
\begin{equation}\label{comult n}
\xymatrix{     \Th^{\c}_{\cL}\ar[r]^-{\mu^{n}_{\cL}}\ar[d]^{\e_{\cL}} &  (\Th^{\c}_{\cL})^{\star n}
\ar[d]^{\e^{\star n}_{\cL}}     \\
\d_{\cL}\ar[r]^-{\sim} & (\d_{\cL})^{\star n}      }
\end{equation}
where the bottom arrow is the canonical isomorphism from the monoidal unit structure on $\d_{\cL}$. 

Let $\un\d_{\cL}=\om\d_{\cL}$.  By an iterated application of Proposition \ref{p:conv Th}(3), we see that $\pH^{i}((\uTh^{\c}_{\cL})^{\star n})=0$ for $i<N_{\cL}$, and $\uTh^{\c}_{\cL}[N_{\cL}]\cong\pH^{N_{\cL}}((\uTh^{\c}_{\cL})^{\star n})$. Therefore there is a nonzero map $\un\mu^{n}: \uTh^{\c}_{\cL}\to (\uTh^{\c}_{\cL})^{\star n}$,  unique up to a scalar. Since nonzero maps $\uTh^{\c}_{\cL}\to \un\d_{\cL}$ are unique up to a scalar, the Claim below shows that there is a unique nonzero multiple of $\un\mu^{n}$, call it $\un\mu^{n}_{\cL}$, that makes the non-mixed version of the diagram \eqref{comult n} (i.e., the diagram after applying $\om$ to all terms) commutative. The uniqueness of $\un\mu^{n}_{\cL}$ implies that it is invariant under Frobenius; moreover $\Hom(\uTh^{\c}_{\cL},(\uTh^{\c}_{\cL})^{\star n}[-1])=0$ for perverse degree reasons, therefore $\un\mu_{\cL}^{n}$ determines uniquely a morphism $\mu_{\cL}^{n}: \Th^{\c}_{\cL}\to(\Th^{\c}_{\cL})^{\star n}$ in ${}_{\cL}\cD^{\c}_{\cL}$. 

\begin{claim} The composition (for any nonzero choice of $\mu^{n}$)
\begin{equation*}
\uTh^{\c}_{\cL}\xr{\un\mu^{n}}(\uTh^{\c}_{\cL})^{\star n}\xr{\un\e^{\star n}_{\cL}}(\un\d_{\cL})^{\star n}\cong\un\d_{\cL}
\end{equation*}
is nonzero. 
\end{claim}
\begin{proof}[Proof of Claim]
We prove the claim by induction on $n$. For $n=2$, we take the degree zero stalks of the above maps at the identity element $\dot e\in G$, the map becomes (where $i:\{\dot e\}\incl G$ is the inclusion)
\begin{eqnarray}\label{stalk1}
&&i^{*}\uTh^{\c}_{\cL}\xr{\upH^{0}i^{*}\un\mu^{2}}\upH^{0}i^{*}(\uTh^{\c}_{\cL}\star\uTh^{\c}_{\cL})
\cong \cohog{0}{(G/B)_{k}, \inv^{*}\uTh^{\c}_{\cL}\ot\uTh^{\c}_{\cL}}\\
\notag &\xr{\res}& \upH^{0}i^{*}(\inv^{*}\uTh^{\c}_{\cL}\ot\uTh^{\c}_{\cL})\\
\notag &\cong& i^{*}\uTh^{\c}_{\cL}\ot i^{*}\uTh^{\c}_{\cL}\cong i^{*}\un\d_{\cL}\ot i^{*}\un\d_{\cL}=\Qlbar=i^{*}\un\d_{\cL}.
\end{eqnarray}
Here $\upH^{0}i^{*}\un\mu^{2}$ is an isomorphism since $\uTh^{\c}_{\cL}\star\uTh^{\c}_{\cL}$ is a direct sum of $\uTh^{\c}_{\cL}$ and $\uTh^{\c}_{\cL}[-j]$ for $j>0$ by Proposition \ref{p:conv Th}, and $i^{*}\uTh^{\c}_{\cL}$ is concentrated in degree 0 by Proposition \ref{p:stalk Th}. The second isomorphism follows from the definition of the convolution, where $\inv:G\to G$ is the inversion map. The map ``\res'' is the restriction map to $\{\dot e\}$. To prove the claim, we show that the composition \eqref{stalk1} is an isomorphism. It suffices to show that $\res$ is an isomorphism. Let $\cF=\inv^{*}\uTh^{\c}_{\cL}\ot\uTh^{\c}_{\cL}\in D^{b}(B_{k}\bs G_{k}/B_{k})$. By Proposition \ref{p:stalk Th}, the stalk of $\cF$ along the cell $X_{w,k}=G_{w,k}/B_{k}$ vanishes if $w\notin \WL$, and if $w\in \WL$, it lies in degree $2(-\ell(w)+\ell_{\cL}(w))$. We compute $\cohog{*}{(G/B)_{k},\cF}$ using the stratification $G/B=\sqcup_{w\in W}X_{w}$,  the contribution of $X_{w}$ is $0$ if $w\notin \WL$ and is  $\cohoc{*}{X_{w,k}, \Qlbar[2\ell(w)-2\ell_{\cL}(w)]}\cong \Qlbar[-2\ell_{\cL}(w)]$ for $w\in \WL$. This shows that the only contribution to $\cohog{0}{(G/B)_{k},\cF}$ is from the point stratum $X_{e}$, hence $\res$ is an isomorphism. This proves the case $n=2$.

Suppose the Claim is proved for $n-1$. Up to a nonzero scalar, $\un\mu^{n}$ is equal to the composition
\begin{equation*}
\uTh^{\c}_{\cL}\xr{\un\mu^{2}}\uTh^{\c}_{\cL}\star\uTh^{\c}_{\cL}\xr{\un\mu^{n-1}\star\id}(\uTh^{\c}_{\cL})^{\star (n-1)}\star\uTh^{\c}_{\cL}.
\end{equation*}
Composing with $\un\e^{\star n}_{\cL}$, we see that up to a nonzero scalar, $\un\e^{\star n}_{\cL}\c\un\mu^{n}$ can be rewritten as the composition
\begin{equation*}
\uTh^{\c}_{\cL}\xr{\un\mu^{2}}\uTh^{\c}_{\cL}\star\uTh^{\c}_{\cL}\xr{(\un\e^{\star (n-1)}_{\cL}\c\un\mu^{n-1})\star\id}\un\d_{\cL}\star\uTh^{\c}_{\cL}\xr{\id\star\un\e_{\cL}}\un\d_{\cL}\star\un\d_{\cL}\cong\un\d_{\cL}.
\end{equation*}
By inductive hypothesis, $\un\e^{\star (n-1)}_{\cL}\c\un\mu^{n-1}$ is a nonzero multiple of $\un\e_{\cL}$, therefore the above composition is, up to a nonzero scalar, $\un\e_{\cL}^{\star 2}\c\un\mu^{2}$, which is nonzero by the $n=2$ case proved above. This completes the induction step.
\end{proof}

We continue with the proof of Proposition \ref{p:coalg}.  Co-associativity of $\mu^{2}_{\cL}$ follows by the uniqueness of $\mu^{3}_{\cL}$. It remains to check the counit axioms, i.e., the compositions
\begin{eqnarray}\label{right unit}
\Th^{\c}_{\cL}\xr{\mu^{2}_{\cL}}\Th^{\c}_{\cL}\star\Th^{\c}_{\cL}\xr{\id\star\e_{\cL}}\Th^{\c}_{\cL}\star\d_{\cL}\cong\Th^{\c}_{\cL},\\
\label{left unit}\Th^{\c}_{\cL}\xr{\mu^{2}_{\cL}}\Th^{\c}_{\cL}\star\Th^{\c}_{\cL}\xr{\e_{\cL}\star\id}\d_{\cL}\star\Th^{\c}_{\cL}\cong\Th^{\c}_{\cL}
\end{eqnarray}
are the identity maps.  Composing \eqref{right unit} with $\e_{\cL}$ we recover the map $\Th^{\c}_{\cL}\xr{\mu^{2}_{\cL}}\Th^{\c}_{\cL}\star\Th^{\c}_{\cL}\xr{\e^{\star2}_{\cL}}\d_{\cL}$ which is equal to $\e_{\cL}$ by construction. This forces \eqref{right unit} to be the identity because  the endomorphisms of $\Th^{\c}_{\cL}$ are  scalars. The same argument works to show that \eqref{left unit} is the identity map.
\end{proof}

\begin{defn}\label{d:can IC} For $w\in \WL$, define
\begin{eqnarray*}
&&C(w)_{\cL}^{\da}:=i^{*}_{w}\Th^{\c}_{\cL}\j{\ell_{\cL}(w)},\\
&&\D(w)_{\cL}^{\da}:=i_{w!}C(w)_{\cL}^{\da}, \quad \nb(w)_{\cL}^{\da}:=i_{w*}C(w)_{\cL}^{\da}, \quad \IC(w)_{\cL}^{\da}:=i_{w!*}C(w)_{\cL}^{\da}.
\end{eqnarray*}
\end{defn}
By Proposition \ref{p:stalk Th}, $\om C(w)_{\cL}^{\da}\cong \uC(w)_{\cL}$. By Proposition \ref{p:parity purity}(1), $C(w)_{\cL}^{\da}$ is pure of weight zero. Therefore $\om\IC(w)_{\cL}^{\da}\cong \uIC(w)_{\cL}$ and $\IC(w)_{\cL}^{\da}$ is pure of weight zero. We call $\IC(w)^{\da}_{\cL}$ a {\em rigidified IC sheaf}. Note there is a canonical isomorphism $\Th_{\cL}^{\c}\j{N_{\cL}}\cong\IC(w_{\cL,0})^{\da}_{\cL}$.

\begin{lemma}\label{l:map Th to IC}
There is a unique map $\th^{\da}_{w}: \Th^{\c}_{\cL}\to \IC(w)_{\cL}^{\da}\j{-\ell_{\cL}(w)}$ whose restriction under $i^{*}_{w}$ is the identity map of $C(w)^{\da}_{\cL}\j{-\ell_{\cL}(w)}$.
\end{lemma}
\begin{proof}
By Corollary \ref{c:filM}, there is a filtration on $M=\Homb(\Th^{\c}_{\cL}, \IC(w)^{\da}_{\cL}\j{-\ell_{\cL}(w)})$ indexed by $\{v\in\WL;v\le w\}$ such that $\Gr^{F}_{v}M\cong \Homb(i^{*}_{v}\Th^{\c}_{\cL}, i^{!}_{v}\IC(w)^{\da}_{\cL}\j{-\ell_{\cL}(w)})$ as graded $R\ot R$-modules.  By Proposition \ref{p:stalk Th}, $\om\Gr^{F}_{v}M\cong \Homb(\uC(v)_{\cL}, i_{v}^{!}\uIC(w)_{\cL})[\ell_{\cL}(v)-\ell_{\cL}(w)]$. If $v<w$, $i_{v}^{!}\uIC(w)_{\cL}$ lies in perverse degrees $>0$; moreover, this costalk is zero unless $v<_{\WL}w$ by Lemma \ref{l:stalk order} (in particular $\ell_{\cL}(v)<\ell_{\cL}(w)$). These imply that $\Gr^{F}_{v}M$ is concentrated in degrees $\ge2$ for $v<w$. Therefore the quotient map $M\to \Gr^{F}_{w}M$ is an isomorphism in degrees $\le1$, and in particular in degree $0$. Now $\Gr^{F}_{w}M=\Homb(i^{*}_{w}\Th^{\c}_{\cL}, i_{w}^{*}\IC(w)^{\da}_{\cL}\j{-\ell_{\cL}(w)})\cong \Homb(C(w)^{\da}_{\cL}\j{-\ell_{\cL}(w)}, C(w)^{\da}_{\cL}\j{-\ell_{\cL}(w)})$ and the quotient map $M\to \Gr^{F}_{w}M$ is induced by $i^{*}_{w}$. Therefore there is a unique $\th^{\da}_{w}\in M^{0}$ mapping to $\id\in \End(C(w)^{\da}_{\cL}\j{-\ell_{\cL}(w)})=(\Gr^{F}_{w}M)^{0}$.
\end{proof}

\begin{lemma}\label{l:ICda es}
\begin{enumerate}
\item There is a unique isomorphism $\io_{e}: \IC(e)^{\da}_{\cL}\cong\d_{\cL}$ such that $\io_{e}\c\th^{\da}_{e}=\e_{\cL}$.
\item Let $s\in W$ be a simple reflection and $s\in \WL$. Recall the object $\IC(s)_{\cL}$ introduced in \S\ref{ss:inert s}. Then there is a unique isomorphism $\io_{s}: \IC(s)^{\da}_{\cL}\cong\IC(s)_{\cL}$ such that the composition $\io_{s}\c\th^{\da}_{s}: \Th^{\c}_{\cL}\to \IC(s)_{\cL}\j{-1}$ restricts to the identity map on the stalks at $e\in G$. (Recall the stalks of both $\Th^{\c}_{\cL}$ and $\IC(s)_{\cL}\j{-1}$ are equipped with an isomorphism with the trivial $\Fr$-module $\Qlbar$.)
\end{enumerate}
\end{lemma}
\begin{proof}
(1) The rigidification $\e_{\cL}: \Th^{\c}_{\cL}\to \d_{\cL}$ gives by adjunction a nonzero map $C(e)^{\da}_{\cL}=i^{*}_{e}\Th^{\c}_{\cL}\to C(\dot e)_{\cL}$, which has to be an isomorphism. This induces the desired isomorphism $\io_{e}$. The uniqueness part is clear.

(2) By Lemma \ref{l:IC descend},  we can write $\Th^{\c}_{\cL}=\pi_{s}^{*}\ov\Th$ for some shifted perverse sheaf $\ov\Th\in {}_{\cL}\cD_{\wt\cL}$. Since the stalk of $\Th^{\c}_{\cL}$ at $\dot e$ is the trivial $\Fr$-module by the rigidification $\e_{\cL}$, we have $\ov\Th|_{P_{s}}\cong \wt\cL$, and hence $i^{*}_{\le s}\Th^{\c}_{\cL}\cong \wt\cL\in {}_{\cL}\cD(\le s)_{\cL}$.  By adjunction $\th^{\da}_{s}$ gives a nonzero map $\wt\cL\cong i^{*}_{\le s}\Th^{\c}_{\cL}\to i^{*}_{\le s}\IC(s)^{\da}_{\cL}\j{-1}$, which has to be an isomorphism. This induces an isomorphism $\IC(s)^{\da}_{\cL}\cong i_{\le s *}\wt\cL\j{1}=\IC(s)_{\cL}$. The uniqueness of $\io_{s}$ is clear.
\end{proof}

\subsection{Rigidified maximal IC sheaves in general} Let $\cL,\cL'\in \fo$ and $\b\in{}_{\cL'}\un W_{\cL}$.   Recall we have defined the rigidified minimal IC sheaves $\IC(w^{\b})^{\da}_{\cL}$ in \S\ref{ss:rig min}. We define the {\em rigidified maximal IC sheaf} in block $\b$ to be
\begin{equation*}
{}_{\cL'}\Th^{\b}_{\cL}:=\Th^{\c}_{\cL'}\star\IC(w^{\b})^{\da}_{\cL}.
\end{equation*}
Then $\om({}_{\cL'}\Th^{\b}_{\cL})\cong \un\IC(w^{\b})_{\cL}[N_{\cL}]$. It carries a canonical nonzero map 
\begin{equation*}
{}_{\cL'}\e^{\b}_{\cL}: {}_{\cL'}\Th^{\b}_{\cL}=\Th^{\c}_{\cL'}\star\IC(w^{\b})^{\da}_{\cL}\xr{\e_{\cL'}\star\id} \d_{\cL'}\star\IC(w^{\b})^{\da}_{\cL}=\IC(w^{\b})^{\da}_{\cL}.
\end{equation*}
The pair $({}_{\cL'}\Th^{\b}_{\cL},{}_{\cL'}\e^{\b}_{\cL})$ is unique up to a unique isomorphism.

We consider a slightly more general notion of rigidified maximal IC sheaves that does not require rigidifying the minimal IC sheaves first. Let $\xi\in {}_{\cL'}\fP^{\b}_{\cL}$ be any minimal IC sheaf in the block $\b$  (see \S\ref{ss:min}). Let
\begin{equation*}
\Th(\xi):=\Th^{\c}_{\cL}\star\xi.
\end{equation*}
Then $\Th(\xi)$ is equipped with a nonzero map
\begin{equation*}
\e(\xi):=\e_{\cL'}\star\id_{\xi}: \Th(\xi)=\Th^{\c}_{\cL'}\star\xi\to \d_{\cL'}\star\xi\cong\xi.
\end{equation*}
The pair $(\Th(\xi),\e(\xi))$ has trivial automorphism group. We denote
\begin{equation*}
(\uTh(\xi), \un\e(\xi)):=\om(\Th(\xi), \e(\xi))\in {}_{\cL'}\un\cD^{\b}_{\cL}.
\end{equation*}
By Proposition \ref{p:min th}(3),   $\uTh(\xi)\cong \uIC(w_{\b})_{\cL}[-N_{\cL}]$. 

\begin{lemma}\label{l:switch Th} For $\xi\in {}_{\cL'}\fP^{\b}_{\cL}$,  there is a unique isomorphism
\begin{equation*}
\t(\xi): \xi\star\Th^{\c}_{\cL}\isom\Th^{\c}_{\cL'}\star\xi
\end{equation*}
making the following diagram commutative
\begin{equation}\label{switch Th}
\xymatrix{     \xi\star\Th^{\c}_{\cL}\ar[d]^{\id\star\e_{\cL}}\ar[rr]^-{\t(\xi)} & & \Th^{\c}_{\cL'}\star\xi\ar[d]^{\e_{\cL'}\star\id}   \\
\xi\star\d_{\cL}\ar[r]^-{\sim} & \xi\ar[r]^-{\sim} & \d_{\cL'}\star\xi}
\end{equation}
\end{lemma}
\begin{proof}
By Proposition \ref{p:min th},  both $ \om(\xi\star\Th^{\c}_{\cL})$ and $\om(\Th^{\c}_{\cL'}\star\xi) $ are isomorphic to $\uIC(w_{\b})_{\cL}[-N_{\cL}]$, therefore isomorphisms $\un\t:\om(\xi\star\Th^{\c}_{\cL})\isom \om(\Th^{\c}_{\cL'}\star\xi) $ are unique up to a nonzero scalar. Moreover,  since $\xi\star(-)$ is an equivalence, $\Hom(\xi\star\Th^{\c}_{\cL}, \xi)\cong\Hom(\Th^{\c}_{\cL}, \d_{\cL})\cong\Qlbar$.   Therefore there is a unique $\un\t$  making the non-mixed version of the diagram \eqref{switch Th} commutative. Uniqueness of $\un\t$ implies that it is $\Fr$-invariant and lifts to a unique isomorphism $\t(\xi)$ in ${}_{\cL'}\cD^{\b}_{\cL}$.
\end{proof}

\begin{cor}\label{c:Th by left wb} For any block $\b\in {}_{\cL'}\un W_{\cL}$, there is a unique isomorphism
\begin{equation*}
\IC(w^{\b})^{\da}_{\cL}\star\Th^{\c}_{\cL}\isom {}_{\cL'}\Th^{\b}_{\cL}=\Th^{\c}_{\cL'}\star \IC(w^{\b})^{\da}_{\cL} 
\end{equation*}
that intertwines the canonical maps of both sides to $\IC(w^{\b})^{\da}_{\cL}$. In other words, we can alternatively define ${}_{\cL'}\Th^{\b}_{\cL}$ to be $\IC(w^{\b})^{\da}_{\cL}\star\Th^{\c}_{\cL}$.
\end{cor}

Let $\cL''\in\fo$,  $\g\in{}_{\cL''}\un W_{\cL'}$ and $\y\in {}_{\cL''}\fP^{\g}_{\cL'}$. To save notation, we will abbreviate $\y\star\xi$ by $\y\xi$. Consider the composition
\begin{eqnarray*}
\ph(\y,\xi): &  &\Th(\y\xi)=\Th^{\c}_{\cL''}\star(\y\xi)\\
&\xr{\mu^{2}_{\cL''}\star\id}&(\Th^{\c}_{\cL''}\star\Th^{\c}_{\cL'}) \star (\y\xi)=\Th^{\c}_{\cL''}\star(\Th^{\c}_{\cL'} \star \y)\star\xi\\
&\xr{\id\star\t(\y)^{-1}\star\id}& \Th^{\c}_{\cL''}\star\y\star\Th^{\c}_{\cL'}\star\xi=\Th(\y)\star\Th(\xi).
\end{eqnarray*}
Here $\mu^{2}_{\cL''}: \Th^{\c}_{\cL''}\to \Th^{\c}_{\cL''}\star\Th^{\c}_{\cL''}$ is the  comultiplication constructed in Proposition \ref{p:coalg}.

\begin{prop}\label{p:coalg ext} Notation as above.
\begin{enumerate}
\item The composition
\begin{equation*}
\Th(\y\xi)\xr{\ph(\y,\xi)}\Th(\y)\star\Th(\xi)\xr{\e(\y)\star\e(\xi)}\y\xi
\end{equation*}
is the same as $\e(\y\xi)$.

\item The following compositions  are the identity maps
\begin{eqnarray*}
\Th(\xi)\xr{\ph(\d_{\cL'},\xi)}\Th^{\c}_{\cL'}\star\Th(\xi)\xr{\e_{\cL'}\star\id}\d_{\cL'}\star\Th(\xi)\cong\Th(\xi),\\
\Th(\xi)\xr{\ph(\xi,\d_{\cL})}\Th(\xi)\star\Th^{\c}_{\cL}\xr{\id\star\e_{\cL}}\Th(\xi)\star\d_{\cL}\cong\Th(\xi).
\end{eqnarray*}

\item For $\cL'''\in\fo$ and $\z\in{}_{\cL'''}\Xi_{\cL''}$,  the following diagram is commutative
\begin{equation}\label{asso}
\xymatrix{\Th(\z\y\xi)\ar[d]_{\ph(\z\y,\xi)}\ar[rr]^-{\ph(\z,\y\xi)} && \Th(\z)\star\Th(\y\xi)\ar[d]^{\id\star\ph(\y,\xi)}\\
\Th(\z\y)\star\Th(\xi)\ar[rr]^-{\ph(\z,\y)\star\id} &&
\Th(\z)\star\Th(\y)\star\Th(\xi)
}
\end{equation} 
\end{enumerate}
\end{prop}
\begin{proof} Part (1) follows from the definition of $\ph(\y,\xi)$ and the characterizing property of the comultiplication $\mu^{2}_{\cL}$ on $\Th^{\c}_{\cL}$ that $\e_{\cL}^{\star 2}\c\mu^{2}_{\cL}=\e_{\cL}$.

The proof of (2) is similar to the verification of the counit axioms in the proof of Proposition \ref{p:coalg}. We omit it here.

To prove (3),  we observe that by Proposition \ref{p:conv Th},  $\Th(\z\y\xi)$ is identified with the lowest nonzero perverse cohomology of $\Th(\z)\star\Th(\y)\star\Th(\xi)$, therefore  nonzero maps $\Th(\z\y\xi)\to \Th(\z)\star\Th(\y)\star\Th(\xi)$  are unique up to a scalar. Therefore it suffices to show that, after composing with $\e(\z)\star\e(\y)\star\e(\xi): \Th(\z)\star\Th(\y)\star\Th(\xi)\to \z\y\xi$, both compositions in the diagram \eqref{asso} are equal to $\e(\z\y\xi)$. But this follows from iterated applications of part (1).
\end{proof}

Now we can extend the definition of rigidified IC sheaves to all blocks.
Recall we have defined the rigidified minimal IC sheaves $\IC(w^{\b})^{\da}_{\cL}$ in \S\ref{ss:rig min}. 
\begin{defn}\label{def:rig IC} Let $w\in {}_{\cL'}W_{\cL}$ and let  $\b\in {}_{\cL'}\un W_{\cL}$ be its block. Write $w=xw^{\b}$ for $x\in W_{\cL'}^{\c}$. We define
\begin{eqnarray*}
&&\D(w)_{\cL}^{\da}:=\D(x)^{\da}_{\cL'}\star \IC(w^{\b})^{\da}_{\cL}, \quad \nb(w)_{\cL}^{\da}:=\nb(x)^{\da}_{\cL'}\star \IC(w^{\b})^{\da}_{\cL},\\
&& \IC(w)_{\cL}^{\da}:=\IC(x)^{\da}_{\cL'}\star\IC(w^{\b})_{\cL}^{\da}.\\
&&C(w)_{\cL}^{\da}:=i^{*}_{w}\D(w)^{\da}_{\cL}=i^{*}_{w}\nb(w)^{\da}_{\cL}=i^{*}_{w}\IC(w)^{\da}_{\cL}.
\end{eqnarray*}
\end{defn}
By Lemma \ref{l:map Th to IC}, there is a canonical map $\th^{\da}_{w}: {}_{\cL'}\Th^{\b}_{\cL}\to \IC(w)^{\da}_{\cL}\j{-\ell_{\b}(w)}$ by applying $\star\IC(w^{\b})^{\da}_{\cL}$ to $\th^{\da}_{x}$, that induces an isomorphism after applying $i^{*}_{w}$. Therefore $C(w)_{\cL}^{\da}$ can be alternatively defined as $i^{*}_{w}({}_{\cL'}\Th^{\b}_{\cL})\j{\ell_{\b}(w)}$, and we may in turn define $\D(w)^{\da}_{\cL},\nb(w)^{\da}_{\cL}$ and $\IC(w)^{\da}_{\cL}$ as the $!$, $*$ and middle extensions of $C(w)_{\cL}^{\da}$. 

Using this alternative definition of $C(w)^{\da}_{\cL}, \D(w)^{\da}_{\cL}, \nb(w)^{\da}_{\cL}$ and $\IC(w)^{\da}_{\cL}$ in terms of stalks of ${}_{\cL'}\Th^{\b}_{\cL}$, and using Corollary \ref{c:Th by left wb}, we see that if we write $w=w^{\b}y$ for $y\in W_{\cL}^{\c}$, there are {\em canonical} isomorphisms
\begin{equation*}
\D(w)_{\cL}^{\da}:=\IC(w^{\b})^{\da}_{\cL}\star\D(y)^{\da}_{\cL}, \quad 
\nb(w)_{\cL}^{\da}:=\IC(w^{\b})^{\da}_{\cL}\star\nb(y)^{\da}_{\cL},\quad
\IC(w)_{\cL}^{\da}:=\IC(w^{\b})^{\da}_{\cL}\star\IC(y)^{\da}_{\cL}.\end{equation*}

\section{Monodromic Soergel functor}\label{s:Soergel}

In this section we introduce the Soergel functor between the monodromic Hecke category and the category of graded $R$-bimodules, construct its monoidal structure and prove an analogue of Soergel's Extension Theorem for this functor.

\subsection{$R$-bimodules with Frobenius actions}\label{ss:Rbimod} Let $R=\upH^{*}_{T_{k}}(\pt_{k},\Qlbar)\cong\Sym(\xch(T)_{\Qlbar})$, with the grading $\deg\xch(T)_{\Qlbar}=2$, and the Frobenius action on $\xch(T)_{\Qlbar}$ by $q$. Let $\RRM$ be the category of $\ZZ$-graded $R\ot R$-modules. Let $(R\ot R, \Fr)\gmod$ be the category of $\ZZ$-graded $R\ot R$-modules $M=\oplus_{n}M^{n}$ with a degree-preserving automorphism $\Fr: M\to M$ compatible with the Frobenius action on $R\ot R$; i.e., for homogeneous $a\in  R$ and $m\in M$, we have $\Fr((a\ot 1)m)=q^{\deg(a)/2}(a\ot 1)\Fr(m)$ and $ \Fr((1\ot a)m)=q^{\deg(a)/2}(1\ot a)\Fr(m)$. Let $\om: (R\ot R, \Fr)\gmod\to \RRM$ be the functor forgetting the Frobenius action.

We use $[1]$ for the degree shift for graded $(R\ot R,\Fr)$-modules, i.e., if $M=\oplus_{n\in\ZZ}M^{n}\in \RRM$,  $M[1]$ is the graded $R\ot R$-module with $(M[1])^{n}=M^{n+1}$ as $\Fr$-modules. For $M\in (R\ot R, \Fr)\gmod$, $M(n/2)$ is the same graded $R\ot R$-module as $M$ with the Frobenius action multiplied by $q^{-n/2}$. Let $\j{n}$ be the composition $[n](n/2)$.

For $M_{1},M_{2}\in(R\ot R, \Fr)\gmod$, we understand $M_{1}\ot_{R}M_{2}$ as the tensor product of $M_{1}$ and $M_{2}$ with respect to the second $R$-action on $M_{1}$ and the first $R$-action on $M_{2}$.

For $M_{1}, M_{2}\in \RRM$, their inner Hom is the graded $R\ot R$-module
\begin{equation*}
\Homb(M_{1}, M_{2})=\bigoplus_{n\in \ZZ}\Hom_{\RRM}(M_{1}, M_{2}[n]).
\end{equation*}
If $M_{1}, M_{2}\in \RRF$, then $\Homb(M_{1}, M_{2})$ is also naturally an object in $(R\ot R, \Fr)\gmod$.

For two objects $\cF,\cG\in {}_{\cL'}\cD^{\b}_{\cL}$, let
\begin{equation*}
\Homb(\cF,\cG):=\bigoplus_{n\in\ZZ}\Hom(\cF, \cG[n]).
\end{equation*} 
Since $\Homb(\cF,\cG)=\upH^{*}_{T_{k}\times T_{k}}((U\bs G/U)_{k}, \RuHom(\cF,\cG))$, it is a graded $(R\ot R,\Fr)$-module, where the $R\ot R=\upH^{*}_{T_{k}\times T_{k}}(\pt_{k})$-action comes from the $T\times T$-action on $U\bs G/U$ given in  \S\ref{ss:mono Hecke}. Same notation applies to ${}_{\cL'}\cD(w)_{\cL}$ and ${}_{\cL'}\cD(\le w)_{\cL}$.
 
For each $w\in W$, let $R(w)$ be the graded $R$-bimodule which is the quotient of $R\ot R$ by the ideal generated by $w(a)\ot 1-1\ot a$ for all $a\in R$. We have a canonical isomorphism in $(R\ot R, \Fr)\gmod$
\begin{equation*}
R(w)\cong \upH^{*}_{\G(w)_{k}}(\pt_{k})\cong \Homb(C(\dw)_{\cL},C(\dw)_{\cL}).
\end{equation*}

\begin{defn}
\begin{enumerate}
\item Let $\b\in {}_{\cL'}\un W_{\cL}$ and $\xi\in {}_{\cL'}\fP^{\b}_{\cL}$ be a minimal IC sheaf in the block $\b$. The {\em mixed Soergel functor} associated to $\xi$ is the functor
\begin{equation*}
\MM_{\xi}:=\Homb(\Th(\xi), -):{}_{\cL'}\cD^{\b}_{\cL} \to \RRF.
\end{equation*}
\item The {\em non-mixed  Soergel functor} associated to $\xi$ is 
\begin{equation*}
\uMM_{\xi}:=\Homb(\uTh(\xi), -):{}_{\cL'}\un\cD^{\b}_{\cL} \to \RRM.
\end{equation*}
\item When $\xi=\d_{\cL}\in {}_{\cL}\fP^{\c}_{\cL}$, we denote the corresponding Soergel functors  by $\MM^{\c}=\Homb(\Th^{\c}_{\cL}, -)$ and $\uMM^{\c}=\Homb(\uTh^{\c}_{\cL},-)$.
\end{enumerate}
\end{defn}

\begin{lemma}\label{l:MM th} Let $\xi\in{}_{\cL'}\fP^{\b}_{\cL}$. There is a canonical isomorphism in $\RRF$
\begin{equation*}
\MM_{\xi}(\xi)\cong R(w^{\b})
\end{equation*}
under which the canonical map $\e(\xi):\Th(\xi)\to \xi$ corresponds to $1\in R(w^{\b})$.
\end{lemma}
\begin{proof} By definition, we have
\begin{equation*}
\RHom(\uTh(\xi), \om\xi)\cong \RHom(\uIC(w_{\b})_{\cL}[-N_{\cL}], \unb(w^{\b})_{\cL})=\RHom(i^{*}_{w^{\b}}\uIC(w_{\b})_{\cL}, \uC(w^{\b})_{\cL})[N_{\cL}]. 
\end{equation*}
By Proposition \ref{p:stalk Th}, $i^{*}_{w^{\b}}\uIC(w_{\b})_{\cL}\cong \uC(w^{\b})_{\cL}[N_{\cL}]$. Therefore,
\begin{equation*}
\RHom(\uTh(\xi), \om\xi)\cong\RHom(\uC(w^{\b})_{\cL}, \uC(w^{\b})_{\cL}).
\end{equation*}
Taking cohomology we get an isomorphism of graded $R\ot R$-modules $\a: \uMM_{\xi}(\xi)\cong \om R(w^{\b})$, well-defined up to a scalar. We  normalize  this isomorphism by requiring that $\e(\xi)$ go to $1\in R(w^{\b})$. Since both $\e(\xi)$ and $1\in R(w^{\b})$ are invariant under $\Fr$, $\a$ is also $\Fr$-equivariant.
\end{proof}

\begin{lemma}\label{l:MM ICs} Let $s\in W$ be a simple reflection and $s\in \WL$. Recall we have a canonical isomorphism $\IC(s)_{\cL}\cong \IC(s)^{\da}_{\cL}$ given by Lemma \ref{l:ICda es}(2). 
\begin{enumerate} 
\item Let $\xi\in {}_{\cL'}\fP^{\b}_{\cL}$ for some block $\b\in{}_{\cL'}\un W_{\cL}$ and $\cF\in {}_{\cL'}\cD^{\b}_{\cL}$. There is a canonical isomorphism in $\RRF$
\begin{equation}\label{MFICs}
\MM_{\xi}(\cF)\ot_{R^{s}}R\j{1}\isom\MM_{\xi}(\cF\star\IC(s)^{\da}_{\cL})
\end{equation}
such that the composition
\begin{equation}\label{MFMF}
\xymatrix{\MM_{\xi}(\cF)\ar[r]^-{\eqref{MFICs}} & \MM_{\xi}(\cF\star\IC(s)^{\da}_{\cL}\j{-1})\ar[rr]^-{\textup{Lemma \ref{l:inert s}}}_-{\sim} && \MM_{\xi}(\pi^{*}_{s}\pi_{s*}\cF)\ar[r]^-{\textup{adj}} & \MM_{\xi}(\cF)
}
\end{equation}
is the identity.
\item There is a canonical isomorphism in $\RRF$
\begin{equation}\label{MICsRR}
\MM^{\c}(\IC(s)^{\da}_{\cL}\j{-1})\cong R\ot_{R^{s}}R
\end{equation}
under which $\th^{\da}_{s}$ corresponds to $1\ot1$.

\end{enumerate}
\end{lemma}
\begin{proof}
(1) Let $\wt\cL\in\Ch(L_{s})$ be the extension of $\cL$. By Lemma \ref{l:IC descend},  we can write $\Th(\xi)=\pi_{s}^{*}\ov\Th$ for some shifted perverse sheaf $\ov\Th\in {}_{\cL}\cD_{\wt\cL}$. By Lemma \ref{l:inert s}, we have
\begin{equation}\label{MF des Ps}
\Homb(\Th(\xi), \cF)\cong\Homb(\pi_{s}^{*}\ov\Th, \cF)\cong \Homb(\ov\Th, \pi_{s*}\cF)=\cohog{*}{(B\bs G/P_{s})_{k}, \RuHom(\ov\Th, \pi_{s*}\cF)}.
\end{equation}
The right side above is naturally a graded $(R\ot R^{s},\Fr)$-module, for $R^{s}=\upH^{*}_{(P_{s})_{k}}(\pt_{k})$. 

Let $\pi_{s}$ also denote the projection $B\bs G/B\to B\bs G/P_{s}$.  For any complex $\cK\in D^{b}_{m}(B\bs G/P_{s})$ the pullback $\cohog{*}{(B\bs G/P_{s})_{k}, \cK}\to \cohog{*}{(B\bs G/B)_{k}, \pi_{s}^{*}\cK}$ is right $R^{s}$-linear. It then induces a natural map in $\RRF$
\begin{equation}\label{pullPs}
\cohog{*}{(B\bs G/P_{s})_{k}, \cK}\ot_{R^{s}}R\to \cohog{*}{(B\bs G/B)_{k}, \pi_{s}^{*}\cK}.
\end{equation}
This is in fact a bijection, because 
\begin{equation*}
\cohog{*}{(B\bs G/B)_{k}, \pi_{s}^{*}\cK}\cong \cohog{*}{(B\bs G/P_{s})_{k}, \pi_{s*}\pi_{s}^{*}\cK}\cong \cohog{*}{(B\bs G/P_{s})_{k}, \cK\ot \pi_{s*}\Qlbar}
\end{equation*}
and $\pi_{s*}\Qlbar\cong \Qlbar\oplus\Qlbar\j{-2}$ (in $D_{m}^{b}(B\bs G/P_{s})$) corresponding to the decomposition $R=R^{s}\op \a_{s}R^{s}$.
Applying the isomorphism \eqref{pullPs} to $\cK=\RuHom(\ov\Th, \pi_{s*}\cF)$ we get
\begin{eqnarray*}
&&\cohog{*}{(B\bs G/P_{s})_{k}, \RuHom(\ov\Th, \pi_{s*}\cF)}\ot_{R^{s}}R \cong\cohog{*}{(B\bs G/B)_{k}, \pi_{s}^{*}\RuHom(\ov\Th, \pi_{s*}\cF)}\\
&\cong&\cohog{*}{(B\bs G/B)_{k},\RuHom(\pi_{s}^{*}\ov\Th, \pi_{s}^{*}\pi_{s*}\cF)}\cong \Homb(\Th(\xi), \cF\star\IC(s)^{\da}_{\cL}\j{-1}).
\end{eqnarray*}
Here we have used Lemma \ref{l:inert s}.
Combining this with \eqref{MF des Ps}, we get an isomorphism
\begin{equation*}
\MM_{\xi}(\cF)\ot_{R^{s}}R\j{1}=\Homb(\Th(\xi), \cF)\ot_{R^{s}}R\j{1}\cong \Homb(\Th(\xi), \cF\star\IC(s)^{\da}_{\cL})=\MM_{\xi}(\cF\star\IC(s)^{\da}_{\cL}).
\end{equation*}
The construction above shows that the composition \eqref{MFMF} is induced by applying $\Homb(\ov\Th,-)$ to the composition of adjunction maps $\pi_{s*}\cF\to \pi_{s*}\pi_{s}^{*}\pi_{s*}\cF\to \pi_{s*}\cF$, which is the identity.

(2) Taking $\cF=\d_{\cL}$ in (1), we get the canonical isomorphism \eqref{MICsRR}. The fact that $\th^{\da}_{s}$ corresponds to $1\ot 1$ follows from the fact that \eqref{MFMF} is the identity for $\cF=\d_{\cL}$.  This proves part (2).
\end{proof}

\subsection{Monoidal structure}\label{ss:monoidal}
Let $\b\in{}_{\cL'}\un W_{\cL}, \g\in{}_{\cL''}\un W_{\cL'}, \cF\in {}_{\cL'}\cD^{\b}_{\cL},\cG\in {}_{\cL''}\cD^{\g}_{\cL'},\xi\in{}_{\cL'}\fP^{\b}_{\cL}$ and $\y\in{}_{\cL''}\fP^{\g}_{\cL'}$. Consider the maps
\begin{eqnarray*}
\Hom(\Th(\y), \cG[i])\times\Hom(\Th(\xi), \cF[j])\xr{\star} \Hom(\Th(\y)\star\Th(\xi), \cG\star\cF[i+j])\\
\xr{(-)\c\ph(\y,\xi)}\Hom(\Th(\y\xi), \cG\star\cF[i+j]).
\end{eqnarray*}
Taking direct sum over $i,j\in\ZZ$ we get a pairing
\begin{equation*}
(\cdot,\cdot): \MM_{\y}(\cG)\times \MM_{\xi}(\cF)\to \MM_{\y\xi}(\cG\star\cF).
\end{equation*}
satisfying the following relations for $a\in R, f\in \MM_{\xi}(\cF)$ and $g\in \MM_{\y}(\cG)$
\begin{equation*}
((1\ot a)\cdot g,f)=(g,(a\ot 1)\cdot f),\quad ((a\ot 1)\cdot g,f)=(a\ot 1)\cdot (g,f), \quad (g,(1\ot a)\cdot f)=(1\ot a)\cdot (g,f).
\end{equation*}
Therefore it induces a map in $\RRF$
\begin{equation}\label{MM mono}
c_{\y,\xi}(\cG,\cF): \MM_{\y}(\cG)\ot_{R}\MM_{\xi}(\cF)\to \MM_{\y\xi}(\cG\star\cF).
\end{equation}
As $\cF$ and $\cG$ vary, the above maps form a natural transformation between two bifunctors ${}_{\cL''}\cD^{\g}_{\cL'}\times {}_{\cL'}\cD^{\b}_{\cL}\to \RRF$ defined by the left and right sides. The co-associativity of $\{\ph(\y,\xi)\}$  as shown in Proposition  \ref{p:coalg ext}(3) implies that the maps \eqref{MM mono} are associative for three composable $\xi,\y,\z$.

\begin{lemma}\label{l:M conv th} With the above notation,  $c_{\y,\xi}(\cG, \xi)$ is an isomorphism in $\RRF$. In particular (by Lemma \ref{l:MM th}) there is a canonical isomorphism
\begin{equation*}
\MM_{\y}(\cG)\ot_{R}R(w^{\b})\isom \MM_{\y\xi}(\cG\star\xi).
\end{equation*}
Similar statement holds when $\cG$ appears in the second factor.
\end{lemma}
\begin{proof}
By definition we have
\begin{eqnarray*}
\psi: && \MM_{\y\xi}(\cG\star\xi)=\Homb(\Th(\y\xi), \cG\star\xi)=\Homb(\Th(\y)\star\xi, \cG\star\xi)\\
&\cong& \Homb(\Th(\y), \cG)=\MM_{\y}(\cG)
\end{eqnarray*}
where we used the fact that $\star\xi$ is an equivalence (Proposition \ref{p:min th}).  The composition
\begin{equation*}
\MM_{\y}(\cG)\ot_{R}R(w^{\b})\cong\MM_{\y}(\cG)\ot_{R}\MM_{\xi}(\xi)\xr{c_{\y,\xi}(\cG,\xi)}\MM_{\y\xi}(\cG\star\xi)\xr{\psi} \MM_{\y}(\cG)
\end{equation*}
sends  $f\ot 1$ to $f$, hence it is an isomorphism. This implies that $c_{\y,\xi}(\cG,\xi)$ is an isomorphism.

For the statement where $\cG$ appears as the second factor, we use the canonical isomorphism $\Th(\xi)\cong \xi\star\Th^{\c}_{\cL}$ given in Lemma \ref{l:switch Th}. The rest of the argument is the same as above. 
\end{proof}

\begin{lemma}\label{l:M conv s} Let $s\in W$ be a simple reflection and $s\in \WL$. Let $\xi\in {}_{\cL'}\fP^{\b}_{\cL}$ for some block $\b\in{}_{\cL'}\un W_{\cL}$ and $\cF\in {}_{\cL'}\cD^{\b}_{\cL}$. Then the map $c_{\xi,\d_{\cL}}(\cF,\IC(s)_{\cL})$ is an isomorphism. 
\end{lemma}
\begin{proof}
Lemma \ref{l:MM ICs} already gives us an isomorphism
\begin{equation*}
\mu_{\cF,s}: \MM_{\xi}(\cF)\ot_{R}\MM^{\c}(\IC(s)_{\cL})\cong \MM_{\xi}(\cF)\ot_{R^{s}}R\j{1}\cong \MM_{\xi}(\cF\star\IC(s)_{\cL}).
\end{equation*}
It remains to show that $\mu_{\cF,s}$ is the same as $c_{\xi,\d_{\cL}}(\cF,\IC(s)_{\cL})$. To prove this, after a diagram chasing, it is enough to show that the following composition (we are using notation from the proof of Lemma \ref{l:MM ICs})
\begin{equation}\label{THS}
\pi^{*}_{s}\ov\Th=\Th(\xi)\xr{\ph(\xi,\d_{\cL})}\Th(\xi)\star\Th^{\c}_{\cL}\xr{\id\star \psi_{s}}\Th(\xi)\star\IC(s)_{\cL}\j{-1}\cong \pi_{s}^{*}\pi_{s*}\Th(\xi)=\pi_{s}^{*}\pi_{s*}\pi^{*}_{s}\ov\Th
\end{equation}
is the natural map given by the adjunction $\ov\Th\to \pi_{s*}\pi_{s}^{*}\ov\Th$.
By Proposition \ref{p:conv Th}(2)(3),  $\pi_{s}^{*}\pi_{s*}\Th(\xi)\cong \Th(\xi)\star\IC(s)_{\cL}\j{-1}$ lies in perverse degree $\ge0$, with $\om\pH^{0}\pi_{s}^{*}\pi_{s*}\Th(\xi)\cong \uTh(\xi)$, we see that $\Hom(\Th(\xi), \pi_{s}^{*}\pi_{s*}\Th(\xi))$ is one-dimensional. Therefore it suffices to show that the composition of \eqref{THS} with the adjunction $\pi_{s}^{*}\pi_{s*}\Th(\xi)\to\Th(\xi)$ is the identity map of $\Th(\xi)$. This boils down to the commutativity of the following diagram
\begin{equation*}
\xymatrix{ \Th(\xi)\ar[r]^-{\ph(\xi,\d_{\cL})}\ar@{=}[d]^{\id} & \Th(\xi)\star\Th^{\c}_{\cL}\ar[r]^-{\id\star\psi_{s}}\ar[d]^{\id\star\e_{\cL}} & \Th(\xi)\star\IC(s)_{\cL}\j{-1}\ar[dl]^{\id\star\e_{s}}          \\
\Th(\xi)\ar[r]^-{\sim} & \Th(\xi)\star\d_{\cL}}
\end{equation*}
Here $\e_{s}:\IC(s)_{\cL}\j{-1}\to \d_{\cL}$ is the map that induces the identity at $\dot e\in G$.  The left square is commutative by Proposition \ref{p:coalg ext}(2); the right triangle is commutative by the characterization of $\psi_{s}$. This finishes the proof.
\end{proof}

\begin{cor}\label{c:MM mono} The map $c_{\y,\xi}(\cG,\cF)$ in \eqref{MM mono} is an isomorphism if either $\cF\in{}_{\cL'}\cD^{\b}_{\cL}$ is a semisimple complex or $\cG\in {}_{\cL''}\cD^{\g}_{\cL'}$ is a semisimple complex.
\end{cor}
\begin{proof}
By symmetry we only need to treat the case $\cF$ semisimple, and it suffices to work with non-mixed complexes. Since any simple perverse sheaf $\uIC(w)_{\cL}$ is a direct summand of a successive convolution $\uIC(s_{i_{1}},\cdots, s_{i_{n}})_{\cL}:=\uIC(s_{i_{1}})_{\cL_{1}}\star\cdots\star\uIC(s_{i_{n}})_{\cL}$,  it suffices to prove the statement for $\cF=\uIC(s_{i_{1}},\cdots, s_{i_{n}})_{\cL}$ for any sequence of simple reflections $(s_{i_{1}},\cdots, s_{i_{n}})$ in $W$. But the latter case follows by successive application of either Lemma \ref{l:M conv th} or Lemma \ref{l:M conv s}.
\end{proof}

The next result is the main result of this section. It is a monodromic version of Soergel's \cite[Erweiterungssatz 17]{S}. For the non-monodromic Hecke categories, the Erweiterungssatz (Extension Theorem) of Soergel is a special case of a more general result of Ginzburg for varieties with $\Gm$-actions \cite{Ginz}. Our argument below is specific to the Hecke categories. 

\begin{theorem}\label{th:Hom} Let $\b\in{}_{\cL'}\un W_{\cL}, \xi\in{}_{\cL'}\fP^{\b}_{\cL}$ and let $\cF,\cG\in {}_{\cL'}\cD^{\b}_{\cL}$ be semisimple complexes. Then the natural map
\begin{equation*}
m(\cF,\cG): \Homb(\cF,\cG)\to \Homb_{\RRM}(\MM_{\xi}(\cF), \MM_{\xi}(\cG))
\end{equation*}
is an isomorphism in $\RRF$.
\end{theorem}
\begin{proof} Since $m(\cF,\cG)$ is $\Fr$-equivariant, if suffices to prove that $m(\cF,\cG)$ is an isomorphism in $\RRM$. Therefore we may assume $\cF,\cG\in {}_{\cL'}\un\cD^{\b}_{\cL}$.  In the rest of the argument we only consider the non-mixed Soergel functors, and we do not specify the non-mixed minimal IC sheaves defining them (the non-mixed minimal IC sheaves are unique up to isomorphism in each block); we simply write $\un\MM$ for the non-mixed Soergel functor.

Since every semisimple complex is a direct sum of shifts of $\uIC(w)_{\cL}$ (for $w\in\b$), it suffices to prove the above isomorphism for $\cF=\uIC(w)_{\cL}$. For a sequence $\un w=(s_{i_{1}},\cdots, s_{i_{n}})$ of simple reflections, write $\uIC(\un w)_{\cL}=\uIC(s_{i_{1}})_{s_{i_{2}}\cdots s_{i_{n}}\cL}\star\cdots\star\uIC(s_{i_{n}})_{\cL}$. By the decomposition theorem \cite{BBD},  every $\uIC(w)_{\cL}$ is a direct summand of $\uIC(\un w)_{\cL}$ for some  sequence  $\un w$, it suffices to treat the case $\cF=\uIC(\un w)_{\cL}$ for a sequence $\un w=(s_{i_{1}},\cdots, s_{i_{n}})$ of simple reflections, i.e., showing the following is an isomorphism
\begin{equation}\label{ww'}
\Homb(\uIC(\un w)_{\cL}, \cG)\to \Homb_{\RRM}( \un\MM(\uIC(\un w)_{\cL}), \un\MM(\cG)).
\end{equation}

We prove this by induction on the length of $\un w$ (and varying block $\b$ accordingly). If $\un w=\varnothing$, this means $\uIC(\un w)_{\cL}\cong \un\d_{\cL}$. This case will be treated in Lemma \ref{l:Hom from d}.

Now suppose \eqref{ww'} is an isomorphism for all $\un w$ of length $<n$. Consider a sequence $\un w=(s_{i_{1}}, \cdots, s_{i_{n}})$ of length $n$ and arbitrary semisimple complex $\cG$ in the same block as $\uIC(\un w)_{\cL}$. Let $\un w'=(s_{i_{1}}, \cdots, s_{i_{n-1}}), s=s_{i_{n}}$ then $\uIC(\un w)_{\cL}\cong\uIC(\un w')_{s\cL}\star\uIC(s)_{\cL}$. Consider the following diagram where each solid arrow is well-defined up to a nonzero scalar
\begin{equation}\label{HomIC}
\xymatrix{ \Homb(\uIC(\un w')_{s\cL}\star\uIC(s)_{\cL}, \cG)\ar[r]^-{a}\ar[d]^{m} &      \Homb(\uIC(\un w')_{s\cL}, \cG\star\uIC(s)_{s\cL})\ar[d]^{m'}      \\
\Homb(\un\MM(\uIC(\un w')_{s\cL}\star\uIC(s)_{\cL}), \un\MM(\cG))\ar[d]^{u} &      \Homb(\un\MM(\uIC(\un w')_{s\cL}), \un\MM(\cG\star\uIC(s)_{s\cL}))\ar[d]^{u'}    \\
\Homb(\un\MM(\uIC(\un w')_{s\cL})\ot_{R}\un\MM(\uIC(s)_{\cL}), \un\MM(\cG))\ar@{-->}[r]^-{b} &      \Homb(\un\MM(\uIC(\un w')_{s\cL}), \un\MM(\cG)\ot_{R}\un\MM(\uIC(s)_{s\cL}))   
}
\end{equation}
Here the map $a$ is the adjunction isomorphism either from Lemma \ref{l:clean s equiv} if $s\notin \WL$ or as in Corollary \ref{c:adj ICs} if $s\in \WL$. The maps $m$ and $m'$ are given by the functor $\un\MM$, and $m'$ is an isomorphism by  inductive hypothesis for $\un w'$.  The isomorphisms $u$ and $u'$ are induced by the monoidal structure of $\un\MM$ proved in Corollary \ref{c:MM mono}. Therefore, to show that $m$ is an isomorphism, it suffices to construct the dotted arrow $b$ which is an isomorphism and makes the diagram commutative up to a nonzero scalar.

If $s\notin \WL$, using Lemma \ref{l:M conv th} we have 
\begin{eqnarray*}
&&\Homb(\un\MM(\uIC(\un w')_{s\cL})\ot_{R}\un\MM(\uIC(s)_{\cL}), \un\MM(\cG))\cong \Homb(\un\MM(\uIC(\un w')_{s\cL})\ot_{R}R(s),\un\MM(\cG))  \\
&\cong&   \Homb(\un\MM(\uIC(\un w')_{s\cL}),\un\MM(\cG)\ot_{R}R(s)) \cong \Homb(\un\MM(\uIC(\un w')_{s\cL}), \un\MM(\cG)\ot_{R}\un\MM(\uIC(s)_{s\cL})). 
\end{eqnarray*}
Let $b$ be the composition of the above isomorphisms. It is easy to check that $b$ makes  \eqref{HomIC} commutative, hence $m$ is an isomorphism.

If $s\in \WL$, by Lemma \ref{l:M conv s} it suffices to construct an isomorphism
\begin{eqnarray*}
b': \Homb(\un\MM(\uIC(\un w')_{\cL}\ot_{R^{s}}R\j{1}, \un\MM(\cG))\isom \Homb(\un\MM(\uIC(\un w')_{\cL}), \un\MM(\cG)\ot_{R^{s}}R\j{1}).
\end{eqnarray*}
By \cite[Proposition 5.10(2)]{Sb}, for $M_{1}, M_{2}\in \RRM$, there is a bifunctorial isomorphism of $R$-bimodules
\begin{equation}\label{adj for bimod}
\Homb_{\RRM}(M_{1}, M_{2}\ot_{R^{s}}R\j{1})
\cong \Homb_{\RRM}(M_{1}\ot_{R^{s}}R\j{1}, M_{2}).
\end{equation}
Indeed, since $R=R^{s}\op \a_{s} R^{s}$ (note $\a_{s}^{2}\in R^{s}$), we may identify $R$ with $R^{s}\j{1}\op R^{s}\j{-1}$ as graded $R^{s}$-modules. For an $R$-bimodule map $f: M_{1}\to M_{2}\ot_{R^{s}}R\j{1}=M_{2}\j{1}\op  M_{2}\j{-1}$, we write $f(x)=(f_{-1}(x),f_{1}(x))$, where $f_{\pm1}: M_{1}\to M_{2}\j{\pm1}$ is $R\ot R^{s}$-linear. Then $f\mapsto f_{1}$ gives an isomorphism $\Homb_{\RRM}(M_{1}, M_{2}\ot_{R^{s}}R\j{1})\cong\Homb_{R\textup{-Mod-}R^{s}}(M_{1}, M_{2}\j{-1})$, with inverse $f_{1}\mapsto (f:x\mapsto (f_{1}(x\a_{s}), f_{1}(x)))$. On the other hand, we also have $\Homb_{\RRM}(M_{1}\ot_{R^{s}}R\j{1}, M_{2})\cong \Homb_{R\textup{-Mod-}R^{s}}(M_{1}\j{1}, M_{2})\cong \Homb_{R\textup{-Mod-}R^{s}}(M_{1}, M_{2}\j{-1})$ by the adjunction between tensor and forgetful functors. Combining these isomorphisms we get the desired isomorphism \eqref{adj for bimod}. Moreover, \eqref{adj for bimod} is compatible with the adjunction in Corollary \ref{c:adj ICs} under the isomorphisms in Lemma \ref{l:MM ICs}. The isomorphism \eqref{adj for bimod} gives the desired isomorphism $b'$, hence $b$ that makes the diagram \eqref{HomIC} commutative up to a nonzero scalar. Therefore $m$ is again an isomorphism in this case. This finishes the proof.
\end{proof}

\begin{lemma}\label{l:Hom from d} For any semisimple complex $\cG\in {}_{\cL}\un\cD^{\c}_{\cL}$, the natural map
\begin{equation}\label{Fd}
m(\un\d_{\cL},\cG):\Homb(\un\d_{\cL},\cG)\to \Homb_{\RRM}(R(e), \un\MM^{\c}(\cG))
\end{equation}
is an isomorphism of graded $R$-bimodules. 
\end{lemma}
\begin{proof}
Recall the adjunction $i_{e*}: {}_{\cL}\un\cD(e)_{\cL}\leftrightarrow {}_{\cL}\un\cD^{\c}_{\cL}: i_{e}^{!}$. The adjunction map $i_{e*}i^{!}_{e}\cG\to \cG$ gives a commutative diagram
\begin{equation*}
\xymatrix{\Homb(\un\d_{\cL}, i_{e*}i^{!}_{e}\cG)\ar[d]^{a}\ar[rr]^-{m(\un\d_{\cL}, i_{e*}i^{!}_{e}\cG)} & & \Homb_{\RRM}(R, \un\MM^{\c}(i_{e*}i^{!}_{e}\cG))\ar[d]^{b}\\
\Homb(\un\d_{\cL},\cG)\ar[rr]^-{m(\un\d_{\cL},\cG)} & & \Homb_{\RRM}(R,\un\MM^{\c}(\cG))
}
\end{equation*}
We will show that $m(\un\d_{\cL},\cG)$ is an isomorphism by showing that the other three arrows in the above diagram are isomorphisms. Here the arrows $a$ and $b$ are induced by the adjunction map $i_{e*}i^{!}_{e}\cG\to \cG$ and $a$ is an isomorphism by adjunction. 

We show that $m(\un\d_{\cL}, i_{e*}i^{!}_{e}\cG)$ is also an isomorphism. Indeed, by Proposition \ref{p:parity purity}(2), $i^{!}_{e}\cG$ is a direct sum of shifts of $\uC(e)_{\cL}$, it suffices to treat the case where $i^{!}_{e}\cG$ is replaced by $\uC(e)_{\cL}$, or equivalently replacing $i_{e*}i^{!}_{e}\cG$ with $\un\d_{\cL}$, in which case both sides are identified with the regular $R$-bimodule $R(e)=R$.

Finally we show that the arrow $b$ is an isomorphism. Using the filtration $F_{\le w}\un\MM^{\c}(\cG)$ introduced in Corollary \ref{c:filM}, we have $F_{\le e}\un\MM^{\c}(\cG)=\Homb(\un\Th^{\c}_{\cL}, i_{e*}i^{!}_{e}\cG)=\un\MM^{\c}(i_{e*}i^{!}_{e}\cG)$, which implies that $b$ is injective.

To see $b$ is surjective, we argue that any $R\ot R$-linear map $\ph: R(e)\to \un\MM^{\c}(\cG)$ must land in $F_{\le e}\un\MM^{\c}(\cG)$. Extend the partial order on $\WL$ to a total order, and suppose $w\in \WL$ is the smallest element under this total order such that $\ph(R(e))\subset F_{\le w}\un\MM^{\c}(\cG)$. Then the projection $R\xr{\ph} F_{\le w}\un\MM^{\c}(\cG)\to \Gr^{F}_{w}\un\MM^{\c}(\cG)$ must be nonzero for otherwise $\ph(R(e))$ would land in the previous step of the filtration. However, by \eqref{GrFw} and Proposition \ref{p:parity purity}(2), $\Gr^{F}_{w}\un\MM^{\c}(\cG)$ is a free $R(w)$-module. For there to exist a nonzero $R\ot R$-linear map $R(e)\to R(w)$, we must have $w=e$, which implies that $\ph(R(e))\subset F_{\le e}\un\MM^{\c}(\cG)$, as desired. This finishes the proof of the Lemma.
\end{proof}


\section{Soergel bimodules}
After reviewing basics about Soergel bimodules, the main result of this section is Proposition \ref{p:Sw} that connects simple perverse sheaves in the monodromic Hecke category with indecomposable Soergel bimodules via the Soergel functor introduced in the previous section. 

\subsection{Soergel bimodules}\label{ss:SBim}
Consider a finite Weyl group $(W_{0},S_{0})$ with reflection representation $V$ over $\Qlbar$.  Let $R=\Sym(V^{*})$ (graded with $V^{*}$ in degree $1$). We recall the notion of Soergel $R$-bimodules. 

For any sequence $(s_{i_{n}},\cdots, s_{i_{1}})$ of simple reflections, we have the {\em Bott--Samelson bimodule} $\SS(s_{i_{n}},\cdots, s_{i_{1}}):=R\ot_{R^{s_{i_{n}}}}R\ot_{R^{s_{i_{n-1}}}}\ot\cdots\ot_{R^{s_{i_{1}}}}R$. 

The {\em indecomposable  Soergel bimodules} are, up to degree shifts, indecomposable direct summands of $\SS(s_{i_{n}},\cdots, s_{i_{1}})$ for some sequence $(s_{i_{n}},\cdots, s_{i_{1}})$ of simple reflections in $W_{0}$.  A Soergel bimodule is a direct sum of indecomposable Soergel bimodules. Let $\SB(W_{0})\subset \RRM$ be the full subcategory consisting of the Soergel bimodules. Then $\SB(W_{0})$ carries a monoidal structure given by  the tensor product $(-)\ot_{R}(-)$.

Soergel \cite{Sb} shows that, for each $w\in W_{0}$, there is an indecomposable Soergel bimodule $\SS(w)$ characterized (up to isomorphism), among graded $R\ot R$-modules, by the following two properties.
\begin{enumerate}
\item $\Supp(\SS(w))\subset V\times V$ contains $\G(w)=\{(wx,x)|x\in V\}$, the graph of the $w$ action on $V$.
\item For some (equivalently any) reduced expression $w=s_{i_{n}}\cdots s_{i_{1}}$ in $W_{0}$, $\SS(w)$ is a direct summand of the Bott--Samelson bimodule $\SS(s_{i_{n}},\cdots, s_{i_{1}})$.
\end{enumerate}
To emphasize the dependence on the Coxeter group $W_{0}$, we denote $\SS(w)$ also by $\SS(w)_{W_{0}}$.

\subsection{Rigidified Soergel bimodules}
It is easy to see that the degree zero part of $\SS(w)$ is one-dimensional. Moreover,  the endomorphism ring of $\SS(w)$ inside $\RRM$ consists of scalars. Indeed in our case $W_{0}$ is a Weyl group, so we may interpret $\SS(w)$ as the equivariant intersection cohomology of an IC sheaf of a Schubert variety in a flag variety, and deduce the statement about the endomorphism ring from it (see \cite[Lemma 19 and Eweiterungssatz 17]{S}).  For fixed $w\in W_{0}$, consider a pair $(M,\one_{M})$ where $M\in \SB(W_{0})$ is isomorphic to $\SS(w)$, and $\one_{M}\in M^{0}$ is any nonzero element. Then the automorphism group of such a pair is trivial, and any two such pairs are isomorphic by a unique isomorphism. Therefore we may identify all such pairs with one pair, and denote it by $(\SS(w),\one)$.

\subsection{Extended Soergel bimodules} For $\cL\in\fo$ and $w\in W$ we define a graded $R\ot R$-module $\SS(w)_{\cL}$ as follows. Let $\b\in {}_{w\cL}\un W_{\cL}$ be the block containing $w$. Write $w=xw^{\b}$ for $x\in W_{w\cL}^{\c}$. Then we define
\begin{equation*}
\SS(w)_{\cL}:=\SS(x)_{W_{w\cL}^{\c}}\ot_{R}R(w^{\b}).
\end{equation*}
Again we can rigidify $\SS(w)_{\cL}$ by equipping it with the degree zero element $\one\ot 1$.

We also define a generalization of Bott-Samelson modules. For a sequence  $(s_{i_{n}},\cdots, s_{i_{1}})$ of simple reflections in $W$, and $\cL\in\fo$, let $\cL_{j}=s_{i_{j}}\cdots s_{i_{1}}\cL$. Define 
\begin{equation*}
\SS(s_{i_{n}},\cdots, s_{i_{1}}):=\SS(s_{i_{n}})_{\cL_{n-1}}\ot_{R}\SS(s_{i_{n-1}})_{\cL_{n-2}}\ot_{R}\ot\cdots\ot_{R}\SS(s_{i_{1}})_{\cL}.
\end{equation*}
Note that $\SS(s_{i_{j}})_{\cL_{j-1}}\cong R(s_{i_{j}})$ if $s_{i_{j}}\notin W_{\cL_{j}}^{\c}$ and is otherwise isomorphic to $R\ot_{R^{s_{i_{j}}}}R$.

\begin{lemma}\label{l:BSL}
Let $(s_{i_{n}},\cdots, s_{i_{1}})$ be a reduced word of simple reflections in $W$ and $\cL\in\fo$, $\cL_{j}=s_{i_{j}}\cdots s_{i_{1}}\cL$ for $1\le j\le n$. Let $\b\in {}_{\cL_{n}}\un W_{\cL}$ be the block containing $w=s_{i_{n}}\cdots s_{i_{1}}$. Then there is a reduced word $(t_{m},\cdots, t_{1})$ of simple reflections in the Coxeter group $W_{\cL_{n}}^{\c}$ such that $w=t_{m}t_{m-1}\cdots t_{1}w^{\b}$ and
\begin{equation*}
\SS(s_{i_{n}},\cdots, s_{i_{1}})_{\cL}\cong\SS(t_{m},\cdots, t_{1})_{W_{\cL_{n}}^{\c}}\ot_{R}R(w^{\b})
\end{equation*}
as graded $R\ot R$-modules.
\end{lemma}
\begin{proof}
We prove the lemma by induction on $n$. For $n=1$ and $s=s_{i_{1}}\notin \WL$, then $s=w^{\b}$ and $\SS(s)_{\cL}\cong R(w^{\b})$ (corresponding to $m=0$). For $n=1$ and $s=s_{i_{1}}\in \WL$, we have $w^{\b}=1$ and $s$ is a simple reflection in $\WL$, and $\SS(s)_{\cL}\cong \SS(s)_{\WL}$. 

Now suppose the statement is proved for reduced words of length less than $n$ ($n\ge2$). Let $\b'\in {}_{\cL_{n}}\un W_{\cL_{1}}$ be the block containing $w'=s_{i_{n}}\cdots s_{i_{2}}$. By inductive hypothesis, there is a reduced word $(t_{m},\cdots, t_{1})$ in $W^{\c}_{\cL_{n}}$ such that $t_{m}\cdots t_{1}w^{\b'}=w'$
\begin{equation*}
\SS(s_{i_{n}},\cdots,s_{i_{2}})_{\cL_{1}}\cong \SS(t_{m},\cdots, t_{1})_{W^{\c}_{\cL_{n}}}\ot_{R}R(w^{\b'}).
\end{equation*}
Write $s=s_{i_{1}}$. If $s\notin\WL$, then $w^{\b}=w^{\b'}s$, and we have
\begin{eqnarray*}
&&\SS(s_{i_{n}},\cdots,s_{i_{1}})_{\cL}\cong\SS(s_{i_{n}},\cdots,s_{i_{2}})_{\cL_{1}}\ot_{R}\SS(s)_{\cL}\cong\SS(t_{m},\cdots, t_{1})_{W^{\c}_{\cL_{n}}}\ot_{R}R(w^{\b'})\ot_{R}R(s)\\
&\cong&\SS(t_{m},\cdots, t_{1})_{W^{\c}_{\cL_{n}}}\ot_{R}R(w^{\b'}s)=\SS(t_{m},\cdots, t_{1})_{W^{\c}_{\cL_{n}}}\ot_{R}R(w^{\b}).
\end{eqnarray*}
We have $t_{m}\cdots t_{1}w^{\b}=t_{m}\cdots t_{1}w^{\b'}s=w's=w$. 

If $s\in\WL$, then $\cL_{1}=\cL, \b'=\b$. Moreover, $t=w^{\b}sw^{\b,-1}$ is a simple reflection in $W_{\cL_{n}}^{\c}$ by Corollary \ref{c:conj wb}. Hence
\begin{eqnarray*}
&&\SS(s_{i_{n}},\cdots,s_{i_{1}})_{\cL}\cong\SS(s_{i_{n}},\cdots,s_{i_{2}})_{\cL_{1}}\ot_{R^{s}}R\cong \SS(t_{m},\cdots, t_{1})_{W^{\c}_{\cL_{n}}}\ot_{R}R(w^{\b})\ot_{R^{s}}R\\
&\cong& \SS(t_{m},\cdots, t_{1})_{W^{\c}_{\cL_{n}}}\ot_{R^{t}}R(w^{\b})=\SS(t_{m},\cdots, t_{1},t)_{W^{\c}_{\cL_{n}}}.
\end{eqnarray*}
Here we have used  $\SS(s)_{\cL}=R\ot_{R^{s}}R$ and $R(w^{\b})\ot_{R^{s}}R\cong R\ot_{R^{t}}R(w^{\b})$. We have $t_{m}\cdots t_{1}tw^{\b}=t_{m}\cdots t_{1}w^{\b}s=w's=w$. Since $\ell_{\b}(w)=\ell_{\b'}(w')+1$ by Lemma \ref{l:ell beta}(4), $(t_{m},\cdots, t_{1},t)$ is a reduced word for $ww^{\b,-1}$. This completes the inductive step.
\end{proof}

We have the following characterization for $\SS(w)_{\cL}$.
\begin{lemma}\label{l:crit Sw}
Let $\cL\in\fo$ and $w\in W$. 
Let $M$ be an indecomposable graded $R\ot R$-module such that
\begin{enumerate}
\item $\Supp(M)\supset \G(w)$ as a subset of $\Spec(R\ot R)=V\times V$.
\item For some reduced expression $w=s_{i_{n}}s_{i_{n-1}}\cdots s_{i_{1}}$ in $W$, $M$ is a direct summand of $\SS(s_{i_{n}},\cdots, s_{i_{1}})_{\cL}$.
\end{enumerate}
Then $M\cong \SS(w)_{\cL}$.
\end{lemma}
\begin{proof}
Let $\b\in{}_{w\cL}\un W_{\cL}$ be the block containing $w$. Write $w=xw^{\b}$ for $x\in W_{w\cL}^{\c}$. Let $M'=M\ot_{R}R(w^{\b,-1})$. Then $M'$ is an indecomposable $R\ot R$-module whose support contains $\G(ww^{\b,-1})=\G(x)$. By Lemma \ref{l:BSL}, $\SS(s_{i_{n}},\cdots, s_{i_{1}})_{\cL}\cong \SS(t_{m},\cdots, t_{1})_{W_{\cL_{n}}^{\c}}\ot_{R} R(w^{\b})$ for a reduced expression $t_{m}\cdots t_{1}$ of $x=ww^{\b,-1}$ in $W_{\cL_{n}}^{\c}$. Therefore, $M'$ is a direct summand of $\SS(t_{m},\cdots, t_{1})_{W_{\cL_{n}}^{\c}}$.  By Soergel's criterion in \S\ref{ss:SBim}, $M'\cong \SS(x)_{W_{\cL_{n}}^{\c}}$. Hence $M=M'\ot_{R}R(w^{\b})\cong \SS(x)_{W_{\cL_{n}}^{\c}}\ot_{R}R(w^{\b})=\SS(w)_{\cL}$.
\end{proof}

\subsection{Soergel bimodules with Frobenius action}
Let $\RRF_{\pure}$ be the full subcategory of $\RRF$ consisting of those $M=\oplus_{n}M^{n}$ such that $M^{n}$ is pure of weight $n$ as a $\Fr$-module (\S\ref{sss:Fr}). Forgetting the Frobenius action gives a functor
\begin{equation*}
\om: \RRF_{\pure}\to\RRM.
\end{equation*}
This functor admits a one-side inverse 
\begin{equation*}
(-)^{\na}: \RRM\to \RRF_{\pure}
\end{equation*}
that sends a graded $R\ot R$-module  $M=\oplus_{n}M^{n}$  to the same graded $R\ot R$-module $M$ with $\Fr$ acting on $M^{n}$ by $q^{n/2}$. 

Let $\SB_{m}(\WL)\subset\RRF_{\pure}$ be the preimage of $\SB(\WL)$ under $\om$, i.e., it is the full subcategory consisting of $M\in \RRF_{\pure}$ such that $\om M\in \SB(\WL)$. Then $\SB_{m}(\WL)$ also carries a monoidal structure given by $(-)\ot_{R}(-)$.

\begin{prop}\label{p:Sw} Let $\cL\in\fo$ and $w\in \WL$. Then there is a unique isomorphism in $\RRF$
\begin{equation}\label{MICw}
\MM^{\c}(\IC(w)^{\da}_{\cL}\j{-\ell_{\cL}(w)})\cong \SS(w)^{\na}_{\WL}
\end{equation}
under which $\th_{w}^{\da}$ corresponds to $\one\in\SS(w)^{\na}_{\WL}$.
\end{prop}
\begin{proof}
We first prove more generally for any $w\in W$, we have an isomorphism in $\RRM$
\begin{equation}\label{MwSw}
\un\MM(\uIC(w)_{\cL}[-\ell_{\b}(w)])\cong \SS(w)_{\cL}.
\end{equation}
Here $\b\in {}_{w\cL}\un W_{\cL}$ is the block containing $w$, and we are suppressing the choice of a minimal IC sheaf $\xi\in {}_{w\cL}\fP^{\b}_{\cL}$ from $\un\MM_{\xi}$ because the isomorphism class of the functor $\un\MM_{\xi}$ is independent of $\xi$. To show \eqref{MwSw} we apply the criterion in Lemma \ref{l:crit Sw} to $M=\un\MM(\uIC(w)_{\cL}[-\ell_{\b}(w)])$. By Theorem \ref{th:Hom}, $\End(M)=\End(\uIC(w)_{\cL})=\Qlbar$, hence $M$ is indecomposable. By Corollary \ref{c:filM}, $M$ admits a filtration indexed by $\{v\in W; v\le w\}$ with the last associated graded $\Gr^{F}_{w}M\cong\Homb(i^{*}_{w}\uIC(w_{\b})[-N_{\cL}], i^{!}_{w}\uIC(w)_{\cL}[-\ell_{\b}(w)])$, which by Proposition \ref{p:stalk Th} is $\Homb(\uC(w)[-\ell_{\b}(w)], \uC(w)[-\ell_{\b}(w)])\cong R(w)$. Therefore $\Supp(M)\supset\Supp(R(w))=\G(w)$. Finally, for any reduced expression $w=s_{i_{n}}\cdots s_{i_{1}}$, $\uIC(w)_{\cL}$ is a direct summand of $\uIC(s_{i_{n}},\cdots, s_{i_{1}})_{\cL}=\uIC(s_{i_{n}})_{\cL_{n-1}}\star\cdots\star\uIC(s_{i_{1}})_{\cL}$ by the decomposition theorem. Therefore $M$ is a direct summand of $\un\MM(\uIC(s_{i_{n}},\cdots, s_{i_{1}})_{\cL}[-\ell_{\b}(w)])$. By repeated applications of  Lemma \ref{l:MM th} and Lemma \ref{l:MM ICs}, one sees that 
\begin{equation*}
\un\MM(\uIC(s_{i_{n}},\cdots, s_{i_{1}})_{\cL}[-\ell_{\b}(w)])\cong \SS(s_{i_{n}},\cdots, s_{i_{1}})_{\cL}.
\end{equation*}
The shift by $\ell_{\b}(w)$ matches the number of $1\le j\le n$  such that $s_{i_{j}}\in W^{\c}_{\cL_{j-1}}$ by Lemma \ref{l:ell beta}(4), which enters into the above calculation because of the shift $[1]$ that appears in Lemma \ref{l:MM ICs}(1). The above checks the conditions in Lemma \ref{l:crit Sw} and hence \eqref{MwSw} is proved.

Now consider the case $w\in \WL$ and let $M^{\da}=\MM^{\c}(\IC(w)^{\da}_{\cL}\j{-\ell_{\cL}(w)})$. We have already proved that $\om M^{\da}\cong\SS(w)_{\WL}$ as a graded $R\ot R$-module. Now the Frobenius action on $\SS(w)_{\WL}$ compatible with the grading and the $R\ot R$-action is unique up to a scalar  (proof: if $F$ and $F'$ are two such Frobenius actions on $\SS(w)_{\WL}$, then $F'\c F^{-1}\in\Aut_{\RRM}(\SS(w)_{\WL})\cong\Aut(\uIC(w)_{\cL})=\Qlbar^{\times}$). Therefore, $M^{\da}\cong \SS(w)^{\na}_{\WL}\ot V$ for some one-dimensional $\Fr$-module $V$.   In particular, we have an identification of $\Fr$-modules $(M^{\da})^{0}\cong V$. Now $0\ne\th^{\da}_{w}\in (M^{\da})^{0}$ is $\Fr$-invariant, hence $V$ is a trivial $\Fr$-module. Therefore $M^{\da}\cong \SS(w)^{\na}_{\WL}$; such an isomorphism is unique up to a scalar, and it becomes unique if we require $\th^{\da}_{w}$ to go to $\one$. 
\end{proof}
\subsection{More on Soergel bimodules} The rest of the section is only used in the proof of Prop. \ref{p:equiv CL}. Let $\cL\in\fo$ and consider Soergel bimodules for $\WL$. In the rest of the section we shall denote $\SS(w)_{\WL}$ simply by $\SS(w)$. To each indecomposable Soergel bimodule $S\cong\SS(w)[n]$ we assign the integer $d(S):=-n+\ell_{\cL}(w)$.  This is the analogue of the perverse degree for Soergel bimodules.

\begin{lemma}\label{l:Sdeg} Let $S,S'$ be indecomposable Soergel bimodules for $\WL$.
\begin{enumerate}
\item If $d(S)<d(S')$ then $\Hom_{\RRM}(S,S')=0$.
\item If $d(S)=d(S')$ and $S$ and $S'$ are not isomorphic, then $\Hom_{\RRM}(S,S')=0$.
\end{enumerate}
\end{lemma}
\begin{proof}
If $S=\SS(w)[n]$ and $S'=\SS(w')[n']$, then by Proposition \ref{p:Sw}, $\un\MM^{\c}(\uIC(w)_{\cL}[n-\ell_{\cL}(w)])\cong S$,  $\un\MM^{\c}(\uIC(w')_{\cL}[n'-\ell_{\cL}(w')])\cong S'$. By Theorem \ref{th:Hom}, $\Hom_{\RRM}(S,S')=\Hom(\uIC(w)[n-\ell_{\cL}(w)], \uIC(w')[n'-\ell_{\cL}(w')])=\Hom(\uIC(w)[-d(S)],\uIC(w')[-d(S')])$. 

If $d(S)<d(S')$, then $\Hom(\uIC(w)[-d(S)],\uIC(w')[-d(S')])=0$ by perverse degree reasons, therefore $\Hom_{\RRM}(S,S')=0$. 

If $d(S)=d(S')$, then we have $\Hom(\uIC(w)[-d(S)],\uIC(w')[-d(S')])=\Hom(\uIC(w),\uIC(w'))$ which vanishes if $w\ne w'$, in which case $\Hom_{\RRM}(S,S')=0$.
\end{proof}

\begin{prop}\label{p:filM}
Let $M\in \SB_{m}(\WL)$.  There exists a finite filtration $0=F_{0}M\subset F_{1}M
\subset\cdots\subset F_{n}M=M$ by subobjects in $\SB_{m}(\WL)$ with the following properties 
\begin{enumerate}
\item For $1\le i\le n$, $\Gr^{F}_{i}M\cong \SS(w_{i})^{\na}\j{n_{i}}\ot V_{i}$ for some $w_{i}\in \WL, n_{i}\in\ZZ$  and finite-dimensional $\Fr$-module $V_{i}$ pure of weight zero.
\item The filtration $\om F_{\bu}M$ of $\om M$ splits in $\RRM$.
\end{enumerate}
\end{prop}
\begin{proof}
Let $M=\oplus_{n}M^{n}\in \SB_{m}(\WL)$. Since $M$ is finitely generated as a graded $R\ot R$-module, each $M^{n}$ is finite-dimensional, and may be decomposed into generalized eigenspaces of $\Fr$. We group the generalized Frobenius eigenvalues according to the cosets $\Qlbar^{\times}/q^{\ZZ}$
\begin{equation*}
M=\oplus_{\l\in \Qlbar^{\times}/q^{\ZZ}}M_{\l}.
\end{equation*} 
Since $\Fr$ acts on $R\ot R$ by integer powers of $q$,   each $M_{\l}$ is itself an object in $\RRF$; since $\om M_{\l}$ is a direct summand of a Soergel bimodule, it is also a Soergel bimodule. Hence  $M_{\l}\in \SB_{m}(\WL)$. We only need to produce a filtration for each $M_{\l}$. Without loss of generality, we consider the case $\l=1$ and assume $M=M_{1}$, i.e., $\Fr$-eigenvalues on $M$ are in $q^{\ZZ}$. In particular, $M$ is evenly graded. 

Consider any decomposition $\om M=\oplus_{\a\in I} S_{\a}$ where each $S_{\a}$ is an indecomposable Soergel bimodule.  Let $F_{i}M=\oplus_{\a\in I, d(S_{\a})\le i}S_{\a}\subset M$. By Lemma \ref{l:Sdeg}(1), $F_{i}M$ is independent of the decomposition of $\om M$ into indecomposables. Therefore each $F_{i}M$ is stable under $\Fr$, and hence an object in $\SB_{m}(\WL)$. Moreover, by Lemma \ref{l:Sdeg}(2), we can canonically write $\om\Gr^{F}_{i}M=\oplus_{w\in \WL}\SS(w)[\ell_{\cL}(w)-i]\ot {}_{i}V_{w}$ where ${}_{i}V_{w}=\Hom_{\RRM}(\SS(w)[\ell_{\cL}(w)-i], \om\Gr^{F}_{i}M)$.   Equip ${}_{i}V_{w}$ with the Frobenius action by viewing it as $\Hom_{\RRM}(\SS(w)^{\na}\j{\ell_{\cL}(w)-i},\Gr^{F}_{i}M)$, then $\Gr^{F}_{i}M\cong \oplus_{w\in \WL}\SS(w)^{\na}\j{\ell_{\cL}(w)-i}\ot {}_{i}V_{w}$ as objects in $\RRF$. After refining the filtration $F_{\bu}M$ and renumbering, it becomes a filtration satisfying the required conditions.
\end{proof}

\section{Equivalence for the neutral block}\label{s:endo}
In this section we prove Theorem \ref{th:intro neutral}.

\subsection{The endoscopic group}\label{ss:endo} Let $\cL\in\fo$ and consider the neutral block $_{\cL}\cD^{\c}_{\cL}$. Let $H$ be the reductive group over $\FF_{q}$ with a maximal torus identified with $T$ and the root system $\Phi(H,T)=\Phi_{\cL}\subset \xch(T)$. In particular, the Weyl group of $H$ with respect to $T$ is identified with $\WL$.  We call $H$ the {\em endoscopic group of $G$ corresponding to $\cL$}. 
Let $B_{H}\subset H$ be the Borel subgroup containing $T$ corresponding to the positive roots $\Phi^{+}_{\cL}$. As defined $H$ is unique up to non-unique isomorphisms. We will give a rigidification of $H$ later in \S\ref{ss:rel pin}.

To justify the terminology, we recall the usual definition of endoscopic groups of $G$.  Let $\dG$ be the Langlands dual group of $G$ defined over $\Qlbar$, with a maximal torus $\dT$ and roots $\Phi(\dG,\dT)\subset\xch(\dT)=\xcoch(T)$ identified with the coroots $\Phi^{\vee}(G,T)$ of $(G,T)$.  Let $\k$ be a semisimple element in $\dG$, and $\dH$ be the neutral component of the centralizer $\dG_{\k}$. An endoscopic group of $G$ associated to $\k$ is a reductive group over $\FF_{q}$ whose Langlands dual is isomorphic to $\dH$.

Now let $\dH\subset \dG$ be the connected reductive subgroup containing $\dT$ with roots $\Phi(\dH,\dT)=\Phi^{\vee}_{\cL}\subset \Phi^{\vee}(G,T)=\Phi(\dG,\dT)$. Then $\dH$ is dual to $H$. We claim that $\dH$ is the neutral component of the centralizer in $\dG$ of a semisimple element $\k\in\dT$, which would imply that $H$ is an endoscopic group of $G$ in the usual sense. In fact, the bijection \eqref{cs1T} allows us to identify $\Ch(T)$ with $\xch(T)\ot_{\ZZ} \Hom(\FF_{q}^{\times},\Qlbar^{\times})$. Choosing a generator $\z\in \FF_{q}^{\times}$, we get an isomorphism $\Hom(\FF_{q}^{\times},\Qlbar^{\times})\cong \mu_{q-1}(\Qlbar)$ by evaluating at $\z$, hence $\Ch(T)\isom \xch(T)\ot_{\ZZ}\mu_{q-1}(\Qlbar)=\dT[q-1](\Qlbar)$. This allows us to turn $\cL\in\Ch(T)$ into an element $\k\in \dT(\Qlbar)$ such that $\k^{q-1}=1$. Then we have $\dH=\dG_{\kappa}^{\c}$ as subgroups of $\dG$.

With the correspondence $\cL\bij\k\in\dT[q-1](\Qlbar)$ above,   $W_{\cL}$  is identified with the Weyl group of the possibly disconnected group $\dG_{\kappa}$ with respect to $\dT$, i.e., $W_{\cL}\cong N_{\dG_{\kappa}}(\dT)/\dT$. If $G$ has connected center, so that $\dG$ has simply-connected derived group,  $\dG_{\kappa}$  is connected by Steinberg \cite[Theorem 8.1]{St}, hence $W_{\cL}=\WL$ in this case (see also \cite[Theorem 5.13]{DL}).

Consider the usual Hecke category for $H$
\begin{equation*}
\cD_{H}:=D^{b}_{m}(B_{H}\bs H/B_{H}).
\end{equation*}
We denote by $\IC(w)_{H}, \D(w)_{H}$ and $\nb(w)_{H}$ the objects in $\cD_{H}$ that are the intersection complex, standard perverse sheaf and costandard perverse sheaf supported on the closure of the Schubert cell $B_{H}wB_{H}/B_{H}\subset H/B_{H}$ defined similarly as in \eqref{std obj} for $H$ in place of $G$ and the trivial character sheaf on $T$ in place of $\cL$.

\begin{theorem}[Monodromic-Endoscopic equivalence for the neutral block]\label{th:main} Let $\cL\in \Ch(T)$ and $H$ be the endoscopic group of $G$ attached to $\cL$ as in \S\ref{ss:endo}. Then there is a canonical monoidal equivalence of triangulated categories
\begin{equation*}
\Psi^{\c}_{\cL}: \cD_{H}\isom{}_{\cL}\cD^{\c}_{\cL}
\end{equation*}
satisfying
\begin{enumerate}
\item For all $w\in \WL$
\begin{eqnarray}
\label{Phi IC} \Psi^{\c}_{\cL}(\IC(w)_{H})\cong \IC(w)^{\da}_{\cL}, \\
\label{Phi D}\Psi^{\c}_{\cL}(\D(w)_{H})\cong \D(w)^{\da}_{\cL}, \\
\label{Phi nb}\Psi^{\c}_{\cL}(\nb(w)_{H})\cong \nb(w)^{\da}_{\cL}.
\end{eqnarray}
In particular, $\Psi^{\c}_{\cL}$ is $t$-exact for the perverse $t$-structures.
\item There is a functorial isomorphism of graded $(R\ot R, \Fr)$-modules for all $\cF,\cF'\in \cD_{H}$ \footnote{This does not automatically follow from the equivalence $\Psi^{\c}_{\cL}$, because as in \eqref{Homconvention} $\Hom(-,-)$ denotes the Hom space after base change to $k=\ov\FF_{q}$.}
\begin{equation}\label{HomPsi}
\Homb(\cF,\cF')\isom \Homb(\Psi^{\c}_{\cL}(\cF), \Psi^{\c}_{\cL}(\cF')).
\end{equation}
\end{enumerate}
\end{theorem}

The proof will occupy \S\ref{ss:Weil sh} to \S\ref{ss:finish}. 

\subsection{DG model for ${}_{\cL}\cD_{\cL}^{\c}$}\label{ss:Weil sh} 
We apply the construction of \cite[\S B.1-B.2]{BY} to the category ${}_{\cL}\cD^{ \c}_{\cL}$. Let ${}_{\cL}\cC^{\c}_{\cL}\subset {}_{\cL}\cD^{\c}_{\cL}$ be the full subcategory consisting of objects that are pure of weight zero. By Proposition \ref{p:parity purity}, any object $\cF\in{}_{\cL}\cC^{\c}_{\cL}$ is also {\em very pure} in the sense that  $i^{*}_{w}\cF$ and $i^{!}_{w}\cF$ are pure of weight zero for all $w\in \WL$. Then ${}_{\cL}\cC^{\c}_{\cL}$ is an additive Karoubian category stable under the operation $(-)\ot V$, where $V$ is any bounded complex of finite-dimensional $\Fr$-modules such that $\upH^{i}V$ has weight $i$. In particular, ${}_{\cL}\cC^{\c}_{\cL}$ is stable under $\j{n}$, for all $n\in\ZZ$.  By Lemma \ref{l:conv pure}(2), ${}_{\cL}\cC^{\c}_{\cL}$ is  a monoidal category under convolution.

Let ${}_{\cL}\un\cC^{\c}_{\cL}$ be the essential image of ${}_{\cL}\cC^{\c}_{\cL}$ under $\om:{}_{\cL}\cD^{\c}_{\cL}\to{}_{\cL}\un\cD^{\c}_{\cL}$. Then ${}_{\cL}\un\cC^{\c}_{\cL}$ is the category of semisimple complexes in ${}_{\cL}\un\cD^{\c}_{\cL}$. Let $K^{b}({}_{\cL}\cC^{\c}_{\cL})$ be the homotopy category of bounded complexes in ${}_{\cL}\cC^{\c}_{\cL}$. Let $K^{b}({}_{\cL}\cC^{\c}_{\cL})_{0}\subset K^{b}({}_{\cL}\cC^{\c}_{\cL})$ be the thick subcategory consisting of complexes that are null-homotopic when mapped to $K^{b}({}_{\cL}\un\cC^{\c}_{\cL})$.

As in \cite[\S B.1]{BY}, with the help of a filtered version of ${}_{\cL}\cD^{\c}_{\cL}$, there is a triangulated functor (the realization functor) $\wt\r: K^{b}({}_{\cL}\cC^{\c}_{\cL})\to {}_{\cL}\cD^{\c}_{\cL}$. 

\begin{lemma} The functor $\wt\r$ descends to an equivalence
\begin{equation*}
\r: K^{b}({}_{\cL}\cC^{\c}_{\cL})/K^{b}({}_{\cL}\cC^{\c}_{\cL})_{0}\to {}_{\cL}\cD^{\c}_{\cL}.
\end{equation*}
\end{lemma}
\begin{proof}
By \cite[Proposition B.1.7]{BY}, $\wt\r$ induces an equivalence of triangulated categories
\begin{equation}\label{rho}
\r: K^{b}({}_{\cL}\cC^{\c}_{\cL})/\ker(\wt\r)\isom {}_{\cL}\cD^{\c}_{\cL}.
\end{equation}
We claim that $\ker(\wt\r)=K^{b}({}_{\cL}\cC^{\c}_{\cL})_{0}$. The inclusion $\ker(\wt\r)\subset K^{b}({}_{\cL}\cC^{\c}_{\cL})_{0}$ is proved in \cite[Lemma B.1.6]{BY}. We now show the inclusion in the other direction. Suppose $\cK^{\bu}\in K^{b}({}_{\cL}\cC^{\c}_{\cL})_{0}$, and $h:\om\cK^{\bu}\to \om\cK^{\bu}[-1]$ is a homotopy between $\id_{\cK^{\bu}}$ and $0$. Then for any $\cF\in {}_{\cL}\un\cC^{\c}_{\cL}$, $\Homb(\cF,\cK^{\bu})$ is calculated by a spectral sequence whose $E_{1}$-page consists of $E^{i,j}_{1}=\Ext^{j}(\cF,\cK^{i})$ with differentials $E^{i,j}_{1}\to E^{i+1,j}_{1}$ induced by the differentials of $\cK^{\bu}$. The chain homotopy $h$ implies that $E^{\bu,j}_{1}$ is null-homotopic, hence $E_{2}=0$ and $\Homb(\cF,\cK^{\bu})=0$ for all $\cF\in {}_{\cL}\un\cC^{\c}_{\cL}$. This implies that $\om\wt\r(\cK^{\bu})=0$ in ${}_{\cL}\un\cD^{\c}_{\cL}$. Now $\om: {}_{\cL}\cD^{\c}_{\cL}\to {}_{\cL}\un\cD^{\c}_{\cL}$ is conservative, hence $\wt\r(\cK^{\bu})=0$, and $\cK^{\bu}\in\ker(\wt\r)$.
\end{proof}

\begin{prop}\label{p:equiv CL}
The restriction of $\MM^{\c}$ gives a monoidal equivalence
\begin{equation*}
\ph_{0}: {}_{\cL}\cC^{\c}_{\cL}\isom \SB_{m}(\WL)
\end{equation*}
such that for $\cF,\cF'\in {}_{\cL}\cC^{\c}_{\cL}$, there is a canonical isomorphism in $\RRF$
\begin{equation}\label{HomSB}
\Homb(\cF,\cF')\cong \Homb_{\RRM}(\ph_{0}(\cF),\ph_{0}(\cF')).
\end{equation}
\end{prop}
\begin{proof}
The monoidal structure of $\MM^{\c}$ restricted to semisimple complexes is proved in Corollary \ref{c:MM mono}.
Let $\wt\ph_{0}:=\MM^{\c}|_{{}_{\cL}\cC_{\cL}^{\c}}: {}_{\cL}\cC_{\cL}^{\c}\to (R\ot R,\Fr)\gmod$. The isomorphism \eqref{HomSB} with $\ph_{0}$ replaced by $\wt\ph_{0}$ follows from Theorem \ref{th:Hom}. 

Now for $\cF,\cF'\in {}_{\cL}\cC_{\cL}^{\c}$, $\Ext^{i}(\cF,\cF')$ is pure of weight $i$ by the $*$-purity of $\cF$ and $!$-purity of $\cF'$ (cf. \cite[Lemma 3.1.5]{BY}). This implies $\hom_{{}_{\cL}\cC_{\cL}^{\c}}(\cF,\cF')=\Hom(\cF,\cF')^{\Fr}$ since $\Ext^{-1}(\cF,\cF')$ is pure of weight $-1$, and also implies $(\Homb(\cF,\cF'))^{\Fr}=\Hom(\cF,\cF')^{\Fr}=\hom_{{}_{\cL}\cC_{\cL}^{\c}}(\cF,\cF')$.  On the other hand, $\hom_{(R\ot R,\Fr)\gmod}(M,M')=\Hom_{\RRM}(M,M')^{\Fr}$ for $M,M'\in \SB_{m}(\WL)$. Taking Frobenius invariants of both sides of \eqref{HomSB} we conclude that $\ph_{0}$ is fully faithful.  

We show that the image of $\MM^{\c}|_{{}_{\cL}\cC_{\cL}^{\c}}$ lies in  $\SB_{m}(\WL)$. Indeed for $\cF\in {}_{\cL}\cC_{\cL}^{\c}$, $\om\cF$ is a semisimple complex, hence $\om\MM^{\c}(\cF)=\un\MM^{\c}(\om\cF)$ is a direct sum of shifts of $\un\MM^{\c}(\uIC(w)_{\cL})$, which is isomorphic to a shift of $\SS(w)_{\WL}$ by Proposition \ref{p:Sw}. On the other hand, the very purity of $\Th^{\c}_{\cL}$ and $\cF$ implies that $\Ext^{i}(\Th^{\c}_{\cL}, \cF)$ is pure of weight $i$, hence $\MM^{\c}(\cF)\in \RRF_{\pure}$. We conclude that $\MM^{\c}(\cF)\in \SB_{m}(\WL)$.  

Finally we show that any $M\in \SB_{m}(\WL)$ is in the essential image of $\ph_{0}$. Let $0=F_{0}M\subset F_{1}M\subset \cdots\subset F_{n}M=M$ be a filtration satisfying the  conditions in Proposition \ref{p:filM}. In particular, $\Gr^{F}_{i}M\cong \SS(w_{i})^{\na}_{\WL}\ot V_{i}$ for some $w_{i}\in \WL$ and $V_{i}\in D^{b}_{m}(\pt)$ pure of weight zero. We prove by induction on $i$ that $F_{i}M$ is in the essential image of $\ph_{0}$. For $i=0$ there is nothing to prove. 
Suppose $i\ge1$ and we have found $\cF_{i-1}\in {}_{\cL}\cC_{\cL}^{\c}$ such that $\ph_{0}(\cF_{i-1})\cong F_{i-1}M$. Let $\cK_{i}=\IC(w_{i})^{\da}_{\cL}\j{-\ell_{\cL}(w_{i})}\ot V_{i}\in {}_{\cL}\cC_{\cL}^{\c}$. Let $\e\in \Ext^{1}_{(R\ot R,\Fr)}(\Gr^{F}_{i}M,F_{i-1}M)$ be the extension class of
\begin{equation}\label{FiM}
0\to F_{i-1}M\to F_{i}M\to \Gr^{F}_{i}M\to 0
\end{equation}
in $(R\ot R,\Fr)\lmod$ (non-graded modules).  We have a short exact sequence
\begin{equation*}
0\to \Hom_{R\ot R}(\Gr^{F}_{i}M,F_{i-1}M)_{\Fr}\to \Ext^{1}_{(R\ot R,\Fr)}(\Gr^{F}_{i}M,F_{i-1}M)\to \Ext^{1}_{R\ot R}(\Gr^{F}_{i}M,F_{i-1}M)^{\Fr}\to 0.
\end{equation*}
Since \eqref{FiM} splits in $R\ot R\lmod$, the image of $\e$ in $\Ext^{1}_{R\ot R}(\Gr^{F}_{i}M,F_{i-1}M)^{\Fr}$ is zero, therefore $\e$ comes from a class $\wt\e\in \Hom_{R\ot R}(\Gr^{F}_{i}M,F_{i-1}M)_{\Fr}$. By Theorem \ref{th:Hom}, $\MM^{\c}$ induces an isomorphism of $\Fr$-modules
\begin{equation*}
\Homb(\cK_{i}, \cF_{i-1})\isom \Homb_{\RRM}(\Gr^{F}_{i}M,F_{i-1}M)=\Hom_{R\ot R}(\Gr^{F}_{i}M,F_{i-1}M).
\end{equation*}
Therefore $\wt\e$ can be viewed as a class $\wt\e'\in \Homb(\cK_{i}, \cF_{i-1})_{\Fr}=\Hom(\cK_{i}, \cF_{i-1})_{\Fr}$ (because $\Ext^{j}(\cK_{i}, \cF_{i-1})$ has weight $j$). Let $\e'$ be the image of $\wt\e'$ under $\Hom(\cK_{i}, \cF_{i-1})_{\Fr}\to \hom(\cK_{i}, \cF_{i-1}[1])$ (the latter is calculated in ${}_{\cL}\cD_{\cL}^{\c}$). Let $\cF_{i}=\Cone(\e')[-1]\in {}_{\cL}\cD^{\c}_{\cL}$. Then $\cF_{i}$ fits into a distinguished triangle $\cF_{i-1}\to \cF_{i}\to \cK_{i}\to \cF_{i-1}[1]$. Therefore $\cF_{i}\in {}_{\cL}\cC_{\cL}^{\c}$ and $\ph_{0}(\cF_{i})\cong F_{i}M$ by construction.
\end{proof}

To state the next theorem, we need some notation. For $\cF,\cF'\in {}_{\cL}\cD^{\c}_{\cL}$, we let $\Ext^{n}(\cF,\cF')_{m}$ be the weight $m$ summand of the $\Fr$-module $\Ext^{n}(\cF,\cF')$. For $M,M'\in K^{b}(R\ot R\lmod)$,  their morphism space in $K^{b}(R\ot R\lmod)$ is denoted $\HOM_{K^{b}(R\ot R\lmod)}(M,M')$, and it is the homotopy classes of $R\ot R$-linear chain maps $M\to M'$. We denote the degree shift of complexes in $K^{b}(R\ot R\lmod)$ by $\{1\}$. We denote 
\begin{equation*}
\HOM^{\bu}_{K^{b}(R\ot R\lmod)}(M,M')=\bigoplus_{n\in\ZZ}\HOM_{K^{b}(R\ot R\lmod)}(M,M'\{n\}).
\end{equation*}
When $M,M'\in K^{b}(\RRM)$, $\HOM_{K^{b}(R\ot R\lmod)}(M,M')$ also carries an internal grading from the gradings of each component $M^{i}$ and $M'^{i}$ (which are in $R\ot R\gmod$), and we denote the graded pieces by $\HOM_{K^{b}(R\ot R\lmod)}(M,M')_{m}$. 
If $M,M'\in K^{b}(\RRF)$, $\HOM_{K^{b}(R\ot R\lmod)}(M,M')$ also inherits a $\Fr$-module structure.

\begin{theorem}\label{th:SB} The equivalence $K^{b}(\ph_{0})$ on the homotopy categories of ${}_{\cL}\cC^{\c}_{\cL}$ and $\SB_{m}(\WL)$  induces a  monoidal equivalence of triangulated categories
\begin{equation*}
\ph_{\cL}: {}_{\cL}\cD^{\c}_{\cL}\xr{\r^{-1}}K^{b}({}_{\cL}\cC^{\c}_{\cL})/K^{b}({}_{\cL}\cC^{\c}_{\cL})_{0}\xr{K^{b}(\ph_{0})}K^{b}(\SB_{m}(\WL))/K^{b}(\SB_{m}(\WL))_{0}.
\end{equation*}
Here $K^{b}(\SB_{m}(\WL))_{0}$ consists of complexes of objects in $\SB_{m}(\WL)$ that become null-homotopic in $\SB(\WL)$. 

Moreover, for $\cF,\cF'\in {}_{\cL}\cD^{\c}_{\cL}$, we have a functorial isomorphism of $(R\ot R,\Fr)$-modules 
\begin{equation}\label{HomDSB}
\Homb(\cF,\cF')\isom \HOM^{\bu}_{K^{b}(R\ot R\lmod)}(\om\ph_{\cL}(\cF), \om\ph_{\cL}(\cF'))
\end{equation}
under which
\begin{equation}\label{HomDSBgr}
\Ext^{n}(\cF,\cF')_{m}\isom \HOM_{K^{b}(R\ot R\lmod)}(\om\ph_{\cL}(\cF), \om\ph_{\cL}(\cF')\{n-m\})_{m}, \quad \forall n,m\in\ZZ.
\end{equation}
\end{theorem}
\begin{proof}
In view of the equivalence $\ph_{0}$, to prove $\ph_{\cL}$ is an equivalence, it suffices to show that the image of $K^{b}({}_{\cL}\cC^{\c}_{\cL})_{0}$ under $K^{b}(\ph_{0})$ is  $K^{b}(\SB_{m}(\WL))_{0}$. If $\cK^{\bu}\in K^{b}({}_{\cL}\cC^{\c}_{\cL})_{0}$, then there exists a chain homotopy $h_{i}:\om\cK^{i}\to \om\cK^{i-1}$ between $\id_{\om\cK^{\bu}}$ and $0$.  Applying $\ph_{0}$ to $h_{i}$ we get $\ph_{0}(h_{i}): \om\ph_{0}(\cK^{i})\to \om\ph_{0}(\cK^{i-1})$ giving a chain homotopy between $\id_{\om\ph_{0}(\cK^{\bu})}$ and $0$. The same argument shows that $\ph_{0}^{-1}$ sends $K^{b}(\SB_{m}(\WL))_{0}$ to $K^{b}({}_{\cL}\cC^{\c}_{\cL})_{0}$. This shows that $\ph_{\cL}$ is an equivalence. The monoidal structure of $\ph_{\cL}$ comes from that of $\ph_{0}$, since $K^{b}({}_{\cL}\cC^{\c}_{\cL})_{0}$ and $K^{b}(\SB_{m}(\WL))_{0}$ are monoidal ideals.

We prove the isomorphism \eqref{HomDSB}. For $\cF^{\bu}, \cF'^{\bu}$ two bounded complexes in ${}_{\cL}\cC^{\c}_{\cL}$, let $\cF=\r(\cF^{\bu}), \cF'=\r(\cF'^{\bu})\in {}_{\cL}\cD^{\c}_{\cL}$. Then there is a spectral sequence with $E_{1}^{a,b}=\oplus_{j-i=a}\Ext^{b}(\cF^{i}, \cF'^{j})$ that converges to $\Ext^{a+b}(\cF,\cF')$. The differential $d_{1}:E^{a,b}_{1}\to E^{a+1,b}_{1}$ is given by an alternating sum of maps induced by the differentials in $\cF^{\bu}$ and  $\cF'^{\bu}$. Since $E_{1}^{a,b}$ is pure of weight $b$, the spectral sequence degenerates at $E_{2}$. This implies that
\begin{equation}\label{E2}
E_{2}^{a,b}=\Ext^{a+b}(\cF,\cF')_{b}.
\end{equation}

Let $M=\om\ph_{\cL}(\cF)$, i.e., $M$ is a complex with terms $M^{i}=\om\MM^{\c}(\cF^{i})\in\SB(\WL)$; similarly let $M'=\om\ph_{\cL}(\cF')$. By Theorem \ref{th:Hom}, $\MM^{\c}$ induces  an isomorphism of $\Fr$-modules $E_{1}^{a,b}\cong \oplus_{j-i=a}\Hom_{\RRM}(M^{i},M'^{j}[b])$,  and the differential $d_{1}$  is given by an alternating sum of differentials in $M^{\bu}$ and $M'^{\bu}$. Therefore $E_{2}^{a,b}$ is isomorphic to $\HOM_{K^{b}(\RRM)}(M, M'[b]\{a\})$, which is the same as the degree $b$ part of $\HOM_{K^{b}(R\ot R\lmod)}(M,M'\{a\})$ for the internal grading
\begin{equation}\label{E2SB}
E_{2}^{a,b}\cong\HOM_{K^{b}(R\ot R\lmod)}(M,M'\{a\})_{b}.
\end{equation}
Comparing \eqref{E2} and \eqref{E2SB} we get \eqref{HomDSBgr}. Taking direct sum over all $n$ and $m$ in \eqref{HomDSBgr}, we get \eqref{HomDSB}.
\end{proof}

\begin{remark}
The functor $K^{b}(\om): K^{b}(\SB_{m}(\WL))\to K^{b}(\SB(\WL))$ clearly factors through $K^{b}(\SB_{m}(\WL))_{0}$. It induces a functor
\begin{equation*}
\ov\ph_{\cL}: {}_{\cL}\cD^{\c}_{\cL}\to K^{b}(\SB(\WL)).
\end{equation*}
The homotopy category $K^{b}(\SB(\WL))$ can be viewed as a $\Fr$-semisimplified version of the monodromic Hecke category ${}_{\cL}\cD^{\c}_{\cL}$.
\end{remark}

\subsection{Finish of the proof of Theorem \ref{th:main}}\label{ss:finish} 
Apply Theorem \ref{th:SB} to the endoscopic group $H$ and the trivial character sheaf $\cL=\Qlbar\in \Ch(T)$,  we get a monoidal equivalence
\begin{equation*}
\ph_{H}: \cD_{H}\isom K^{b}(\SB_{m}(W_{H}))/K^{b}(\SB_{m}(W_{H}))_{0}
\end{equation*}
such that for $\cK,\cK'\in \cD_{H}$, there is a natural $R\ot R$-linear isomorphism
\begin{equation}\label{HomDSBH}
\Homb(\cK,\cK')\isom \HOM^{\bu}_{K^{b}(R\ot R\lmod)}(\om\ph_{H}(\cK), \om\ph_{H}(\cK'))
\end{equation}
with an analogue of \eqref{HomDSBgr}. Since $W_{H}=\WL$, we may identify the target categories of $\ph_{\cL}$ and $\ph_{H}$. 

Let $\Psi^{\c}_{\cL}=\ph^{-1}_{\cL}\c\ph_{H}$. Then $\Psi^{\c}_{\cL}$ is a monoidal equivalence of triangulated categories. Combining \eqref{HomDSB} and \eqref{HomDSBH} we get \eqref{HomPsi}.

It remains to show \eqref{Phi IC}-\eqref{Phi nb}. By Proposition \ref{p:Sw} and its analogue for $\cD_{H}$, we know $\ph_{\cL}(\IC(w)^{\da}_{\cL})\cong \SS(w)^{\na}_{\WL}\j{\ell_{\cL}(w)}\cong \ph_{H}(\IC(w)_{H})$ for all $w\in \WL$, therefore $\Psi^{\c}_{\cL}(\IC(w)_{H})\cong \IC(w)^{\da}_{\cL}$ for all $w\in \WL$.  This proves \eqref{Phi IC}.

Now we show \eqref{Phi D}. Let $\cF=(\Psi^{\c}_{\cL})^{-1}(\D(w)^{\da}_{\cL})$. From the properties of $\D(w)^{\da}_{\cL}$ and the fact that $(\Psi^{\c})^{-1}$ preserves IC sheaves and Hom spaces, we have
\begin{enumerate}
\item $\cF$ lies in the full triangulated subcategory generated by $\IC(w')_{H}\ot V$ for $w'\le w$ (Bruhat order of $W$) and Fr-modules $V$;
\item $\Hom^{\bullet}(\cF, \IC(w')_{H})=0 $ for all $w'<w$.
\item There is a canonical isomorphism $\Hom^{\bullet}(\cF, \IC(w)_{H})\cong R(w)$ as graded $(R\ot R, \Fr)$-modules.
\end{enumerate}
We show that these properties imply that a canonical isomorphism $\cF\cong \D(w)_{H}$. Let $Y\subset H$ (resp. $Z\subset H$) be the union of $B_{H}w'B_{H}$ for $w'\le w$ (resp, $w'<w$) in the Bruhat order of $W$. By Lemma \ref{l:order}(3), $Y$ and $Z$ are closed, and $Y-Z=B_{H}wB_{H}=H(w)$ is open in $Y$. Notice that $Y$ is not necessarily the closure of $H(w)$, see Remark \ref{r:order}. Property (1) above implies that $\cF$ is supported on $B_{H}\bs Y/B_{H}$; property (2) implies that $\cF|_{Z}=0$. Therefore $\cF=j_{!}\cG$ for some $\cG\in D^{b}_{m}(B_{H}\bs H(w)/B_{H})=:\cD_{H}(w)$ (where $j:H(w)\incl Y$ is the open embedding). Now $\Hom^{\bullet}(\cF,\IC(w)_{H})\cong \Hom^{\bullet}_{\cD_{H}(w)}(\cG, \Qlbar)$, and property (3)  above gives a map $\cF\to \IC(w)_{H}$  in $\cD_{H}$ (corresponding to $1\in \Qlbar$ under the isomorphism $\Hom(\cF, \IC(w)_{H})\cong\Qlbar$, which is $\Fr$-invariant hence lifts uniquely to $\hom(\cF,\IC(w)_{H})$ for $\Hom^{-1}(\cF, \IC(w)_{H})=0$), hence a nonzero map $c: \cG\to \Qlbar$ in $\cD_{H}(w)\cong D^{b}_{\G(w),m}(\pt)$. Moreover, property (3) implies that $\Hom^{\bullet}(\cG, \Qlbar)$ is free left $R$-module of rank one and is generated in degree zero, which implies that $\cG$ has rank one and is concentrated in degree zero. Therefore the canonical map $c$ is an isomorphism, and it induces an isomorphism $\cF=j_{!}\cG\isom j_{!}\Qlbar=\D(w)_{H}$. Therefore we get a canonical isomorphism $\Psi^{\c}_{\cL}(\D(w)_{H})\cong \Psi_{\cL}^{\c}(\cF)=\D(w)^{\da}_{\cL}$.

A similar argument proves \eqref{Phi nb}. \qed

%

\begin{remark}[Parabolic version] It is possible to extend Theorem \ref{th:main} to a parabolic version. Namely, consider two standard parabolic subgroups $P$ and $Q$ of $G$ with unipotent radicals $U_{P}$ and $U_{Q}$ and Levi subgroups $L$ and $M$ containing $T$. Suppose  that $\cL\in\Ch(T)$ extends to rank one local systems $\cK\in\Ch(L)$ and $\cK'\in \Ch(M)$. Then we may consider the category ${}_{\cK'}\cD_{\cK}=D^{b}_{(M\times L,\cK'\boxtimes\cK^{-1}),m}(U_{Q}\bs G/U_{P})$. We still have a block decomposition of ${}_{\cK'}\cD_{\cK}$ indexed by $\Om_{\cL}=W_{\cL}/\WL$.

By Lemma \ref{l:ext L} we have $\Phi(L,T), \Phi(M,T)\subset \Phi_{\cL}$, hence $L$ and $M$ determine standard parabolic subgroups $P_{H}$ and $Q_{H}$ of $H$ whose Levi factors have roots $\Phi(L,T)$ and $\Phi(M,T)$. Then there is an equivalence between the neutral block ${}_{\cK'}\cD^{\c}_{\cK}$ and $D^{b}_{m}(Q_{H}\bs H/P_{H})$, which can be proved using similar techniques used in this paper.
\end{remark}

We have the following strengthening of the purity result in Proposition \ref{p:parity purity} to include Frobenius semisimplicity for the stalks of IC sheaves. To state it, recall that $P^{\WL}_{x,y}(t)$ is the Kazhdan-Lusztig polynomial \cite{KL} for the Coxeter group  $\WL$, of degree less than $\frac{1}{2}(\ell_{\cL}(y)-\ell_{\cL}(x))$ if $x<y$. The numerical part of the following result was first proved in \cite[Lemma 1.11]{L-book}.
 
\begin{prop}\label{p:stalk ss} Let $\cL,\cL'\in\fo$.
\begin{enumerate}
\item For $v\le w\in \WL$, write
\begin{equation*}
P^{\WL}_{v,w}(t)=\sum_{n\ge0} a_{v,w}(n)t^{n}.
\end{equation*}
Then we have
\begin{eqnarray}
i^{*}_{v}\IC(w)^{\da}_{\cL}\cong C(v)^{\da}_{\cL}\j{\ell_{\cL}(w)-\ell_{\cL}(v)}\ot\left(\bigoplus_{n\ge0}\Qlbar\j{-2n}^{\oplus a_{v,w}(n)}\right),\\
\label{costalk sum} i^{!}_{v}\IC(w)^{\da}_{\cL}\cong C(v)^{\da}_{\cL}\j{-\ell_{\cL}(w)+\ell_{\cL}(v)}\ot\left(\bigoplus_{n\ge0}\Qlbar\j{2n}^{\oplus a_{v,w}(n)}\right).
\end{eqnarray}

\item Let $w,v$ be in the same block $\b\in {}_{\cL'}\un W_{\cL}$. Write $v=w^{\b}x$ and $w=w^{\b}y$ for $x,y\in \WL$.  Then there is a one-dimensional $\Fr$-module $V_{1}$ (depending on $\dw$ and $\dv$) of weight zero such that
\begin{eqnarray*}
i^{*}_{v}\IC(\dw)_{\cL}\cong C(\dv)_{\cL}\j{\ell_{\cL}(y)-\ell_{\cL}(x)}\ot\left(\bigoplus_{n\ge0}\Qlbar\j{-2n}^{\oplus a_{x,y}(n)}\right)\ot V_{1},\\
i^{!}_{v}\IC(\dw)_{\cL}\cong C(\dv)_{\cL}\j{-\ell_{\cL}(y)+\ell_{\cL}(x)}\ot\left(\bigoplus_{n\ge0}\Qlbar\j{2n}^{\oplus a_{x,y}(n)}\right)\ot V_{1}.
\end{eqnarray*}
\end{enumerate}
\end{prop}
\begin{proof} (1) We first treat the costalk $i^{!}_{v}\IC(w)^{\da}_{\cL}$. In Proposition \ref{p:parity purity}, we already proved that $\cK_{v}=i^{!}_{v}\IC(w)^{\da}_{\cL}\in{}_{\cL}\cD(v)_{\cL}$ is a successive extension of $C(v)^{\da}\j{n}\ot V_{n}$ for finite-dimensional $\Fr$-modules $V_{n}$ pure of weight zero, and $n\equiv \ell(w)-\ell(v)\mod2$. We shall first show
\begin{equation}\label{Frss}
\mbox{$\cK_{v}$ is a direct sum of $C(v)^{\da}_{\cL}\j{n}$ for $n\equiv\ell(w)-\ell(v)\mod2$.}
\end{equation}
Let ${}_{\cL}\cC(v)_{\cL}\subset {}_{\cL}\cD(v)_{\cL}$ be the subcategory of complexes that are pure of weight zero. Applying Proposition \ref{p:equiv CL} to the case $G=T$ (now $C(v)^{\da}_{\cL}$ plays the role of $\Th^{\c}_{\cL}$), we see that $\Hom^{\bu}(C(v)^{\da}_{\cL}, -)$ induces a full embedding
\begin{equation}\label{emb h}
h: {}_{\cL}\cC(v)_{\cL}\incl (R(v), \Fr)\lmod
\end{equation}
Here $R(v)=\upH^{*}_{\G(v)_{k}}(\pt_{k})$ is introduced in \S\ref{ss:Rbimod}. Under this embedding, to show \eqref{Frss}, it suffices to show that $h(\cK_{v})$ is a direct sum of $R(v)\j{n}$ for $n\equiv\ell(w)-\ell(v)\mod2$. By Proposition \ref{p:parity purity},  $h(\cK_{v})$ is a successive extension of $R(v)\j{n}\ot V_{n}$ for  $n\equiv\ell(w)-\ell(v)\mod2$ and for $\Fr$-modules $V_{n}$ pure of weight zero. In particular, $h(\cK_{v})$ is free as an $R(v)$-module. Therefore it suffices to show that $h(\cK_{v})$, as a $\Fr$-module, is a direct sum of $\Qlbar\j{n}$ for $n\in\ZZ$ (necessarily of the same parity as $\ell(w)-\ell(v)$): for then $h(\cK_{v})\ot_{R(v)}\Qlbar$ is a direct sum of $\Qlbar\j{n}$, and we can lift a basis of $h(\cK_{v})\ot_{R(v)}\Qlbar$ consisting of Frobenius eigenvectors to Frobenius eigenvectors in $h(\cK_{v})$, giving an $R(v)$-basis of $h(\cK_{v})$.

To summarize, to show \eqref{Frss}, we only need to show that $h(\cK_{v})$ is a direct sum of $\Qlbar\j{n}$ as a $\Fr$-module. By Corollary \ref{c:filM}, $\Homb(i^{*}_{v}\Th^{\c}_{\cL}, \cK_{v})=\Gr^{F}_{v}\MM^{\c}(\IC(w)^{\da}_{\cL})$, the latter being a subquotient of $\SS(w)^{\na}_{\WL}\j{\ell_{\cL}(w)}$ (by Proposition \ref{p:Sw}), is hence a direct sum of $\Qlbar\j{n}$. Since $i^{*}_{v}\Th^{\c}_{\cL}=C(v)_{\cL}^{\da}\j{-\ell_{\cL}(v)}$,  $h(\cK_{v})=\Gr^{F}_{v}\MM^{\c}(\IC(w)^{\da}_{\cL})\j{-\ell_{\cL}(v)}$ is a direct sum of $\Qlbar\j{n}$. This proves \eqref{Frss}.

By Theorem \ref{th:main}, we have $h(\cK_{v})=\Homb(\D(v)^{\da}_{\cL}, \IC(w)^{\da}_{\cL})\cong \Homb(\D(v)_{H}, \IC(w)_{H})=\Homb(C(v)_{H}, i^{!}_{v}\IC(w)_{H})$. Therefore the multiplicity of $C(v)^{\da}_{\cL}\j{n}$ in $\cK_{v}$ is the same as the  multiplicity of $C(v)_{H}\j{n}$ in $i^{!}_{v}\IC(w)_{H}$, which is well-known to be expressed in terms of the coefficients of $P^{\WL}_{v,w}$, as in \eqref{costalk sum}.

The statement for $i^{*}_{v}\IC(w)^{\da}_{\cL}$ can be proved in the same way by analyzing $\Homb(\IC(w)^{\da}_{\cL}, \nb(v)^{\da}_{\cL})$ and comparing it to $\Homb(\IC(w)_{H}, \nb(v)_{H})$. We omit details.


(2) By Proposition \ref{p:min th}, there is a minimal IC sheaf $\xi\in {}_{\cL'}\fP^{\b}_{\cL}$ such that $\IC(\dw)_{\cL}\cong \xi\star\IC(y)^{\da}_{\cL}$. Then $\xi\star\D(x)^{\da}_{\cL}\cong \D(\dv)_{\cL}\ot V_{1}$ for some one-dimensional $\Fr$-module $V_{1}$. We have $\Homb(C(\dv)_{\cL}\ot V_{1}, i_{v}^{!}\IC(\dw)_{\cL})=\Homb(\D(\dv)_{\cL}\ot V_{1}, \IC(\dw)_{\cL})\cong\Homb(\xi\star\D(x)^{\da}_{\cL},  \xi\star\IC(y)_{\cL}^{\da})=\Homb(\D(x)^{\da}_{\cL}, \IC(y)_{\cL}^{\da})$ which is a direct sum of $\Qlbar\j{n}$ as a  $\Fr$-module by (1). By the same argument as in (1) using the embedding \eqref{emb h} , this implies that $i_{v}^{!}\IC(\dw)_{\cL}$ is a direct sum of  $C(\dv)_{\cL}\j{n}\ot V_{1}$, with multiplicities given by the coefficients of $P^{\WL}_{x,y}$. The argument for $i_{v}^{*}\IC(\dw)_{\cL}$ is similar, using $\Homb(\IC(\dw)_{\cL}, \nb(\dv)_{\cL}\ot V_{1})\cong \Homb(\IC(y)^{\da}_{\cL}, \nb(x)^{\da}_{\cL})$.
\end{proof}

Similarly we have the Frobenius semisimplicity of convolution.

\begin{prop}\label{p:conv ss}
\begin{enumerate}
\item For $w,w'\in \WL$, the convolution $\IC(w')^{\da}_{\cL}\star\IC(w)^{\da}_{\cL}$  is a direct sum of $\IC(v)^{\da}_{\cL}\j{n}$ for $v\in\WL$ and $n\equiv\ell_{\cL}(w)+\ell_{\cL}(w')-\ell_{\cL}(v)\mod2$.

\item Let $\cL,\cL',\cL''\in \fo$, $w\in {}_{\cL'}W_{\cL}$ and  $w'\in {}_{\cL''}W_{\cL'}$. Let $\b\in {}_{\cL'}\un W_{\cL}$ and $\b'\in{}_{\cL''}\un W_{\cL'}$ be the blocks containing $w$ and $w'$. Then the convolution $\IC(\dw')_{\cL'}\star\IC(\dw)_{\cL}$ is a direct sum of $\IC(\dv)_{\cL}\j{n}\ot V_{\dw',\dw}^{\dv}$ for $v\in \b'\b\subset {}_{\cL''}W_{\cL}, n\equiv\ell_{\b}(w)+\ell_{\b'}(w')-\ell_{\b'\b}(v)\mod2$ and a one-dimensional $\Fr$-module $V_{\dw',\dw}^{\dv}$ depending only on $\dw,\dw'$ and $\dv$.
\end{enumerate}
\end{prop}
\begin{proof} 
(1) The same statement for $\cD_{H}$ holds by \cite[Proposition 3.2.5]{BY}, hence (1) follows from the equivalence $\Psi^{\c}_{\cL}$.

(2) Write $w=xw^{\b}$ for $x\in W^{\c}_{\cL'}$; $w'=w^{\b'}y$ for $y\in W^{\c}_{\cL'}$. Let $\xi\in {}_{\cL'}\fP^{\b}_{\cL}$ and $\y\in {}_{\cL''}\fP^{\b'}_{\cL'}$ be such that $\IC(\dw)_{\cL}\cong \IC(x)_{\cL'}^{\da}\star\xi$ and $\IC(\dw')_{\cL'}\cong \y\star\IC(y)^{\da}_{\cL'}$. For $v\in \b'\b$, we have $v=w^{\b'}zw^{\b}$ for $z\in W^{\c}_{\cL'}$, and let $V_{\dw',\dw}^{\dv}$ be the one-dimensional $\Fr$-module such that
\begin{equation*}
\IC(\dv)_{\cL}\ot V_{\dw',\dw}^{\dv}\cong \y\star\IC(z)^{\da}_{\cL'}\star\xi.
\end{equation*}
By (1), $\IC(y)^{\da}_{\cL'}\star\IC(x)^{\da}_{\cL'}$ is a direct sum of $\IC(z)^{\da}_{\cL'}\j{n}$ for $z\in W^{\c}_{\cL'}$ and $n\equiv\ell_{\cL'}(x)+\ell_{\cL'}(y)-\ell_{\cL'}(z)\mod2$. Therefore $\IC(\dw')_{\cL'}\star\IC(\dw)_{\cL}\cong \y\star\IC(y)^{\da}_{\cL'}\star\IC(x)^{\da}_{\cL'}\star\xi$ is a direct sum of $\y\star\IC(z)^{\da}_{\cL'}\j{n}\star\xi$ for   $z\in W^{\c}_{\cL'}$ and $n\equiv\ell_{\cL'}(x)+\ell_{\cL'}(y)-\ell_{\cL'}(z)\mod2$, or equivalently a direct sum of $\IC(\dv)_{\cL}\j{n}\ot V_{\dw',\dw}^{\dv}$ for $v\in \b'\b$ and $n\equiv\ell_{\cL'}(x)+\ell_{\cL'}(y)-\ell_{\cL'}(z)\mod2$ where $v=w^{\b'}zw^{\b}$. It remains to note that $\ell_{\cL'}(x)=\ell_{\b}(w), \ell_{\cL'}(y)=\ell_{\b'}(w')$ and $\ell_{\cL'}(z)=\ell_{\b'\b}(v)$.
\end{proof}


\section{Equivalence for all blocks}\label{s:all blocks}

In this section we extend the monoidal equivalence for the neutral blocks in Theorem \ref{th:main} to an equivalence for all blocks (Theorem \ref{th:all blocks}). To do this, we will need to extend the endoscopic group to a groupoid, and it will be convenient to organize the various blocks into a $2$-category.

\subsection{The groupoid $\wt\Xi$}\label{ss:tXi}  We define a  groupoid $\wt\Xi$ in $\FF_{q}$-schemes as follows. Its object set is  $\fo$, and the  morphism ${}_{\cL'}\wt\Xi_{\cL}$ between $\cL$ and $\cL'\in\fo$ is the union of connected components of $N_{G}(T)$ whose image in $W$ is in ${}_{\cL'}\Xi_{\cL}$. In other words, ${}_{\cL'}\wt\Xi_{\cL}$ parametrizes liftings  of $w^{\b}$ for blocks $\b\in{}_{\cL'}\un W_{\cL}$.  The composition map  is defined by the multiplication in $N_{G}(T)$. We have an obvious map of groupoids $\wt\Xi\to \Xi$ which is a $T$-torsor. 

For $\b\in {}_{\cL'}\un W_{\cL}$, let ${}_{\cL'}\wt\Xi^{\b}_{\cL}\subset {}_{\cL'}\wt\Xi_{\cL}$ be the component corresponding to $w^{\b}$. Then  ${}_{\cL'}\wt\Xi_{\cL}=\coprod_{\b\in{}_{\cL'}\un W_{\cL}}{}_{\cL'}\wt\Xi^{\b}_{\cL}$.

\subsection{Relative pinning}\label{ss:rel pin} We give a rigidification of the endoscopic group $H=H_{\cL}$ attached to $\cL\in\fo$ as follows. Recall that $H$ contains $T$ as a maximal torus, and has $\Phi^{+}_{\cL}$ as its positive roots with respect to the Borel $B_{H}$. Let $\D_{\cL}\subset \Phi_{\cL}^{+}$ be the set of simple roots. A {\em relative pinning} for the endoscopic group $H$ is a collection of isomorphisms $\io_{\a}: H_{\a}\cong G_{\a}$ for each $\a\in \D_{\cL}$. Here $H_{\a}$ (resp. $G_{\a}$) is the root subgroup for $\a$ (isomorphic to the additive group) of $H$ (resp. $G$). The automorphism group of the data $(H, T,B_{H}, \{\io_{\a}\}_{\a\in\D_{\cL}})$ is trivial. Therefore a relatively pinned endoscopic group attached to $\cL$ is unique up to a unique isomorphism. 

For each $\cL\in\fo$, we use the notation $H^{\c}_{\cL}$ to denote the relatively pinned endoscopic group attached to $\cL$. Its canonical Borel subgroup is denoted $B^{H}_{\cL}$.

Let $\cL,\cL'\in\fo$ and $\ddw\in {}_{\cL'}\wt\Xi_{\cL}$ with image $w\in {}_{\cL'}W_{\cL}$.  There is a unique isomorphism
\begin{equation*}
\s(\ddw): H^{\c}_{\cL}\to H^{\c}_{\cL'}
\end{equation*}
characterized as follows. It is $w$ when restricted to $T$. Since $w$ is minimal in its block, it induces an isomorphism between the based root systems $(\Phi_{\cL}, \D_{\cL})$ and $(\Phi_{\cL'}, \D_{\cL'})$. For each simple root $\a\in \D_{\cL}$, $\s(\ddw)$ is required to restrict to an isomorphism of root subgroups $H^{\c}_{\cL,\a}\isom H^{\c}_{\cL', w\a}$, and we require that the following diagram be commutative
\begin{equation*}
\xymatrix{ H^{\c}_{\cL,\a}\ar[r]^-{\io_{\a}}\ar[d]^{\s(\ddw)} & G_{\a}\ar[d]^{\Ad(\ddw)}\\
H^{\c}_{\cL',w\a}\ar[r]^-{\io_{w\a}} & G_{w\a}
}
\end{equation*}
When $\cL'=\cL$, the above construction gives an action of ${}_{\cL}\wt\Xi_{\cL}$ on $H^{\c}_{\cL}$. When restricted to $T\subset {}_{\cL}\wt\Xi_{\cL}$, it is the conjugation action of $T$ on $H^{\c}_{\cL}$.

\subsection{The groupoid $\fH$}\label{ss:fH}
We construct a groupoid $\fH$ in $\FF_{q}$-schemes together with a map of groupoids $\om_{\fH}: \fH\to \Xi$ as follows. Set $\Ob(\fH)=\fo$ and $\om_{\fH}$ is the identity on objects. For $\cL,\cL'\in\fo$, define the morphism $\FF_{q}$-scheme in $\fH$ as
\begin{equation*}
{}_{\cL'}\fH_{\cL}={}_{\cL'}\wt\Xi_{\cL}\twtimes{T}H^{\c}_{\cL},
\end{equation*}
where the action of $T$ on $H^{\c}_{\cL}$ is by left translation, and its action on ${}_{\cL'}\wt\Xi_{\cL}\subset N_{G}(T)$ is by right translation. 

For $\b\in{}_{\cL'}\un W_{\cL}$, we get a component 
\begin{equation*}
{}_{\cL'}\fH^{\b}_{\cL}:={}_{\cL'}\wt\Xi^{\b}_{\cL}\twtimes{T}H^{\c}_{\cL}.
\end{equation*}
The map $\om_{\fH}: \fH\to \Xi$ then sends ${}_{\cL'}\fH^{\b}_{\cL}$ to $w^{\b}\in {}_{\cL'}\Xi_{\cL}$. There is a canonical isomorphism
\begin{equation*}
{}_{\cL'}\fH^{\b}_{\cL}={}_{\cL'}\wt\Xi^{\b}_{\cL}\twtimes{T}H^{\c}_{\cL}\cong H^{\c}_{\cL'}\twtimes{T}{}_{\cL'}\wt\Xi^{\b}_{\cL}
\end{equation*}
sending $(\ddw, h)\mapsto (\s(\ddw)(h), \ddw)$. Under this isomorphism, ${}_{\cL'}\fH^{\b}_{\cL}$ is a   $(H^{\c}_{\cL'}, H^{\c}_{\cL})$-bitorsor. 

For $\b\in{}_{\cL'}\un W_{\cL}$ and $\g\in{}_{\cL''}\un W_{\cL'}$, the composition map
\begin{equation*}
{}_{\cL''}\fH^{\g}_{\cL'}\times {}_{\cL'}\fH^{\b}_{\cL}\to {}_{\cL''}\fH^{\g\b}_{\cL}
\end{equation*}
is defined as
\begin{eqnarray*}
({}_{\cL''}\wt\Xi^{\g}_{\cL'}\twtimes{T}H^{\c}_{\cL'})\times ({}_{\cL'}\wt\Xi^{\b}_{\cL}\twtimes{T}H^{\c}_{\cL})\to {}_{\cL''}\wt\Xi^{\g\b}_{\cL}\twtimes{T}H^{\c}_{\cL}\\
(\ddw',h',\ddw,h)\mapsto (\ddw'\ddw, \s(\ddw^{-1})h'h).
\end{eqnarray*}
It is easy to check that the composition map is associative. Under the composition map, ${}_{\cL}\fH_{\cL}$ becomes a group scheme over $\FF_{q}$ with neutral component $H_{\cL}^{\c}$ and component group $W_{\cL}/\WL$. Each ${}_{\cL'}\fH_{\cL}$ is a $({}_{\cL'}\fH_{\cL'}, {}_{\cL}\fH_{\cL})$-bitorsor.

The double cosets $B^{H}_{\cL'}\bs {}_{\cL'}\fH_{\cL}/B^{H}_{\cL}$ are in natural bijection with ${}_{\cL'}W_{\cL}$: for $w\in {}_{\cL'}W_{\cL}$ we can write it uniquely as $w^{\b}v$ for the block $\b\in {}_{\cL'}\un W_{\cL}$  containing $w$ and $v\in \WL=W(H_{\cL}^{\c}, T)$. Then $w$ corresponds to the $(B^{H}_{\cL'},B^{H}_{\cL})$-double coset containing $(\dw^{\b}, \dot v)\in {}_{\cL'}\wt\Xi_{\cL}\twtimes{T}H_{\cL}^{\c}={}_{\cL'}\fH_{\cL}$, which we denote by $\fH(w)_{\cL}$.

\subsection{$2$-categories over a groupoid} What we call a $2$-category $\fC$ is called a ``bicategory'' in \cite[Ch.XII.6]{Mac}. It has an object set $\Ob(\fC)$, and for $x,y\in\Ob(\fC)$ the morphisms from $x$ to $y$ form an ordinary category which we denote by ${}_{y}\fC_{x}$. The category ${}_{x}\fC_{x}$ carries an identity $\one_{x}$. For  $x,y,z\in \Ob(\fC)$ there is a bifunctor called composition: ${}_{z}\fC_{y}\times{}_{y}\fC_{x}\to {}_{z}\fC_{x}$. For a quadruple of objects there is a natural isomorphism of functors giving the associativity of composition. These data are required to satisfy the pentagon axiom for associativity and another axiom involving the identities $\{\one_{x}\}$.

From a $2$-category $\fC$ we get an ordinary category $\pi_{\le 1}\fC$ with the same object set and morphism sets $\Hom_{\pi_{\le 1}\fC}(x,y):=|{}_{y}\fC_{x}|$, the set of isomorphism classes of objects of ${}_{y}\fC_{x}$.

Let $\G$ be a small groupoid, viewed as a category where all morphisms are isomorphisms. A $2$-category $\fC$ over $\G$ is a $2$-category with a functor $\om: \pi_{\le 1}\fC\to \G$. In other words, for each object $x\in \Ob(\fC)$ we assign an object $\om(x)\in \Ob(\G)$, and for a pair of objects $x,y\in \Ob(\fC)$, a map ${}_{y}h_{x}: |{}_{y}\fC_{x}|\to {}_{\om(y)}\G_{\om(x)}$ compatible with compositions and sending identities to identities.

If $(\fC,\om:\pi_{\le 1}\fC\to \G)$  is a $2$-category over $\G$, and $x,y\in \Ob(\fC)$, $\xi\in {}_{\om(y)}\G_{\om(x)}$, we denote by $_{y}\fC^{\xi}_{x}\subset {}_{y}\fC_{x}$ the full subcategory of objects whose isomorphism class maps to $\xi$ via ${}_{y}h_{x}$. Then ${}_{y}\fC_{x}=\coprod_{\xi\in{}_{\om(y)}\G_{\om(x)}}{}_{y}\fC^{\xi}_{x}$. The composition functor restricts to a bifunctor
\begin{equation*}
\c: {}_{z}\fC^{\y}_{y}\times {}_{y}\fC^{\xi}_{x}\to {}_{z}\fC^{\y\xi}_{x}, \quad \forall \xi\in{}_{\om(y)}\G_{\om(x)}, \y\in{}_{\om(z)}\G_{\om(y)}.
\end{equation*}

\begin{exam}
The categories $\{{}_{\cL'}\cD^{\b}_{\cL}\}_{\cL,\cL'\in\fo}$ can be organized into a $2$-category $\fD$ over $\Xi$ in an obvious way. The object set is $\fo$ and $\om$ is the identity map on the object sets. For $\cL,\cL'\in\fo$, the morphism category is ${}_{\cL'}\fD_{\cL}=\coprod_{\b\in {}_{\cL'}\un W_{\cL}}{}_{\cL'}\cD^{\b}_{\cL}$ with composition given by convolution (using Proposition \ref{p:conv block}).
\end{exam}

\begin{exam}
For  $\cL,\cL'\in\fo$ and $\b\in{}_{\cL'}\un W_{\cL}$, define
\begin{equation*}
{}_{\cL'}\cE^{\b}_{\cL}:=D^{b}_{m}(B^{H}_{\cL'}\bs {}_{\cL'}\fH^{\b}_{\cL}/B^{H}_{\cL}).
\end{equation*}
Then the groupoid structure on $\fH$ gives a convolution functor for $\b\in{}_{\cL'}\un W_{\cL}$ and $\g\in{}_{\cL''}\un W_{\cL'}$
\begin{equation*}
\star: {}_{\cL''}\cE^{\g}_{\cL'} \times {}_{\cL'}\cE^{\b}_{\cL} \to {}_{\cL''}\cE^{\g\b}_{\cL}
\end{equation*}
carrying an associativity natural transformation satisfying the pentagon axiom. This defines a $2$-category $\fE$ over $\Xi$ with object set $\fo$ and morphism categories ${}_{\cL'}\fE_{\cL}=\coprod_{\b\in{}_{\cL'}\un W_{\cL}}{}_{\cL'}\cE^{\b}_{\cL}$.

If $\b=W^{\c}_{\cL}\subset {}_{\cL}\un W_{\cL}$ is the neutral block, we denote ${}_{\cL}\cE^{\b}_{\cL}$ by ${}_{\cL}\cE^{\c}_{\cL}$. This is the usual Hecke category $\cD_{H_{\cL}^{\c}}$ for the reductive group $H_{\cL}^{\c}$.
\end{exam}

We are ready to state the extension of Theorem \ref{th:main} to all blocks. Recall from \S\ref{ss:fH} that for $\cL\in\fo$ and $w\in W$, we have a $(B^{H}_{w\cL}, B^{H}_{\cL})$-double coset $\fH(w)_{\cL}\subset {}_{w\cL}\fH_{\cL}$. Let $C(w)^{H}_{\cL}=\Qlbar\j{\ell_{\b}(w)}$ (where $\b\in{}_{w\cL}\un W_{\cL}$ is the block containing $w$) be the shifted and twisted constant sheaf on $\fH(w)_{\cL}$. Let $\D(w)^{H}_{\cL}, \nb(w)^{H}_{\cL}$ and $\IC(w)^{H}_{\cL}$ be the $!$-, $*$-, and middle extensions of $C(w)^{H}_{\cL}$ to the closure of $\fH(w)_{\cL}$, viewed as objects in ${}_{w\cL}\cE^{\b}_{\cL}\subset {}_{w\cL}\fE_{\cL}$.

\begin{theorem}[Monodromic-Endoscopic equivalence in general]\label{th:all blocks} 

There is a canonical equivalence of $2$-categories over $\Xi$
\begin{equation*}
\Psi:\fE\cong \fD
\end{equation*}
such that
\begin{enumerate}
\item For $\cL\in\fo$ and $\b$ the unit coset in ${}_{\cL}\un W_{\cL}$, the equivalence $\Psi$ restricts to the equivalence $\Psi^{\c}_{\cL}$ in Theorem \ref{th:main} as monoidal functors.
\item For $\cL,\cL'\in\fo$ and $w\in {}_{\cL'}W_{\cL}$, the equivalence $\Psi$ sends  $\D(w)^{H}_{\cL}, \nb(w)^{H}_{\cL}$ and $\IC(w)^{H}_{\cL}$ in ${}_{\cL'}\fE_{\cL}$ to $\D(w)^{\da}_{\cL}, \nb(w)^{\da}_{\cL}$ and $\IC(w)^{\da}_{\cL}$ in ${}_{\cL'}\cD_{\cL}$ (see Definition \ref{def:rig IC}).
\end{enumerate}
\end{theorem}

\begin{remark} In the statement of the above theorem, $\Psi$ being an equivalence of $2$-categories implies that for $\cL,\cL'\in\fo$ and $\b\in{}_{\cL'}\un W_{\cL}$, it restricts to an equivalence of triangulated categories ${}_{\cL'}\Psi^{\b}_{\cL}: {}_{\cL'}\cE^{\b}_{\cL}\isom {}_{\cL'}\cD^{\b}_{\cL}$; moreover,  the equivalences $\{{}_{\cL'}\Psi^{\b}_{\cL}\}$ are compatible with convolution structures


\end{remark}

\begin{remark}\label{r:dep on pin} Since the definition of the rigidified minimal IC sheaves $\IC(w^{\b})^{\da}_{\cL}$ depends on the choice of the pinning $\bx_{-\a_{s}}: U_{-\a_{s}}\cong \Ga$ and the additive character $\psi_{0}$ on $\FF_{q}$, so does the equivalence $\Psi$
 in the theorem. 
 \end{remark}

The rest of the section is devoted to the proof of Theorem \ref{th:all blocks}.

\subsection{Action of minimal IC sheaves on neutral blocks}\label{ss:action min} For $\cL,\cL'\in\fo, \b\in{}_{\cL'}\un W_{\cL}$ we define the functor
\begin{eqnarray*}
{}^{\b}(-): {}_{\cL}\cD^{\c}_{\cL}&\to& {}_{\cL'}\cD^{\c}_{\cL'}\\
\cF&\mapsto& {}^{\b}\cF:=\xi\star \cF\star\xi^{-1}. 
\end{eqnarray*}
where  $\xi\in {}_{\cL'}\fP^{\b}_{\cL}$, and $\xi^{-1}\in {}_{\cL}\fP^{\b^{-1}}_{\cL'}$ is the inverse of $\xi$ under convolution (i.e., $\xi^{-1}$ is equipped with canonical isomorphisms $\xi^{-1}\star\xi\cong\d_{\cL}$ and $\xi\star\xi^{-1}\cong \d_{\cL'}$ satisfying the usual axioms). We claim that the functor ${}^{\b}(-)$ is independent of the choice of $\xi$ up to a canonical isomorphism. Indeed, if $\xi'\in {}_{\cL'}\fP^{\b}_{\cL}$ is another minimal IC sheaf, then we may canonically write $\xi'=\xi\ot V$ for a one-dimensional $\Fr$-module $V=\Hom(\xi,\xi')$. Then $\xi'^{-1}=\xi^{-1}\ot V^{\vee}$, and $\xi'\star \cF\star\xi'^{-1}\cong \xi\star \cF\star\xi^{-1}\ot (V\ot V^{\vee})\cong \xi\star \cF\star\xi^{-1}$ canonically. 

If $\cL,\cL',\cL''\in\fo$, $\g\in{}_{\cL''}\un W_{\cL'}$ and $\b\in{}_{\cL'}\un W_{\cL}$, then there is a canonical isomorphism making the following diagram commutative
\begin{equation*}
\xymatrix{{}_{\cL}\cD^{\c}_{\cL}\ar@/_{2pc}/[rr]^{{}^{\g\b}(-)} \ar[r]^{{}^{\b}(-)} & {}_{\cL'}\cD^{\c}_{\cL'} \ar[r]^{{}^{\g}(-)} & {}_{\cL''}\cD^{\c}_{\cL''}}
\end{equation*}
Moreover, these isomorphisms are compatible with three step compositions. All these statements can be checked easily using the independence of $\xi$ in defining the functor ${}^{\b}(-)$.

By Corollary \ref{c:conj wb}, we have an isomorphism of Coxeter groups $\WL\to W^{\c}_{\cL'}$ given by $\Ad(w^{\b})$. It induces an equivalence 
\begin{equation*}
{}^{\b}(-): \SB_{m}(\WL)\isom \SB_{m}(W^{\c}_{\cL'}).
\end{equation*}

\begin{lemma}\label{l:bSB} There is a canonical isomorphism making the following diagram commutative 
\begin{equation*}
\xymatrix{ {}_{\cL}\cD^{\c}_{\cL}\ar[r]^-{\ph_{\cL}}\ar[d]^{{}^{\b}(-)} & K^{b}(\SB_{m}(\WL))/K^{b}(\SB_{m}(\WL))_{0}\ar[d] ^{K^{b}({}^{\b}(-))}     \\
{}_{\cL'}\cD^{\c}_{\cL'}\ar[r]^-{\ph_{\cL'}} & K^{b}(\SB_{m}(W^{\c}_{\cL'}))/K^{b}(\SB_{m}(W^{\c}_{\cL'}))_{0}}
\end{equation*}
Moreover, these isomorphisms are compatible for composable blocks $\b,\g$.
\end{lemma}
\begin{proof}
Unwinding the definitions of the functors involved, it suffices to give a canonical isomorphism ${}^{\b}{\Th^{\c}_{\cL}}\cong \Th^{\c}_{\cL'}$. Now $\Th^{\c}_{\cL}\j{N_{\cL}}$ is a maximal IC sheaf equipped with a nonzero map $\e_{\cL}: \Th^{\c}_{\cL}\to \d_{\cL}$. Therefore ${}^{\b}{\Th^{\c}_{\cL}}\j{N_{\cL}}={}^{\b}{\Th^{\c}_{\cL}}\j{N_{\cL'}}$ is a maximal IC sheaf equipped with a nonzero map ${}^{\b}\e_{\cL}: {}^{\b}\Th^{\c}_{\cL}\to {}^{\b}\d_{\cL}=\d_{\cL'}$, i.e., $({}^{\b}{\Th^{\c}_{\cL}},{}^{\b}\e_{\cL})$ is a rigidified maximal IC sheaf in ${}_{\cL'}\cD^{\c}_{\cL'}$. Therefore, by the discussion in \S\ref{ss:rig neutral}, there is a unique isomorphism $({}^{\b}{\Th^{\c}_{\cL}},{}^{\b}\e_{\cL})\cong (\Th^{\c}_{\cL'},\e_{\cL'})$. 
\end{proof}

\subsection{}\label{ss:Xi action on E} For $\cL,\cL'\in \fo, \b\in{}_{\cL'}\un W_{\cL}$, the isomorphism $\s(\dw^{\b}): H_{\cL}^{\c}\isom H^{\c}_{\cL'}$ (see \S\ref{ss:rel pin}) induces an equivalence of neutral blocks
\begin{equation*}
{}^{\b}(-): {}_{\cL}\cE_{\cL}^{\c}\isom {}_{\cL'}\cE_{\cL'}^{\c}.
\end{equation*}
From the definition of $\IC(w^{\b})^{H}_{\cL}$ we get canonically
\begin{equation*}
{}^{\b}\cF\cong\IC(w^{\b})^{H}_{\cL}\star \cF\star \IC(w^{\b,-1})^{H}_{\cL'}, \quad \forall \cF\in {}_{\cL}\cE_{\cL}^{\c}.
\end{equation*}
From this we see that the functor ${}^{\b}(-)$ is independent of the choice of the lifting $\dw^{\b}$ up to a canonical isomorphism.

Lemma \ref{l:bSB} immediately implies the following.
\begin{cor}\label{c:bPsi}
Let $\cL,\cL'\in\fo, \b\in{}_{\cL'}\un W_{\cL}$. There is a canonical isomorphism making the following diagram commutative
\begin{equation*}
\xymatrix{ {}_{\cL}\cE^{\c}_{\cL}\ar[r]^-{\Psi^{\c}_{\cL}}\ar[d]^{{}^{\b}(-)} & {}_{\cL}\cD^{\c}_{\cL}\ar[d] ^{{}^{\b}(-)}     \\
{}_{\cL'}\cE^{\c}_{\cL'}\ar[r]^-{\Psi^{\c}_{\cL'}} & {}_{\cL'}\cD^{\c}_{\cL'}}
\end{equation*}
Moreover, these isomorphisms are compatible with compositions of  $1$-morphisms in $\fE$ and $\fD$ for composable blocks $\b,\g$.
\end{cor}

\subsection{Semi-direct product}
Let $\G$ be a small groupoid. Let $\fC^{\c}=\{\fC^{\c}_{x}\}_{x\in \Ob(\G)}$ be a collection of monoidal categories indexed by $x\in \Ob(\G)$. An {\em action} of $\G$ on $\fC^{\c}$ is the following data and conditions:
\begin{enumerate}
\item For every morphism $\xi\in {}_{y}\G_{x}$ in $\G$ we are given a monoidal equivalence
\begin{equation*}
{}^{\xi}(-): \fC^{\c}_{x}\isom \fC^{\c}_{y}
\end{equation*}
such that $^{\id_{x}}(-)=\id_{\fC^{\c}_{x}}$.
\item For composable morphisms $\xi\in {}_{y}\G_{x}$ and $\y\in {}_{z}\G_{y}$, a natural isomorphisms between functors
\begin{equation*}
c_{\y,\xi}: {}^{\y}({}^{\xi}(-))\isom {}^{\y\xi}(-)\in \Fun(\fC^{c}_{x}, \fC^{\c}_{z}).
\end{equation*}
\item The data $\{c_{\y,\xi}\}$ satisfy the unital and associativity conditions which we do not spell out.
\end{enumerate}

Given an action of $\G$ on $\fC^{\c}$ as above, we can form the semi-direct product $\fC=\fC^{\c}\rtimes\G$ as a $2$-category over $\G$ as follows.  
\begin{enumerate}
\item $\Ob(\fC)=\Ob(\G)$.
\item For $x,y\in\Ob(\G)$, define  ${}_{y}\fC_{x}$ to be the disjoint union of categories ${}_{y}\fC^{\xi}_{x}$ over $\xi\in {}_{y}\G_{x}$, and each ${}_{y}\fC^{\xi}_{x}$ is a copy of $\fC^{\c}_{y}$. For $\cF\in \fC^{\c}_{y}$, we denote the corresponding object in ${}_{y}\fC^{\xi}_{x}$ to be $(\cF,\xi)$.
\item For $x,y,z\in \Ob(\G)$, the composition of $1$-morphisms in $\fC$ 
\begin{equation*}
\c: {}_{z}\fC_{y}\times{}_{y}\fC_{x}\to {}_{z}\fC_{x}
\end{equation*}
is defined as follows. For $(\cF,\xi)\in \Ob({}_{y}\fC_{x})$ and $(\cG,\y)\in \Ob({}_{z}\fC_{y})$, define their composition to be the object
\begin{equation*}
(\cG,\y)\c(\cF,\xi):=(\cG\c {}^{\y}\cF, \y\c\xi)\in \Ob({}_{z}\fC_{x}).
\end{equation*}
\item The associativity isomorphism for three composable $1$-morphism categories  comes from the data $\{c_{\y,\xi}\}$ in the definition of the $\G$-action in an obvious way.
\end{enumerate}

The following lemma is easily checked from the definitions.

\begin{lemma}\label{l:2functor sd}
Let $\fC$ be a $2$-category over a groupoid $\G$. Let $i: \G\to \fC$ be a $2$-functor over $\G$ (where we view $\G$ as a $2$-category where all the $2$-morphisms are identity maps). Let $\fC^{\c}$ be the collection of categories $\cC^{\c}_{x}:={}_{x}\fC^{\id_{x}}_{x}$ for $x\in \Ob(\G)$. 
\begin{enumerate}
\item Every morphism $\xi\in {}_{y}\G_{x}$ defines a monoidal equivalence $\cC^{\c}_{x}\isom\cC^{\c}_{y}$  by $\cF\mapsto {}^{\xi}\cF:=i(\xi)\c \cF\c i(\xi^{-1})$. The functors ${}^{\xi}(-)$ define an action of $\G$ on $\fC^{\c}$.
\item There is a canonical $2$-functor over $\G$
\begin{equation*}
\fC^{\c}\rtimes \G\to \fC
\end{equation*}
that is the identity $\cC^{\c}_{x}\isom {}_{x}\fC^{\id_{x}}_{x}$ for $x\in \Ob(\G)$ and is $i$ on $\G$. 
\end{enumerate}
\end{lemma}

\subsection{Proof of Theorem \ref{th:all blocks}}
The assignment ${}_{\cL'}\un W_{\cL}\ni w^{\b}\mapsto \IC(w^{\b})^{\da}_{\cL}\in {}_{\cL'}\cD_{\cL}^{\b}$ extends to a $2$-functor $\Xi\to \fD$ over $\Xi$. This is the content of Lemma \ref{l:lgb}. By Lemma \ref{l:2functor sd}(1), this gives an action of $\Xi$ on $\fD^{\c}=\{{}_{\cL}\cD_{\cL}^{\c}\}$. The action is the same as the one constructed in \S\ref{ss:action min}. Therefore Lemma  \ref{l:2functor sd}(2) gives a $2$-functor
\begin{equation}\label{2eq D}
\fD^{\c}\rtimes\Xi\to \fD.
\end{equation}
Concretely, for a morphism $\b\in {}_{\cL'}\un W_{\cL}$ in $\Xi$, the above functor sends $(\cF,\b)\in {}_{\cL'}(\fD^{\c}\rtimes\Xi)^{\b}_{\cL}$ to $\cF\star\IC(w^{\b})^{\da}_{\cL}\in {}_{\cL'}\cD^{\b}_{\cL}$, where $\cF\in {}_{\cL'}\cD^{\c}_{\cL'}$. From this we see that \eqref{2eq D} is an equivalence of $2$-categories.

Similarly,  the assignment ${}_{\cL'}\un W_{\cL}\ni w^{\b}\mapsto \IC(w^{\b})^{H}_{\cL}\in {}_{\cL'}\cE_{\cL}^{\b}$ extends to a $2$-functor $\Xi\to \fE$ over $\Xi$. It gives an action of $\Xi$ on $\fE^{\c}=\{{}_{\cL}\cE_{\cL}^{\c}\}$ that is the same as the one constructed in \S\ref{ss:Xi action on E}. By Lemma \ref{l:2functor sd}(2), this gives a $2$-functor
\begin{equation}\label{2eq E}
\fE^{\c}\rtimes\Xi\to \fE,
\end{equation}
which is also clearly an equivalence.

Finally, Theorem \ref{th:main} gives a monoidal equivalence of the  $\fo=\Ob(\Xi)$-collection of categories $\fE^{\c}\cong \fD^{\c}$. Corollary \ref{c:bPsi} says that this equivalence is compatible with the $\Xi$-actions. Therefore we get a $2$-equivalence
\begin{equation*}
\fE^{\c}\rtimes \Xi\cong \fD^{\c}\rtimes\Xi.
\end{equation*}
Combined with the equivalences \eqref{2eq D} and \eqref{2eq E}, we obtained the equivalence $\Psi$. The desired property (1) of $\Psi$ follows from the construction. For property (2), let us check it for IC sheaves and the other statements are checked similarly. For $w\in {}_{\cL'}W_{\cL}$ written as $w=xw^{\b}$ where $w^{\b}$ is the minimal element in the block of $w$ and $x\in W^{\c}_{\cL'}$, then  by construction, 
\begin{equation*}
\Psi(\IC(w)^{H}_{\cL})=\Psi(\IC(x)^{H}_{\cL'}\star \IC(w^{\b})^{H}_{\cL})=\Psi^{\c}_{\cL'}(\IC(x)^{H}_{\cL'})\star \IC(w^{\b})^{\da}_{\cL}\cong \IC(x)^{\da}_{\cL'}\star \IC(w^{\b})^{\da}_{\cL}=\IC(w)^{\da}_{\cL},
\end{equation*}
where the isomorphism in the middle is \eqref{Phi IC}. This completes the proof.
\qed

\section{Application to character sheaves}
In this section we apply Theorem \ref{th:main} to get an equivalence between the asymptotic versions of character sheaves on $G$ with semisimple parameter $\fo$ and unipotent character sheaves on its endoscopic group. To state the theorem, we review three versions of the statement ``character sheaves are categorical center of Hecke categories'' (after passing to asymptotic versions).

In this section, {\em all schemes are defined over $k=\ov\FF_{q}$}.
  
\subsection{Truncated convolution for the usual Hecke category} Let $H$ be a connected reductive group over $k$ with maximal torus $T$ and Borel subgroup $B_{H}$ containing $T$ (later $H$ will be an endoscopic group of $G$).  Let $\bc$ be a two-sided cell in the Weyl group $W_{H}$. Let $\un\cS^{\bc}_{H}$ be the full subcategory of $\un\cD_{H}$ consisting of perverse sheaves that are direct sums of $\IC(w)_{H}$  for $w\in \bc$. Then $\un\cS^{\bc}_{H}$ is a semisimple abelian category equipped with a truncated convolution $(-)\c(-)$ defined in \cite[3.2]{L-center-unip}. Note that the truncated convolution in \cite{L-center-unip} is first defined for the mixed version of $\un\cS^{\bc}_{H}$ via a perverse degree truncation and a weight truncation; the weight truncation is in fact unnecessary because convolution preserves complexes pure of weight zero. Therefore one can directly define truncated convolution on $\un\cS^{\bc}_{H}$.

\subsection{Unipotent character sheaves}\label{ss:unip chsh} We recall the relationship between the usual Hecke category $\cD_{H}$ for a connected reductive group $H$ and unipotent character sheaves on $H$, following \cite{L-center-unip}.

Character sheaves on $H$ are certain simple perverse sheaves on $H$ equivariant under the conjugation action by $H$. Each character sheaf has a semisimple parameter that is a $W_{H}$-orbit of $\Ch(T)$. When the semisimple parameter is the trivial local system on $T$, we call the character sheaf unipotent. Each unipotent character sheaf on $H$ can be assigned a $2$-sided cell in $W_{H}$, see \cite[1.5]{L-center-unip}. Let $\un\CS^{\bc}_{u}(H)$ be the full subcategory of $D^{b}_{H}(H)$ (for the conjugation action) consisting of finite direct sums of unipotent character sheaves belonging to $\bc$. Then $\un\CS^{\bc}_{u}(H)$ is a semisimple $\Qlbar$-linear abelian category. By \cite[4.6, 9.1]{L-center-unip}, truncated convolution is defined on $\un\CS^{\bc}_{u}(H)$ and makes it into a braided monoidal category.

\begin{theorem}[{\cite[Theorem 9.5]{L-center-unip}}]\label{th:center unip} There is a canonical equivalence of braided monoidal categories
\begin{equation*}
\un\CS^{\bc}_{u}(H)\isom \cZ(\un\cS^{\bc}_{H})
\end{equation*} 
where $\cZ(-)$ denotes the categorical center introduced by Joyal and Street \cite{JS}, Majid \cite{Majid} and Drinfeld.
\end{theorem}

\subsection{Truncated convolution for monodromic Hecke categories}\label{ss:cell}
Now consider the situation for $G$. Let $\fo\subset \Ch(T)$ be a $W$-orbit. In \cite[1.11, Case (v)]{L-center-nonunip}, the notion of two-sided cells inside $W\times\fo$ are defined (see also \cite[third paragraph in p.620]{L-center-nonunip}). Such a two-sided cell $\frc\subset W\times\fo$ can be characterized as follows. For $\cL,\cL'\in \fo$ and any block $\b\in{}_{\cL'}\un W_{\cL}$, let $\frc(\b)=\frc\cap(\b\times\{\cL\})\subset {}_{\cL'}W_{\cL}\times\{\cL\}\subset W\times\fo$. Then $\{\frc(\b)\}$ satisfies
\begin{enumerate}
\item For any triple $\cL,\cL',\cL''\in\fo$ and $\b\in {}_{\cL'}\un W_{\cL}$ and $\g\in {}_{\cL''}\un W_{\cL'}$, we have $w^{\g}\frc(\b)=\frc(\g\b)=\frc(\g)w^{\b}$.
\item For $\cL\in\fo$ and $\b$ the neutral block $\b=W_{\cL}^{\c}$, $\frc(\b)$ is the union of a $\Om_{\cL}=W_{\cL}/W^{\c}_{\cL}$-orbit of the usual two-sided cells for $W_{\cL}^{\c}$.
\end{enumerate}
In other words, starting from a fixed $\cL\in\fo$ and a two-sided cell  $\bc\subset W_{\cL}^{\c}$, there is a unique two-sided cell $\frc\subset W\times\fo$, which we denote by $\frc=[\bc]$, such that $\frc\cap (W_{\cL}^{\c}\times\{\cL\})=\cup_{\om\in\Om_{\cL}}\om(\bc)$.

Fix a two-sided cell $\frc$ for $W\times\fo$. For $\b\in {}_{\cL'}\un W_{\cL}$, let ${}_{\cL'}\un\cS^{\frc(\b)}_{\cL}$ be the full subcategory of ${}_{\cL'}\un\cD^{\b}_{\cL}$ consisting of finite direct sums of simple perverse sheaves of the form $\uIC(w)_{\cL}$ for $w\in \frc(\b)$. Let ${}_{\cL'}\un\cS^{\frc}_{\cL}=\oplus_{\b\in {}_{\cL'}\un W_{\cL}}({}_{\cL'}\un\cS^{\frc(\b)}_{\cL})$ and $\un\cS_{\fo}^{\frc}=\oplus_{\cL,\cL'\in\fo}({}_{\cL'}\un\cS^{\frc}_{\cL})$. Truncated convolution \cite[2.24, 4.6]{L-center-nonunip} defines a monoidal structure $(-)\c(-)$ on $\un\cS_{\fo}^{\frc}$. \footnote{In \cite{L-center-nonunip}, the convolution on the monodromic Hecke category is defined in a different way from \S\ref{ss:conv}. Namely, in {\em loc.cit.} the push-forward along  $G\twtimes{U}G\to G$ was used instead of $G\twtimes{B}G\to G$. As a result, convolution as defined in {\em loc.cit.}  does not preserve purity while the  convolution in this paper does.  Therefore, instead of using the definition of the truncated convolution in \cite[2.24, 4.6]{L-center-nonunip}, we may work with the convolution defined in this paper and ignore weight truncation (doing only the cell truncation). In particular, truncated convolution can be defined directly for the non-mixed category $\un\cS_{\fo}^{\frc}$.}

\subsection{Character sheaves with general monodromy}\label{ss:nonunip chsh}
Let $\un\CS_{\fo}(G)$ be the semisimple abelian category of finite direct sums of character sheaves whose semisimple parameter is $\fo$ (see \cite[middle of p.698]{L-center-nonunip}). To each character sheaf $\cA$  with semisimple parameter $\fo$, one can attach a two-sided cell $\frc_{\cA}$ for $W\times\fo$ following \cite[first paragraph of p.699]{L-center-nonunip}. Let $\un\CS^{\frc}_{\fo}(G)$ be the full subcategory of $\un\CS_{\fo}(G)$ consisting of finite direct sums of character sheaves $\cA$ such that $\frc_{\cA}=\frc$. By \cite[5.20, 6.11]{L-center-nonunip}, truncated convolution equips $\un\CS^{\frc}_{\fo}(G)$ with the structure of a braided monoidal category.

\begin{theorem}[{\cite[Theorem 6.13]{L-center-nonunip}}]\label{th:center nonunip}  There is a canonical equivalence of braided monoidal categories
\begin{equation*}
\un\CS^{\frc}_{\fo}(G)\isom \cZ(\un\cS^{\frc}_{\fo}).
\end{equation*} 
\end{theorem}



\subsection{Unipotent character sheaves on a disconnected group as a twisted center}\label{ss:unip chsh disconn} Let $H$ be a reductive group with a finite-order automorphism $\s$. Then there is the notion of $\s$-twisted unipotent character sheaves: these are certain simple perverse sheaves on $H$ equivariant under the $\s$-twisted conjugation action $h\cdot x=hx\s(h)^{-1}$, $h,x\in H$. Let $\bc$ be a two-sided cell of $W_{H}$ invariant under $\s$. Then one can define the category $\un\CS^{\bc}_{u}(H;\s)$ consisting of finite direct sums of $\s$-twisted unipotent character sheaves on $H$ whose two-sided cell is $\bc$.  If $\s$ changes to the automorphism $\s\Ad(h)$ for some $h\in H(k)$, then right translation by $h$ induces an equivalence between $\un\CS^{\bc}_{u}(H;\s)$ and $\un\CS^{\bc}_{u}(H;\s\Ad(h))$. 

On the other hand, if $\s$ stabilizes $B_{H}$, then it induces an auto-equivalence $\s_{*}$ of the monoidal category $\un\cD_{H}$. For a two-sided cell $\bc$ for $W_{H}$ fixed by $\s$, $\un\cS_{H}^{\bc}$ is stable under the $\s$-action, and one can talk about the {\em $\s$-twisted center} of the monoidal category $\un\cS^{\bc}_{H}$, denoted by $\cZ(\un\cS^{\bc}_{H};\s)$. Objects $\cF$ in $\cZ(\un\cS^{\bc}_{H};\s)$ are $\cF\in \un\cS^{\bc}_{H}$ equipped with functorial isomorphisms $\cF\c\s_{*}\cG\cong\cG\c\cF$ for $\cG\in \un\cS^{\bc}_{H}$. If $\s$ changes to $\s\Ad(b)$ for some $b\in B_{H}(k)$, then the actions of $\s$ and $\s\Ad(b)$ on $\un\cD_{H}$ are canonically equivalent (using the $\Ad(B_{H})$-equivariant structures of objects in $\cD_{H}$), hence a canonical equivalence $\cZ(\un\cS^{\bc}_{H};\s)\cong \cZ(\un\cS^{\bc}_{H};\s\Ad(b))$.

\begin{theorem}[{\cite[Theorem 7.3]{L-nonunip-rep}}]\label{th:tw center} Under the above assumptions (in particular $\bc$ is fixed by $\s$), there is a canonical equivalence of categories
\begin{equation*}
\un\CS^{\bc}_{u}(H;\s)\isom \cZ(\un\cS^{\bc}_{H};\s).
\end{equation*} 
\end{theorem}

\subsection{}\label{ss:prep cs} Now we setup notation for our application to character sheaves.  Fix $\cL\in\fo$, and let $\bc$ be a two-sided cell of $W_{\cL}^{\c}$. Let $[\bc]$ be the two-sided cell for $W\times\fo$ constructed from $\bc$ by the procedure described in \S\ref{ss:nonunip chsh}.  Let $\Om_{\bc}\subset \Om_{\cL}$ be the stabilizer of $\bc$  under $\Om_{\cL}$.

Let $H$ be the endoscopic group of $G$ attached to $\cL$. In \S\ref{ss:fH}, we have introduced an algebraic group  ${}_{\cL}\fH_{\cL}$ containing $H=H_{\cL}^{\c}$ as its neutral component. The component group of ${}_{\cL}\fH_{\cL}$ is $\Om_{\cL}$. For $\b\in \Om_{\cL}$ any lifting $\dw^{\b}\in {}_{\cL}\Xi^{\b}_{\cL}=w^{\b}T$  induces an automorphism of $H$ preserving $B_{H}$. The $\b$-twisted character sheaves category $\un\CS^{\bc}_{u}(H;\dw^{\b})$ is independent of the choice of $\dw^{\b}$ up to canonical equivalences as we discussed in \S\ref{ss:unip chsh disconn}. Therefore we may unambiguously identify all these categories and write it as $\un\CS^{\bc}_{u}(H;\b)$. Note that $\un\CS^{\bc}_{u}(H;\b)$ carries an action of $\Om_{\bc}$: for each $\b'\in \Om_{\bc}$ with lifting $\dw^{\b'}\in {}_{\cL}\wt\Xi^{\b'}_{\cL}$,  the $\dw^{\b'}$-action on $H$ induces an auto-equivalence of  $\un\CS^{\bc}_{u}(H;\b)$, which depends only on $\b'$ up to canonical equivalences.  We let $\un\CS^{\bc}_{u}(H;\b)^{\Om_{\bc}}$ denote the category of objects in $\un\CS^{\bc}_{u}(H;\b)$ together with $\Om_{\bc}$-equivariant structures.

\begin{theorem}\label{th:ch} Let $\fo\subset \Ch(T)$ be the $W$-orbit of $\cL$,  and $\bc$ a two-sided cell in $W_{\cL}^{\c}$. There is a canonical equivalence of semisimple abelian categories
\begin{equation}\label{CSGCSH}
\un\CS^{[\bc]}_{\fo}(G)\isom \bigoplus_{\b\in \Om_{\bc}}\un\CS^{\bc}_{u}(H;\b)^{\Om_{\bc}}.
\end{equation} 
\end{theorem}

\begin{remark}
As in Remark \ref{r:dep on pin}, the equivalence \eqref{CSGCSH} depends on the pinning on $G$ and the additive character $\psi_{0}$ on $\FF_{q}$.

The equivalence \eqref{CSGCSH} induces a canonical bijection between simple objects on both sides. Simple objects in  $\un\CS^{\bc}_{u}(H;\b)$ are classified in \cite[\S46]{L-CSDGX}, from which one can get a classification of simple objects in $\un\CS^{[\bc]}_{\fo}(G)$ using \eqref{CSGCSH}.  In the case $\Om_{\bc}$ is trivial, simple objects in both $\un\CS^{\bc}_{u}(H)$ and $\un\CS^{[\bc]}_{\fo}(G)$ are parametrized by the set $\cM(\cG_{\bc})$ by \cite[Theorem 23.1]{L-CS5} (see \cite[\S4.4-4.13]{L-book} for $\cG_{\bc}$ and $\cM(\cG_{\bc})$). This is consistent with \eqref{CSGCSH}.
\end{remark}

The rest of the section is devoted to the proof of Theorem \ref{th:ch}.

\begin{lemma}\label{l:rest L}
The projection from $\un\cS^{[\bc]}_{\fo}$ to ${}_{\cL}\un\cS^{[\bc]}_{\cL}$ induces an equivalence on  their categorical centers
\begin{equation}\label{rest L}
r_{\cL}: \cZ(\un\cS^{[\bc]}_{\fo})\isom \cZ({}_{\cL}\un\cS^{[\bc]}_{\cL}).
\end{equation}
\end{lemma}
\begin{proof}
We construct an inverse to $r_{\cL}$ as follows. Let ${}_{\cL}\cF_{\cL}\in \cZ({}_{\cL}\un\cS^{[\bc]}_{\cL})$. Define $\cF=\op{}_{\cL'}\cF_{\cL}\in \un\cS^{[\bc]}_{\fo}=\oplus_{\cL,\cL'}({}_{\cL'}\un\cS^{[\bc]}_{\cL})$ by ${}_{\cL'}\cF_{\cL}=0$ if $\cL'\ne\cL$, and ${}_{\cL'}\cF_{\cL'}=\xi\star {}_{\cL}\cF_{\cL}\star\xi^{-1}$ for some $\xi\in {}_{\cL'}\fP_{\cL}$. Using the central structure of ${}_{\cL}\cF_{\cL}$ we see that ${}_{\cL'}\cF_{\cL'}$ is independent of the choice of $\xi$ up to canonical isomorphisms. Moreover, we show that $\cF$ carries a central structure. For $\cG\in {}_{\cL''}\un\cS^{[\bc]}_{\cL'}$, upon choosing  $\xi\in {}_{\cL'}\fP_{\cL}$ and $\y\in {}_{\cL''}\fP_{\cL}$, we may write $\cG=\y\star\cH\star\xi^{-1}$ for $\cH\in{}_{\cL}\un\cS^{[\bc]}_{\cL}$. Then we have an isomorphism $\cF\c\cG={}_{\cL''}\cF_{\cL''}\c\cG=(\y\star{}_{\cL}\cF_{\cL}\star\y^{-1})\c(\y\star\cH\star\xi^{-1})\cong \y\star({}_{\cL}\cF_{\cL}\c\cH)\star\xi^{-1}\cong \y\star(\cH\c{}_{\cL}\cF_{\cL})\star\xi^{-1}=\cG\c{}_{\cL'}\cF_{\cL'}=\cG\c\cF$ coming from the central structure of ${}_{\cL}\cF_{\cL}$. Again  this isomorphism is independent of the choices of $\xi$ and $\y$. The construction ${}_{\cL}\cF_{\cL}\mapsto \cF$  gives an inverse to $r_{\cL}$ and shows that $r_{\cL}$ is an equivalence.
\end{proof}

\begin{lemma} There is a canoncial equivalence
\begin{equation}\label{ZZ}
\cZ({}_{\cL}\un\cS^{[\bc]}_{\cL})\isom \bigoplus_{\b\in\Om_{\bc}}\cZ({}_{\cL}\un\cS^{\c,\bc}_{\cL}; \b)^{\Om_{\bc}}
\end{equation}
where ${}_{\cL}\un\cS^{\c,\bc}_{\cL}$ consists of direct sums of $\uIC(w)_{\cL}$ for $w\in \bc\subset \WL$.
\end{lemma}
\begin{proof} For each $\b\in \Om_{\cL}$, let $\xi_{\b}=\om\IC(w^{\b})^{\da}_{\cL}\in {}_{\cL}\un\fP^{\b}_{\cL}$. For $\cF\in\cZ({}_{\cL}\un\cS^{[\bc]}_{\cL})$,  write $\cF=\oplus_{\b\in\Om_{\cL}}\cF_{\b}\star\xi_{\b}$ where $\cF_{\b}\in {}_{\cL}\un\cS^{\c,[\bc]}_{\cL}:={}_{\cL}\un\cS^{\c}_{\cL}\cap {}_{\cL}\un\cS^{[\bc]}_{\cL}$. By the description of $[\bc]$ in \S\ref{ss:nonunip chsh}, $\cF_{\b}$ can be written uniquely as a direct sum $\oplus_{\bc'\in \Om_{\cL}\cdot\bc}\cF^{\bc'}_{\b}$, where $\cF^{\bc'}_{\b}\in {}_{\cL}\un\cS^{\c,\bc'}_{\cL}$. Let $(-)\c(-)$ denote the truncated convolution in ${}_{\cL}\un\cS^{[\bc]}_{\cL}$.  The central structure of $\cF$ gives the following isomorphisms
\begin{equation}\label{FGbg}
(\cF_{\b}\star\xi_{\b})\c(\cG\star \xi_{\g})\cong (\cG\star \xi_{\g})\c(\cF_{\b}\star\xi_{\b}), \quad \forall \b,\g\in \Om_{\cL}, \cG\in {}_{\cL}\un\cS^{\c,[\bc]}_{\cL}.
\end{equation}
Using the action of $\Om_{\cL}$ on ${}_{\cL}\un\cD^{\c}_{\cL}$ introduced in \S\ref{ss:action min}, we may rewrite the above isomorphism as
\begin{equation*}
(\cF_{\b}\c{}^{\b}\cG)\star\xi_{\b}\star\xi_{\g}\cong (\cG\c{}^{\g}\cF_{\b})\star\xi_{\g}\star\xi_{\b}.
\end{equation*}
By Lemma \ref{l:lgb}, there is a canonical isomorphism $\xi_{\b}\star\xi_{\g}\cong \xi_{\g}\star\xi_{\b}$. We may rewrite the above isomorphism as
\begin{equation}\label{FGbg2}
\cF_{\b}\c{}^{\b}\cG\cong \cG\c{}^{\g}\cF_{\b},\quad\forall \b,\g\in\Om_{\cL}, \cG\in {}_{\cL}\un\cS^{\c,[\bc]}_{\cL}.
\end{equation}
Taking $\g=1$ we get isomorphisms
\begin{equation}\label{FG}
\y_{\cG}: \cF_{\b}\c{}^{\b}\cG\cong \cG\c\cF_{\b}, \quad \forall \b\in\Om_{\cL}, \cG\in {}_{\cL}\un\cS^{\c,[\bc]}_{\cL}
\end{equation}
which equip $\cF_{\b}$ with a $\b$-twisted central structure, i.e., $\cF_{\b}$ has a natural lift to an object $\cF_{\b}^{\#}\in \cZ({}_{\cL}\un\cS^{\c,[\bc]}_{\cL};\b)$.
 
Taking $\cG=\d_{\cL}$ in \eqref{FGbg2} we get isomorphisms
\begin{equation}\label{FLbg}
\z_{\g}: \cF_{\b}\cong {}^{\g}\cF_{\b}, \quad\forall \g\in \Om_{\cL}
\end{equation}
which equip $\cF_{\b}$ with a $\Om_{\cL}$-equivariant structure. The central structure implies that the isomorphisms \eqref{FGbg2} satisfy compatibilities with convolution of the $\cG\star\xi_{\g}$'s, which are equivalent to the commutative diagram
\begin{equation*}
\xymatrix{   \cF_{\b}\c{}^{\b\g}\cG\ar@{=}[d]\ar[r]^-{\y_{{}^{\g}\cG} \ot\id} & {}^{\g}\cG\c\cF_{\b}\ar[r]^-{\id\c\z_{\g}} & {}^{\g}\cG\c{}^{\g}\cF_{\b}  \ar@{=}[d]    \\
\cF_{\b}\c{}^{\g\b}\cG \ar[r]^-{\z_{\g}\c\id} & {}^{\g}\cF_{\b}\c{}^{\g\b}\cG \ar[r]^-{{}^{\g}\y_{\cG}} & {}^{\g}(\cG\c\cF_{\b})}
\end{equation*}
for all $\g\in\Om_{\cL}$ and $\cG\in {}_{\cL}\un\cS^{\c,[\bc]}_{\cL}$.  The commutativity of these diagrams means exactly that $\cF^{\#}_{\b}$ carries a $\Om_{\cL}$-equivariant structure as an object in $\cZ({}_{\cL}\un\cS^{\c,[\bc]}_{\cL};\b)$, i.e., $\cF^{\#}_{\b}$ further lifts to an object $\cF_{\b}^{\hs}\in \cZ({}_{\cL}\un\cS^{\c,[\bc]}_{\cL};\b)^{\Om_{\cL}}$.

Take a cell $\bc'\subset W_{\cL}^{\c}$ in  the $\Om_{\cL}$-orbit of $\bc$ and take $\cG\in {}_{\cL}\un\cS^{\c,\bc'}_{\cL}$. Now ${}^{\b}\cG\in {}_{\cL}\un\cS_{\cL}^{\c, \b(\bc')}$. For $w,w'\in W_{\cL}^{\c}$ in different cells, the truncated convolution of $\uIC(w)_{\cL}$ and $\uIC(w')_{\cL}$ vanishes. Therefore the left side of \eqref{FG} lies in ${}_{\cL}\un\cS^{\c, \b(\bc')}_{\cL}$ while the right side lies in ${}_{\cL}\un\cS^{\c, \bc'}_{\cL}$. If $\b(\bc')\ne\bc'$, then both sides of \eqref{FG} must vanish, hence $\uIC(w)_{\cL}\c\cF_{\b}^{\bc'}=0$ for all $w\in \bc'$.  This implies $\cF_{\b}^{\bc'}=0$ if $\b(\bc')\ne\bc'$ since ${}_{\cL}\un\cS^{\c,\bc'}_{\cL}$ has a monoidal unit. Since $\Om_{\cL}$ is abelian, $\b\in\Om_{\cL}$ either fixes all $\bc'$ in the orbit of $\bc$ or none, therefore $\cF_{\b}=0$ if $\b\notin \Om_{\bc}$.

Now we consider $\b\in\Om_{\bc}$. The isomorphisms \eqref{FLbg} allows us to recover $\cF_{\b}^{\bc'}$ for any cell $\bc'$ in the $\Om_{\cL}$-orbit of $\bc$ from $\cF_{\b}^{\bc}$. The object $\cF_{\b}^{\bc}$ lifts to $\cF^{\bc,\hs}_{\b}\in \cZ({}_{\cL}\un\cS^{\c,\bc}_{\cL}; \b)^{\Om_{\bc}}$. The functor $\cF_{\b}^{\hs}\mapsto \cF^{\bc,\hs}_{\b}$ is an equivalence
\begin{equation*}
\cZ({}_{\cL}\un\cS^{\c,[\bc]}_{\cL}; \b)^{\Om_{\cL}}\isom \cZ({}_{\cL}\un\cS^{\c,\bc}_{\cL}; \b)^{\Om_{\bc}}.
\end{equation*}
Combining the above discussions, we arrive at the equivalence \eqref{ZZ} given by $\cF\mapsto \oplus_{\b\in\Om_{\bc}}\cF^{\bc,\hs}_{\b}$.
\end{proof}

\subsection{Proof of Theorem \ref{th:ch}}
Theorem \ref{th:main} implies a monoidal equivalence between semisimple abelian categories
\begin{equation}\label{SHL}
\un\cS_{H}^{\bc}\isom {}_{\cL}\un\cS^{\c,\bc}_{\cL}.
\end{equation}
In \cite{L-center-nonunip}, the value of the $a$-function for $[\bc]$ used in the construction of the truncated convolution is the same as the value of the $a$-function on $\bc$ as a cell for $W_{H}$. By Corollary \ref{c:bPsi},  \eqref{SHL} is equivariant under the actions of $\Om_{\bc}$. Therefore, we get a canonical braided monoidal equivalence for $\b\in\Om_{\bc}$
\begin{equation}\label{ZZLH}
\cZ(\un\cS_{H}^{\bc};\b)^{\Om_{\bc}}\isom\cZ({}_{\cL}\un\cS^{\c,\bc}_{\cL}; \b)^{\Om_{\bc}}.
\end{equation}

Composing the known equivalences we get
\begin{equation*}
\xymatrix{
\un\CS^{[\bc]}_{\fo}(G)\ar[r]_-{\sim}^-{\textup{Th.}\ref{th:center nonunip}} & \cZ(\un\cS^{[\bc]}_{\fo})\ar[r]^-{ \eqref{rest L}}_-{\sim} & \cZ({}_{\cL}\un\cS^{[\bc]}_{\cL})\ar[r]^-{\eqref{ZZ}}_-{\sim} & \oplus_{\b\in\Om_{\bc}}\cZ({}_{\cL}\un\cS^{\c,\bc}_{\cL}; \b)^{\Om_{\bc}}\\
&&& \oplus_{\b\in\Om_{\bc}}\cZ(\un\cS_{H}^{\bc};\b)^{\Om_{\bc}}\ar[u]^{\eqref{ZZLH}}_{\wr} & \oplus_{\b\in\Om_{\bc}}\un\CS^{\bc}_{u}(H;\b)^{\Om_{\bc}}\ar[l]_-{\textup{Th.}\ref{th:tw center}}^-{\sim}.}
\end{equation*}

\section{Application to representations}
In \cite{L-unip-rep} and \cite{L-nonunip-rep}, the first-named author has related the category of representations of $G(\FF_{q})$ to the twisted categorical center of asymptotic versions of the monodromic Hecke category of $G$, in a similar way that character sheaves on $G$ are related to the categorical center of the monodromic Hecke category.   In this section, we apply the monodromic-endoscopic equivalence to prove a relationship between representations of $G(\FF_{q})$ and its endoscopic groups. 

In this section, {\em all schemes are over $k=\ov\FF_{q}$}. We will work in the more general context of disconnected groups as in  \cite{L-nonunip-rep}, and establish a relationship between character sheaves on disconnected groups and unipotent character sheaves on their endoscopic groups.

\subsection{Disconnected groups and forms}\label{ss:cs disc} Let $G$ be a connected reductive group over $k$ with a pinning $(T, B, \{\bx_{-\a_{s}}\})$ (all defined over $k$). Let $\e:G\to G$ be a morphism over $k$ preserving  $(T,B)$ and satisfying one of the following two conditions:
\begin{enumerate}
\item[(A)] $\e$ is the Frobenius map for some rational structure of $G$ over $\FF_{q}$ and for each simple root $\a_{s}$ we have $\bx_{-\a_{\e(s)}}(\e(u))=\bx_{-\a_{s}}(u)^{q}$, $u\in U_{-\a_{s}}$.
\item[(B)]  $\e$ is a finite order automorphism of $G$ over $k$ preservering $\{\bx_{-\a_{s}}\}$.
\end{enumerate}
We will refer to the two conditions above as ``Situation (A)'' and ``Situation (B)''.

We form the semi-direct semigroup product $G\rtimes \e^{\ZZ_{\ge0}}$ (where $\e^{\ZZ_{\ge0}}$ is a copy of $\ZZ_{\ge}$ acting on $G$ via $\e$). There is the notion of character sheaves on the coset $G\cdot \e$, see \cite[\S6.1]{L-nonunip-rep}. These are certain $G$-equivariant simple perverse sheaves on $G\cdot \e$.   In situation (B), this notion is the same as the $\e$-twisted character sheaves on $G$ considered in \S\ref{ss:unip chsh disconn}.

When  $\e$ is in situation (A),  character sheaves on $G\cdot \e$ are exactly the irreducible $\Qlbar$-representations of the finite group $G^{\e}$,  the group of $\FF_{q}$-points of the form of $G$ with Frobenius map $\e$.

The map $\e$ induces an action on $\Ch(T)$: $\e\cL:=\e_{*}\cL$. In the following, we fix a  $W$-orbit $\fo\subset \Ch(T)$ that is stable under the action of $\e$. Fix $\cL\in \Ch(T)$. We have the relatively pinned endoscopic group $H=H_{\cL}^{\c}$ as defined in \S\ref{ss:rel pin} (now over $k$). 

The map $\e$ induces an automorphism of the based root system of $G$ which we still denote by $\e$. It restricts to a bijection of based roots systems $\e: \Phi_{\cL}\isom \Phi_{\e\cL}$ (it sends $\Phi^{+}_{\cL}=\Phi_{\cL}\cap \Phi^{+}$ to $\Phi^{+}_{\e\cL}=\Phi_{\e\cL}\cap \Phi^{+}$ because $\e$ is a bijection of based root systems).  For each $\b\in {}_{\cL}\un W_{\e\cL}$, $w^{\b}$ gives a bijection of base root systems $w^{\b}: \Phi_{\e\cL}\to \Phi_{\cL}$. The composition $w^{\b}\c\e$ is an automorphism of the base root system $(\Phi_{\cL},\D_{\cL})$. 

Fix a lifting $\dw^{\b}$ for each $\b\in {}_{\cL}\un W_{\e\cL}$. In situation (A), there is a unique $\FF_{q}$-Frobenius structure $\s_{\b\e}: H\to H$ preserving $(T,B_{H})$, inducing  $w^{\b}\c\e$ on the root system of $H$ (which is identified with $\Phi_{\cL}$), and such that $(\s_{\b\e},\Ad(\dw^{\b})\c\e)$ is compatible with the relative pinning in the sense that the following diagram is commutative for each simple root of $\a$ of $H$
\begin{equation*}
\xymatrix{     H_{\a}  \ar[d]^{\s_{\b\e}}  \ar[r]^-{\io_{\a}}  & G_{\a}\ar[d]^{\Ad(\dw^{\b})\c\e}\\
H_{w^{\b}\e(\a)} \ar[r]^-{\io_{w^{\b}\e(\a)}} & G_{w^{\b}\e(\a)}}
\end{equation*}
Note that the construction of $\s_{\b\e}$ depends on the choice of the lifting $\dw^{\b}$, but the finite group $H^{\s_{\b\e}}$ is independent of the choice up to inner automorphisms. In particular, $\Rep(H^{\s_{\b\e}})$ is independent of the choice of the lifting $\dw^{\b}$. 

Similarly, in situation (B), there is a unique finite order automorphism $\s_{\b\e}: H\to H$ preserving $(T,B_{H})$, inducing  $w^{\b}\c\e$ on the root system of $H$, and such that $(\s_{\b\e},\Ad(\dw^{\b})\c\e)$ is compatible with the relative pinning in the above sense.

\begin{exam} Suppose $\e$ is the Frobenius map for the split $\FF_{q}$-structure of $G$. In this case, $\e\cL=\cL^{1/q}$ (note that the order of $\cL$ is always prime to $p$). Assume that $W_{\cL}=\{1\}$ (we always assume $\fo$ be stable under $\e$). In this case, $H=T$, and there is a unique $w\in W$ such that $w\cL^{1/q}=\cL$ (the blocks are singletons, so $w$ can be viewed as a block).  Then $\s_{w\e}:T\to T$ is the Frobenius map for the $\FF_{q}$-form of $T$ given by the $W$-conjugacy class of $w$ (note that the conjugacy classes of maximal tori of the split $G$ defined over $\FF_{q}$ are classified by conjugacy classes in $W$).
\end{exam}

\subsection{Character sheaves on disconnected groups as a twisted center}  Recall that we fix a $W$-orbit $\fo\subset \Ch(T)$ stable under $\e$, and also fix $\cL\in \fo$. Let $\frc\subset W\times\fo$ be a two-sided cell that is stable under $\e$. As in \S\ref{ss:cell}, we may write $\frc=[\bc]$ for some two-sided cell $\bc\subset \WL$ which is well-defined up to the action of $\Om_{\cL}$.

Let $\un\CS^{\frc}_{\fo}(G;\e)$ be the semisimple abelian category whose objects are finite direct sums of character sheaves on $G\cdot\e$ with semisimple parameter $\fo$ and belonging to the cell $\frc$ (see \cite[\S6.1]{L-nonunip-rep}). In situation (A), the $G$-conjugation action on $G\cdot \e$ is transitive by Lang's theorem, with the stabilizer of $1\cdot \e$ equal to the finite group $G^{\e}$. Therefore we have an equivalence
\begin{equation}\label{CS rep}
\un\CS^{\frc}_{\fo}(G;\e)\cong \Rep^{\frc}_{\fo}(G^{\e}),
\end{equation}
the latter being the semisimple abelian category of $\Qlbar$-representations of the finite group $G^{\e}$ whose semisimple parameter is $\fo$ and are finite direct sums of irreducible representations belonging to the cell $\frc$. 

The following theorem proved in \cite{L-nonunip-rep} is a common generalization of Theorems \ref{th:center unip}, \ref{th:center nonunip} and \ref{th:tw center}.
\begin{theorem}[{\cite[Theorem 7.3]{L-nonunip-rep}}]\label{th:tw nonunip}
Under the above assumptions (in particular $\fo$ and $\frc$ are stable under  $\e$), there is a canonical equivalence of categories
\begin{equation*}
\un\CS^{\frc}_{\fo}(G;\e)\isom \cZ(\un\cS^{\frc}_{\fo};\e).
\end{equation*} 
\end{theorem}


\subsection{} 
We need some more notation to state the next theorem.   Fix a two-sided cell $\bc$ of $\WL$ contained in $\frc\cap\WL$. Then $\frc\cap\WL$ is the union of two-sided cells that are in the same $\Om_{\cL}$-orbit  of $\bc$. 

For $\b\in {}_{\cL}\un W_{\e\cL}$, $w^{\b}\c \e\cL=\cL$, hence $w^{\b}\c \e$ acts on $W_{\cL}, \WL$ and on $\Om_{\cL}$, and permutes the cells in $\WL$ that belong to $\frc$. Let
\begin{eqnarray*}
\fB_{\bc}=\{\b\in {}_{\cL}\un W_{\e\cL}| \mbox{$w^{\b}\c \e$ preserves the cell $\bc$ of $\WL$}\}.
\end{eqnarray*}
The left translation action of $\Om_{\cL}$ on ${}_{\cL}\un W_{\e\cL}$ (using the multiplication of blocks defined in \S\ref{ss:block W}) restricts to an action of $\Om_{\bc}=\Stab_{\Om_{\cL}}(\bc)$ on $\fB_{\bc}$, making it a $\Om_{\bc}$-torsor. Similarly, the right translation action of $\g\in\Om_{\cL}$ on ${}_{\cL}\un W_{\e\cL}$ by $\b\mapsto\b\e(\g)$ makes $\fB_{\bc}$ into a right $\Om_{\bc}$-torsor. Combining the two actions we get a twisted conjugation action of $\Om_{\cL}$ on ${}_{\cL}\un W_{\e\cL}$
\begin{equation*}
\Ad_{\e}(\g)(\b)=\g\b\e(\g)^{-1}, \quad \g\in \Om_{\cL}, \b\in{}_{\cL}\un W_{\e\cL}. 
\end{equation*}
It restricts to an action of $\Om_{\bc}$ on $\fB_{\bc}$ which we still denote by $\Ad_{\e}$. For $\b\in {}_{\cL}\un W_{\e\cL}$, let $\Om_{\b}\subset \Om_{\cL}$ be its stabilizer under the $\Ad_{\e}$-action of $\Om_{\cL}$; let  $\Om_{\bc,\b}=\Om_{\bc}\cap\Om_{\b}$. Since $\Om_{\cL}$ is abelian, the groups $\Om_{\b}, \Om_{\bc}$ and  $\Om_{\bc,\b}$ are independent of the choices of $\bc$ and $\b$.

When $\b\in \fB_{\bc}$, we can define the semisimple abelian category  $\un\CS^{\bc}_{u}(H;\s_{\b\e})$ consisting of finite direct sums of unipotent character sheaves on $H\cdot \s_{\b\e}$ belonging to the cell $\bc$. In situation (A), we have an equivalence 
\begin{equation}\label{CS rep H}
\un\CS^{\bc}_{u}(H;\s_{\b\e})\cong \Rep^{\bc}_{u}(H^{\s_{\b\e}}),
\end{equation}
the latter being the semisimple abelian category of unipotent $\Qlbar$-representations of the finite group $H^{\s_{\b\e}}$ belonging to the cell $\bc$.


\quash{ For $\b\in {}_{\cL}\un W_{\e\cL}$, we introduce a twisted analogue of the cocycle $\L_{\b}$ defined in \eqref{Lb}. For $\g\in \Om_{\b}\subset \Om_{\cL}$, define
\begin{equation*}
\L_{\b\e}(\g)=\Hom(\xi_{\g}, \xi_{\b}\star \e_{*}\xi_{\g}\star \xi_{\b}^{-1}).
\end{equation*}
This is canonically independent of the choices of $\xi_{\g}\in {}_{\cL}\fP^{\g}_{\cL}$ and $\xi_{\b}\in {}_{\cL}\fP^{\g}_{F_{\e}\cL}$, and it defines a normalized $1$-cocycle $\L_{\b\e}$ of lines on $\Om_{\b}$.
}
Finally, using the group ${}_{\cL}\fH_{\cL}$  with neutral component $H$ and group of components equal to $\Om_{\cL}$, there is a canonical action of $\Om_{\b}$ on $\un\CS_{u}(H;\s_{\b\e})$, defined in the same way as discussed in \S\ref{ss:prep cs}. This restricts to an action of $\Om_{\bc,\b}$ on $\un\CS^{\bc}_{u}(H;\s_{\b\e})$. It therefore makes sense to form the category $\un\CS^{\bc}_{u}(H;\s_{\b\e})^{\Om_{\bc,\b}}$ of objects in $\un\CS^{\bc}_{u}(H;\s_{\b\e})$  equipped with $\Om_{\bc,\b}$-equivariant structures.

The following result gives a relationship between character sheaves for a disconnected group with a fixed semisimple parameter and unipotent character sheaves on its endoscopic groups, generalizing Theorem \ref{th:ch}.

\begin{theorem}\label{th:rep} Choose a representative for each $\Ad_{\e}(\Om_{\bc})$-orbit of $\fB_{\bc}$, and denote this set of representatives by $\dot\fB_{\bc}$. There is a canonical equivalence of semisimple abelian categories
\begin{equation*}
\un\CS^{\frc}_{\fo}(G;\e)\cong \bigoplus_{\b\in \dot\fB_{\bc}}\un\CS^{\bc}_{u}(H;\s_{\b\e})^{\Om_{\bc,\b}}.
\end{equation*}
\end{theorem}

Using \eqref{CS rep} and \eqref{CS rep H}, we get:
\begin{cor}\label{c:rep} In situation (A), under the same notations as Theorem \ref{th:rep}, there is a canonical equivalence of semisimple abelian categories
\begin{equation*}
\Rep^{\frc}_{\fo}(G^{\e})\cong \bigoplus_{\b\in \dot\fB_{\bc}}\Rep^{\bc}_{u}(H^{\s_{\b\e}})^{\Om_{\bc,\b}}.
\end{equation*}
\end{cor}

\begin{exam}\label{ex:SLn rep} Consider $G=\SL_{n}$, and $\e$ is the Frobenius map for the split $\FF_{q}$-structure on $G$.  Let $\cK\in \Ch(\Gm)$ be of order $n$. For rational number $a$ whose denominator is prime to $n$, it makes sense to take the tensor power $\cK^{a}$. Let
\begin{equation*}
\cL=\cK\boxtimes \cK^{2}\boxtimes\cdots \boxtimes \cK^{n}\in \Ch(\Gm^{n}).
\end{equation*} 
Restricting $\cL$ to the diagonal torus $T$ of $G$ (identified with the subtorus $T\subset \Gm^{n}$ with product equal to $1$) we denote it still by $\cL\in \Ch(T)$. The Weyl group $S_{n}$ acts on $\Ch(\Gm^{n})$ by $w(\cL_{1}\boxtimes\cdots\boxtimes \cL_{n})=\cL_{w^{-1}(1)}\boxtimes\cdots\boxtimes \cL_{w^{-1}(n)}$, and it restricts to an action on $\Ch(T)$.

Let $\fo$ be the $W$-orbit of $\cL$. In this case $\WL=\{1\}$ but $W_{\cL}=\Om_{\cL}\cong\ZZ/n\ZZ$ can be identified with the group generated by the cyclic permutation $c: i\mapsto i+1$ in $S_{n}$. We have $H=T$.  Since there is only one cell $\bc$ for $T$ with any semisimple parameter, we will omit $\bc$ from the notation.

We have $\e_{*}\cL=\cK^{1/q}\boxtimes\cK^{2/q}\cdots\boxtimes\cK^{n/q}$. Let $w(i)=i/q\mod n$, viewed as an element in $S_{n}$, then  $\fB={}_{\cL}W_{\e\cL}=\{c^{i}w|i\in \ZZ/n\ZZ\}$. For $\b\in\fB$, we have $\Ad_{\e}(c)(\b)=c\b\e(c)^{-1}=c\b c^{-1}$. Direct calculation shows that $\Ad_{\e}(c)$ sends $c^{i}w\in \fB$ to $c^{i+(q-1)/q}w\in \fB$. Let $d=\gcd(n,q-1)$, then the $\Ad_{\e}$-action of $\Om_{\cL}=\ZZ/n\ZZ$ on $\fB$ has $d$ orbits, and the stabilizers are isomorphic to  $\ZZ/d\ZZ$. By Corollary \ref{c:rep},  $\SL_{n}(\FF_{q})$ has $d^{2}$ irreducible representations with semisimple parameter $\fo$.
\end{exam}

\begin{exam} Consider the case $G=\SL_{n}$ but $\e$ is the Frobenius map corresponding to the special unitary group $\SU_{n}$ over $\FF_{q}$. Its action on the diagonal torus is given by $(x_{1},x_{2}, \cdots, x_{n})\mapsto (x_{n}^{-q},x_{n-1}^{-q},\cdots, x_{1}^{-q})$. We consider the same $\cL\in \Ch(T)$ as in Example \ref{ex:SLn rep}. This time $\e\cL=\cK^{-n/q}\boxtimes\cdots \boxtimes \cK^{-1/q}$. Let $w(i)=(i-n-1)/q\mod n$, viewed as an element in $S_{n}$, then  $\fB={}_{\cL}W_{\e\cL}=\{c^{i}w|i\in \ZZ/n\ZZ\}$. For $\b\in\fB$, we have $\Ad_{\e}(c)(\b)=c\b\e(c)^{-1}=c\b c$ because $\e(c)=c^{-1}$. Then $\Ad_{\e}(c)(c^{i}w)=c^{i+(q+1)/q}w$. Let $d'=\gcd(q+1,n)$. Then as in the discussion in Example \ref{ex:SLn rep},  $\SU_{n}(\FF_{q})$ has $d'^{2}$ irreducible representations with semisimple parameter $\fo$.
\end{exam}

\subsection{Sketch of proof of Theorem \ref{th:rep}} 
Applying Theorem \ref{th:tw nonunip} to $\un\CS^{\frc}_{\fo}(G;\e)$ and to $\un\CS^{\bc}_{u}(H;\s_{\b\e})$ separately,  we reduce to showing that
\begin{equation*}
\cZ(\un\cS_{\fo}^{\frc}; \e)\cong \bigoplus_{\b\in\dot\fB_{\bc}}\cZ(\un\cS_{H}^{\bc}; \s_{\b\e})^{\Om_{\bc,\b}}.
\end{equation*}
By the equivalence in Theorem \ref{th:main}, we may replace  $\un\cS_{H}^{\bc}$ by ${}_{\cL}\cS_{\cL}^{\c, \bc}$ on the right side and reduce to showing that
\begin{equation}\label{twisted center decomp}
\cZ(\un\cS_{\fo}^{\frc}; \e)\cong \bigoplus_{\b\in\dot\fB_{\cL, \bc}}\cZ({}_{\cL}\un\cS_{\cL}^{\c,\bc}; \b\c\e_{*})^{\Om_{\bc,\b}}.
\end{equation}
Here the twisting $\b\c\e_{*}$ that appears on the right side refers to the autoequivalence $\cF\mapsto \xi_{\b}\star\e_{*}\cF\star\xi_{\b}^{-1}$ of ${}_{\cL}\un\cS_{\cL}^{\c,\bc}$, where $\xi_{\b}=\om(\IC(w^{\b})^{\da}_{\cL})$.


The argument for \eqref{twisted center decomp} is similar to that of Theorem \ref{th:ch}, so we only sketch the main steps. We have a twisted analogue of \eqref{rest L}: restriction gives an equivalence
\begin{equation}\label{tw rest L}
\cZ(\un\cS_{\fo}^{\frc}; \e)\isom \cZ({}_{\cL}\un\cS^{\frc}_{\e\cL}; \e)
\end{equation}
Here the right side are objects $\cF\in {}_{\cL}\un\cS^{\frc}_{\e\cL}$ together with functorial isomorphisms $\cF\c \e_{*}\cG\isom \cG\c\cF$ for $\cG\in {}_{\cL}\un\cS^{\frc}_{\cL}$ ($\c$ denotes the truncated convolution).  One can then write $\cF\in \cZ({}_{\cL}\un\cS^{\frc}_{\e\cL}; \e)$ as a sum
\begin{equation*}
\cF=\bigoplus_{\bc'\sim \bc, \b\in {}_{\cL}\un W_{\e\cL}} \cF^{\bc'}_{\b}\star \xi_{\b}
\end{equation*}
for $\cF^{\bc'}_{\b}\in {}_{\cL}\un\cS^{\c, \bc'}_{\cL}$ (where $\bc'$ runs over the $\Om_{\cL}$-orbit of $\bc$).  The $\e$-commutation with $\cG$ in various cell subcategories of ${}_{\cL}\un\cS^{\c}_{\cL}$ implies that if $\cF^{\bc}_{\b}\ne0$, then $w^{\b}\c \e$ preserves $\bc$, i.e., $\b\in \fB_{\bc}$. The $\e$-commutation with $\cG\in {}_{\cL}\un\cS^{\c,\bc}_{\cL}$ shows that $\cF^{\bc}_{\b}$ carries a $\b\c\e_{*}$-twisted central structure. The $\e$-commutation with $\cG=\xi_{\g}$ for $\g\in \Om_{\cL}$ then shows that $\cF^{\bc}_{\b}$ determines $\cF^{\g\cdot\bc}_{\g\cdot_{\e}\b}$, and that $\cF^{\bc}_{\b}$ (for $\b\in\fB_{\bc}$) is equipped with a $\Om_{\bc,\b}$-equivariant structure. Here we are using that $\e_{*}\xi_{\g}$ is canonically isomorphic to $\xi_\g$, since $\e$ is compatible with the pinning and hence acts on the Whittaker categories ${}_{\psi}\un\cM_{\cL}$ which were used to define the rigidified minimal IC sheaves. Sending $\cF\in \cZ({}_{\cL}\un\cS^{\frc}_{\e\cL}; \e)$ to $\{\cF^{\bc}_{\b}\}_{\b\in\dot\fB_{\bc}}$ then induces an equivalence 
\begin{equation*}
\cZ({}_{\cL}\un\cS^{\frc}_{\e\cL}; \e)\cong \bigoplus_{\b\in\dot\fB_{\bc}}\cZ({}_{\cL}\un\cS_{\cL}^{\c,\bc}; \b\c \e_{*})^{\Om_{\bc,\b}}.
\end{equation*} 
Combining this with \eqref{tw rest L} we get \eqref{twisted center decomp}, proving the theorem.
\qed

\end{document}